\documentclass[11pt, reqno]{amsart}
\usepackage{amsthm,amsmath,amsfonts,amssymb,color}
\usepackage[bookmarks]{hyperref}
\reversemarginpar

\newtheorem{thm}{Theorem}[section]
\newtheorem{prop}[thm]{Proposition}
\newtheorem{cor}[thm]{Corollary}
\newtheorem{lemma}[thm]{Lemma}
\newtheorem{defn}[thm]{Definition}

\newtheorem{preremark}[thm]{Remark}
\newenvironment{remark}{\begin{preremark}\rm}{\medskip \end{preremark}}

\numberwithin{equation}{section}

\newcommand{\norm}[1]{\left\Vert#1\right\Vert}
\newcommand{\abs}[1]{\left\vert#1\right\vert}
\newcommand{\set}[1]{\left\{#1\right\}}

\newcommand{\R}{\mathbb R}
\newcommand{\F}{\mathbb F}
\DeclareMathOperator{\Vol}{Vol}
\DeclareMathOperator{\Span}{span}

\newcommand{\grad} {\nabla}

\newcommand{\bb} {\mathrm{b}}
\newcommand{\dx} {\; \mathrm{d} x}
\newcommand{\dd} {\mathrm{d}}

\DeclareMathOperator{\supp}{supp}
\DeclareMathOperator{\dv}{div}
\DeclareMathOperator{\Def}{Def}
\DeclareMathOperator{\Ric}{Ric}
\DeclareMathOperator{\N}{N-S}

\def\H{\mathbb H^{2}(-a^{2})}
\def\be{\begin{equation}}
\def\ee{\end{equation}}
\newcommand{\V}{\mathbb V}
\newcommand{\tV}{\textbf V}
\DeclareMathOperator{\nablab}{\overline \nabla}

\newcommand{\ip}[1]{\langle{#1}\rangle}
\newcommand{\NS}[1]{\N_{#1}}

\begin{document}
\title[Weak formulation of Navier-Stokes on the 2D hyperbolic space]{Remarks on the weak formulation of the Navier-Stokes equations on the 2D hyperbolic space.}

\author[Chan]{Chi Hin Chan}
\address{Department of Applied Mathematics, National Chiao Tung University,1001 Ta Hsueh Road, Hsinchu, Taiwan 30010, ROC}
\email{cchan@math.nctu.edu.tw}
\author[Czubak]{Magdalena Czubak}
\address{Department of Mathematical Sciences, Binghamton University (SUNY),
Binghamton, NY 13902-6000, USA}
\email{czubak@math.binghamton.edu}

\begin{abstract}
The Leray-Hopf solutions to the Navier-Stokes equation are known to be unique on $\R^{2}$.   In our previous work we showed the breakdown of uniqueness in a hyperbolic setting.  In this article, we show how to formulate the problem in order so the uniqueness can be restored.
\end{abstract}
\date{\today}
\maketitle

\tableofcontents

 \section{Introduction}
The Navier-Stokes equation on $\R^{n}$ is given by
\be\tag{$\NS{\R^{n}}$}
\begin{split}
\partial_{t}u -\Delta u + u\cdot \nabla u + \nabla p & = 0 , \\
\dv u& = 0 ,
\end{split}
\ee
where $u=(u_{1}, \dots, u_{n})$ is the velocity of the fluid, $p$ is the pressure, and $\dv u = 0$ means the fluid is incompressible.
Let $T>0$.  Existence of global weak solutions
\be\label{LH1}
u\in L^{\infty}(0, T ; L^{2}(\mathbb{R}^{n}) ) \cap L^{2}(0,T ;\dot H^{1}(\mathbb{R}^{n}))
\ee
 for $n=2, 3$ satisfying the \emph{global energy inequality}
\begin{equation}\label{LH2}
\int_{\mathbb{R}^{n}} |u(t,x)|^{2} \dx + 2 \int_{0}^{t} \int_{\mathbb{R}^{n}}|\nabla
u|^{2} \dx \dd s \leq  \int_{\mathbb{R}^{n}} |u_{0}|^{2} \dd x, \quad 0\leq t \leq T.
\end{equation}
has been established in the work of Leray~\cite{Leray} and  Hopf~\cite{Hopf}.   Solutions satisfying \eqref{LH1}-\eqref{LH2} are referred to as Leray-Hopf solutions.
\newline\indent
The smoothness and uniqueness of Leray-Hopf
solutions for $\NS{\R^{2}}$ is well-known.  In \cite{CC10},  contrary to what is known in the Euclidean 2D setting, we were able to show that there is non-uniqueness in 2D for simply connected manifolds with negative sectional constant curvature (see Theorem \ref{mainthmns} and \ref{thm2ns} below for precise statements).  The goal of this article is to show how we can restore the uniqueness.  In the process, we develop the theory of weak solutions to the Navier-Stokes equations on the 2D hyperbolic space.

Before we state the current results precisely, we briefly review the relevant ones from \cite{CC10}, so we can introduce the required terminology and put the new results in some perspective.  

For a more detailed review of other works in the literature please see \cite{CC10} or \cite{DindosMitrea}.  However,
here we would like to mention that the majority of  the work on the Navier-Stokes equation on manifolds has been done either in a setting of a compact manifold or a bounded subdomain \cite{EbinMarsden, Priebe, MitreaTaylor, DindosMitrea, Mazzucato2003, Taylor3, AvezBamberger, IFilatov, Illin, TemamWang}.

The works that we are aware of on non-compact manifolds are of Q.S. Zhang \cite{QSZhang}, Khesin and Misio\l ek \cite{KhesinMisiolek}, and that of Taylor \cite{Taylor}.  In \cite{QSZhang} the author shows the non-uniqueness of the weak solution with finite $L^{2}$ norm on a connected sum of two copies of $\R^{3}$.  In \cite{KhesinMisiolek} the authors primarily study the stationary Euler equation on the hyperbolic space (but also see Sections \ref{otherD} and \ref{simpler}).   In \cite{Taylor} the author obtains pointwise estimates for the gradient of harmonic functions on the hyperbolic space (also see remarks after \eqref{expdecaydF} and Section \ref{simpler}).

\subsection{Overview of previous results}
In general, we would like to investigate how geometry of an underlying space affects the solutions to the Navier-Stokes equation.
Motivated by the Euclidean problem, we were curious about the non-compact Riemannian manifolds.  When we move from the Euclidean setting to the Riemannian manifold, the first question is how to write the equations.  In particular, what is the natural generalization of the Laplacian, $\Delta$?  This question was addressed by Ebin and Marsden in \cite{EbinMarsden}, where they indicated that the ordinary Laplacian should be replaced by the following
operator
\[
L = 2 \Def^{*}\Def  = (\dd\dd^{*}
+ \dd^{*} \dd) + \dd\dd^{*} -2 \Ric ,
\]
where  $\Def$ and $\Def^{*}$ are the stress tensor and its adjoint respectively,
$(\dd\dd^{*} +
\dd^{*}\dd)= -\Delta$
is the Hodge Laplacian with $\dd^{*}$ as the formal adjoint of the exterior
differential operator $\dd$, and $\Ric$ is the Ricci operator.  Here also recall the Hodge star operator, $\ast$, sends $k$-forms to $n-k$-forms and is defined by  
\be\label{star}
\begin{split}
\alpha \wedge \ast \beta &=g(\alpha ,\beta)\Vol_M.
\end{split}
\ee
Then
\be\label{star2}
\ast \ast \alpha =(-1)^{nk+k}\alpha,
\ee
where $n$ is the dimension of the manifold, and $k$ the degree of $\alpha$, and
\be\label{star3}
\dd ^\ast \alpha=(-1)^{nk+n+1}\ast \dd \ast\alpha.
\ee
First, note that $L$ is the ordinary Laplacian on $\R^{n}$, since then $\Ric \equiv 0$.  Second, $L$
as given above sends $1$-forms into $1$-forms.  Hence, it is more convenient to write the Navier-Stokes equation
on a Riemannian manifold $M$ in terms of $1$-forms $U^{\ast}$ instead
of vector fields $U$ on $M$.  There is a natural correspondence between vector fields $U$ and $1$-forms $U^{\ast}$, which allows us to freely move between the two, and rewrite the equation as
\begin{equation}\tag{$\NS{M}$}
\begin{split}
\partial_{t}U^{*} - \Delta U^{*}  + \overline{\nabla}_{U}U^{*} - 2\Ric(U^{\ast})+ \dd P &= 0,\\
\dd^{\ast}U^{\ast}&=0,
\end{split}
\ee
where $\overline{\nabla}$ stands for the induced Levi-Civita connection on the cotangent
bundle $T^{*}M$.   Arguably less natural equation to study is the one without the Ricci operator
\begin{equation}\label{MNS}
\begin{split}
\partial_{t}U^{*} - \Delta U^{*}  + \overline{\nabla}_{U}U^{*} + \dd P &= 0,\\
\dd^{\ast}U^{\ast}&=0.
\end{split}
\end{equation}
In this article, to simplify notation we use $u$ to denote both the vector field and the corresponding $1$-form.

 In \cite{CC10} we studied both $\NS{M}$ and \eqref{MNS} and we showed

\begin{thm}[Non-uniqueness of $\NS{\H}$\cite{CC10}]\label{mainthmns}
Let $a>0$.  Then, $\NS{\H}$ is ill-posed in the following sense:
there exists smooth $u_{0}\in L^{2}(\H)$, such that there are infinitely many
smooth solutions $u$ satisfying $u(0)=u_0$ and
\begin{align}
\mbox{(finite energy)}&\quad \int_{\H} \abs{u(t)}^{2} < \infty,\label{energydef}\\
\mbox{( finite dissipation)}&\quad \int_{0}^{t} \int_{\H}   \abs{\Def u}^{2}<\infty, \label{dissipationdef}\\
 \mbox{(global energy inequality)} &  \int_{\H} \abs{u(t)}^{2}  + 4 \int_{0}^{t} \int_{\H}      \abs{\Def u}^{2}       \dd s \leq  \int_{\H} |u_{0}|^{2}.  \label{energyM}
\end{align}
\end{thm}
If we do not include the Ricci term in the equation, we can also have a non-uniqueness result on a more general negatively curved Riemannian manifold.
\begin{thm}\cite{CC10}\label{thm2ns}
Let  $a, b>0$ be such that $\frac 12 b<a\leq b$, and let $M$ be a simply connected, complete $2$-dimensional Riemannian manifold with sectional curvature satisfying $-b^{2} \leq K_{M} \leq -a^{2}$.  Then there exist non-unique solutions to \eqref{MNS} satisfying \eqref{energydef}-\eqref{energyM}.
\end{thm}

The non-uniqueness result relies on the existence of nontrivial bounded harmonic functions on negatively curved Riemannian manifolds due to the independent works of Anderson \cite{Anderson} and Sullivan \cite{Sullivan}.   The solution pairs $(u ,p)$ are (for $\NS{\H}$ and similar for \eqref{MNS})
\[
u = \psi(t) \dd F,\quad\mbox{and}\quad
 p = -\partial_{t}\psi(t) F - \frac{1}{2} \psi^{2}(t)|\dd F|^{2} - 2a^{2}\psi(t)F ,
 \]
where  $\psi(t) = \exp(-\frac{At}{2})$ for any $A \geq 4a^{2}$, and $F$ is a nontrival bounded harmonic function on $\mathbb{H}^{2}(-a^{2})$, which arises as a solution to the Dirichlet problem on $\H$ with $C^1$ boundary data prescribed on the sphere at infinity.

Verifying that $(u, p)$ solves $\NS{\H}$ is simple when we use Hodge theory.  In fact, taking solutions of the form $\psi(t)\nabla F$ seems to be a well known convention, and one could set out to try a similar solution on $\R^{n}$.  However, such solutions would not be interesting, because it would not be possible to show that they are even in $L^{2}$  since only bounded harmonic functions on $\R^{n}$ are trivial.  In the hyperbolic setting, given the abundance of the bounded harmonic functions, at least we have a hope, but a priori, it is not obvious that our solutions have to satisfy  \eqref{energydef}-\eqref{energyM}. Hence the main contribution of \cite{CC10} stems from showing \eqref{energydef}-\eqref{energyM}.
\newline\indent  Another proof of the existence of the bounded nontrivial harmonic function \cite{Anderson, Sullivan} can be found in \cite{AndersonSchoen}.   Study of that proof combined with the gradient estimate of S.T. Yau \cite{Yau75} leads to an exponential decay of the gradient of the harmonic function
 \be\label{expdecaydF}
|\nabla F| \leq C(a, \delta, \phi' )e^{-\delta \rho},
\ee
where $\delta<a$ is some constant, $\rho$ is the distance function, and $\phi$ is a $C^1$ boundary data for $F$ at infinity.  Estimate \eqref{expdecaydF} very easily gives property $\eqref{energydef}$, but also is a reason why we could not obtain the result in 3D (for more details see \cite{CC10}).
In \cite{Taylor}, among other things, Taylor showed that on a $2D$ hyperbolic space the decay can be improved to $e^{-a\rho}$ if the data is assumed to be $C^{1+\epsilon}$.

 \subsection{Other dimensions}\label{otherD}
In \cite{KhesinMisiolek}, by means of the result of Dodziuk \cite{Dodziuk}, Khesin and Misio\l ek showed that our construction can only work in 2D.  The main idea is that on $\mathbb H^n(-a^2)$ the only $L^2$ harmonic forms are of degree $k=\frac n2$.  So since we are using $F$ harmonic to obtain $\dd F$ which is a harmonic form of degree $1$,  this construction can only work in $n=2$!  The question of a general negatively curved Riemannian manifold is open as far as we know, but because of \cite{Dodziuk} we do not expect it to be true, especially for simply connected manifolds.

\subsection{Simpler Proofs}\label{simpler}
With hindsight, we could now simplify our previous proofs on $\H$ and partially on a more general manifold.  To see this, we could use the result of \cite{Dodziuk} to know that $\dd F$ belongs to $L^{2}(\H)$ since it is a harmonic 1-form, and $\H$ satisfies properties of the manifolds considered in \cite{Dodziuk}.  More directly, as done in \cite{KhesinMisiolek} and in \cite{Taylor} one can use the conformal equivalence of the Poincar\'e disk and the Euclidean unit disk together with standard elliptic theory to show that $\dd F$ is in $L^2$.  However, to treat more general Riemannian manifolds, we would still need \eqref{expdecaydF}. 
 
 So this leaves showing  \eqref{dissipationdef}.  In \cite{CC10} we reduced it to showing the $L^{1}$ property of $|\nabla |\nabla F|^{2}|$, which was somewhat involved.   It was pointed out to us by M. Struwe that we could eliminate that work by manipulating further one of the already existing formulas and integrating by parts.  This observation has further consequences, and we will present it in a forthcoming article.

\subsection{Classical uniqueness proofs \& why they do not work on $\H$. }

We are aware of two ways that the uniqueness is proved.  One approach (see for example \cite{Taylor3}) first takes the equation and applies the Leray projection to the equation to eliminate the pressure term.  If we apply the Leray projection to the equation when we use our solution pairs, everything projects to $0$.  In the second approach (see \cite{Temam}) the equation is formulated in a way that has already eliminated the pressure term and the uniqueness is shown in the following space
\begin{align*}
{\textbf V}= \mbox{completion of}\  \{ \theta=(\theta_1,\theta_2), \theta_j \in C^{\infty}_{c}(\R^{n}): \dv\theta=0\}\ \ \mbox{with respect to $H^{1}(\R^{2})$ norm}.
\end{align*}
Now compare ${\textbf V}$ to
\begin{align*}
\widetilde{\tV}= \{ u \in H^{1}(\R^{2}) : \dv u=0\}.
\end{align*}
The closure in $\tV$ is taken \emph{after} the divergence condition is imposed, and in $\widetilde \tV$  \emph{before}.  On $\R^{N}$ these two spaces coincide.  It was pointed out by Heywood \cite{Heywood} that this is not true in general, and whether or not these two spaces coincide is equivalent to having uniqueness for both stationary and non-stationary Stokes and Navier-Stokes equations.  For example, it is not true for $$\{ x \in \R^{3}: x_{1}^{2}+x_{2}^{2}< 1 + x^{2}_{3}\}.$$
It is natural to ask if the analogs of these spaces coincide on $\H$? It is a different domain and a different equation. Corollary \ref{heywoodthm} says that the answer is no.  The reason is due precisely to the existence of nontrivial $L^{2}$ harmonic 1-forms.

At the same time, this motivates considering the space $\tV$ as the setting, in which one can strive to prove uniqueness.  Moreover, one can use the space $\tV$ and follow the presentation in \cite{Temam} to establish what it means to solve the equation in the hyperbolic setting and how to obtain existence.  The presentation of \cite{Temam} is done in a general functional analysis framework, so it seems reasonable to expect that it could be extended to include $\H$.  In this article, we set out to do just that. 

 In addition, we formulate the problem not only in $\tV$, but in $$\V=\tV  \oplus \mathbb F,$$ where
\[
\F=\{ \alpha \in L^2(\H) : \alpha=\dd F,\  F \ \mbox{is a harmonic function on} \  \H \}.
\]
Note, by definition the non-unique solutions we considered before belong to $\F$.  Yet, the problem can be formulated in  $\tV  \oplus \mathbb F$ so the resulting solutions are still unique.  We would like to stress that working with space $\F$ demands extra work, so we cannot  just directly quote \cite{Temam}.

\subsection{Summary of function spaces}\label{summary}
 For convenience of the reader we summarize the function spaces employed in the article.  We have
\begin{itemize}
\item $\Lambda_{c}^1(\mathbb{H}^2(-a^2))$: the space of all smooth 1-forms with compact support on $\H$.
\item $\Lambda_{c,\sigma}^1(\mathbb{H}^2(-a^2))=\{\theta \in \Lambda_{c}^1(\mathbb{H}^2(-a^2)): \dd^* \theta = 0\}$.
\item $\textbf{V} = \overline{\Lambda_{c,\sigma}^1(\mathbb{H}^2(-a^2))}^{H^1}$:  the completion of $\Lambda_{c,\sigma}^1(\mathbb{H}^2(-a^2))$ with respect to the $H^1$-norm, which is defined in Definition \ref{H1def}.
\item $\mathbb{F} = \{ \alpha \in L^2(\H) : \alpha=\dd F,\  F \ \mbox{is a harmonic function on} \  \H \}$.
\item $\mathbb{V} = \textbf{V} \oplus \mathbb{F}$.
\item $\textbf{H} = \overline{\Lambda_{c,\sigma}^1(\mathbb{H}^2(-a^2))}^{L^2}$.
\end{itemize}
We note that we are using the same notation $\textbf{V}$ and $\textbf{H}$ as was used in \cite{Temam}, where it denoted the Euclidean counterparts.

\subsection{Main results}
The main result of this article is
\begin{thm}\label{IMPORTANTTHEOREM}
Given any initial data $u_0 \in \textbf{H}\oplus \mathbb{F}$ and any $T \in (0,\infty )$, there exists a unique $u \in C(0,T; \textbf{H}\oplus \mathbb{F} )\cap L^2(0,T; \mathbb{V} )$ which satisfies
\begin{itemize}
\item [a)] $\partial_t u \in L^2(0,T ; \mathbb{V}' )$.
\item [b)]  $u(0,\cdot ) = u_0$ in $\textbf{H}\oplus \F$.
\item [c)]  The following equation holds with the terms to be interpreted as elements in  $L^2(0,T ; \mathbb{V}' )$
\begin{equation}\label{FINALEQUATION}
\partial_t u  + Au + Bu = 0,
\end{equation}
where the terms $Au(t), Bu(t) \in \mathbb{V}'$ are defined, for almost every $t$, as follows
\begin{equation}\label{accuratedefofAB}
\begin{split}
\ip{Au(t) , v}_{\mathbb{V}'\otimes \mathbb{V}} & =2 \int_{\mathbb{H}^2(-a^2)} \overline{g} (\Def u(t) , \Def v) \Vol_{\mathbb{H}^2(-a^2)}
, \\
\ip{Bu(t) , v}_{\mathbb{V}'\otimes \mathbb{V}} & = \int_{\mathbb{H}^2(-a^2)} \overline{g}(\overline{\nabla}_uu(t) , v ) \Vol_{\mathbb{H}^2(-a^2)} .
\end{split}
\end{equation}
\end{itemize}
\end{thm}

\begin{remark}\label{VERYFIRSTRE}
The formulation of Theorem \ref{IMPORTANTTHEOREM} for the sake of brevity implicitly uses the following
\begin{itemize}
\item $\mathbb{F} \subset H^1_0(\mathbb{H}^2(-a^2))$, which is proved through Lemma \ref{IMFinitedisspation} and Lemma \ref{Anothertrivial} ;
\item $Au \in L^2(0,T ; \mathbb{V}')$ holds, for any $u \in L^2(0,T ; \mathbb{V})$. This is justified in Lemma \ref{AONE};
\item $Bu \in L^2(0,T; \mathbb{V}')$ holds, for any $u \in L^2(0,T ; \mathbb{V})$. This is based on estimate \eqref{b2} in Lemma \ref{ATWO}.
\end{itemize}
\end{remark}
We also have
\begin{cor}[Conservation of energy of the Navier-Stokes flow on $\mathbb{H}^2(-a^2)$]\label{ConserLemma}
Given initial data $u_0 \in \textbf{H}\oplus \mathbb{F}$, and $T \in (0, \infty )$, let $u \in C(0,T; \textbf{H}\oplus \mathbb{F} )\cap L^2(0,T; \mathbb{V} )$  be the unique element which satisfies the three conditions as specified in Theorem \ref{IMPORTANTTHEOREM}. Then $u$ satisfies the energy conservation law in the sense that the following holds for all $t \in [0,T]$
\begin{equation}\label{ConservationEnergy}
\|u(t,\cdot )\|_{L^2(\mathbb{H}^2(-a^2))}^2 + 4\int_{0}^{t} \|\Def u(t,\cdot)\|_{L^2(\mathbb{H}^2(-a^2))}^2\dd t = \|u_0\|_{L^2(\mathbb{H}^2(-a^2))}^2.
\end{equation}
\end{cor}

We also show  the survival of \emph{one} solution from the family of the non-unique solutions we considered before.
\begin{cor}[Survival of one solution]\label{survivalcor}
Given $u_0=\dd F \in \F$, $e^{-2a^2t}\dd F$ is the unique solution to the variational problem \eqref{FINALEQUATION} which satisfies all the properties $a)-c)$ in Theorem \ref{IMPORTANTTHEOREM}.  This solution also satisfies \eqref{ConservationEnergy}.
\end{cor}

We address the pressure in the following proposition.
\begin{prop}[Associated pressure] \label{pressuremainlemma}
Let $u\in C([0,T]; \textbf{H}\oplus \F)\cap L^2(0,T; \V)$ be the Leray-Hopf solution obtained in Theorem \ref{IMPORTANTTHEOREM}. Then there exists a family of functions $\textbf{P}(t) \in L^2_{loc}(\H)$ with $t\in [0,T]$ such that the following holds in the sense of $H^{-1}(\H)$
\[
u(t) - u_0 + A U(t) + \textbf{B}(t) = \dd \textbf{P}(t),\quad \forall t \in [0,T],
\]
where
\[
U(t)=\int^t_0 u(s) \dd s\quad\mbox{and}\quad \textbf{B}(t)=\int^t_0 Bu(\tau)\dd \tau.
\]

Moreover, there exists a distribution $p\in \mathcal D'(\H\times(0,T))$ such that
\be\label{eqwithp}
\partial_{t}u +\Delta u  + \overline{\nabla}_{u}u - 2\Ric(u)+ \dd p = 0\ee
holds in the sense of distributions.
\end{prop}
\begin{remark}
We do not have a deep investigation about the regularity property or the far range behavior of the pressure. The reason is that the concept of the pressure does not get involved in the conservation of energy, and hence it is of a secondary nature with respect to the main guiding principle of this article (see next subsection). 
\end{remark}
The following theorem shows that the space $\V$ is the natural choice for the weak formulation.
 \begin{thm}[$\V=\widetilde{\tV}$]\label{everything}   Let  $\V$ be as defined in Section \ref{summary}, and let \text{\textnormal{$\widetilde{\textbf V}$}} be given by
\[
\mbox{ \text{\textnormal{$\widetilde{\textbf V}$}}}=\{ u \in H^1_0(\H): \dd^\ast  u=0\}.
\]
Then $\V=$ \text{\textnormal{$\widetilde{\textbf V}$}}.
\end{thm}
\begin{remark}
In light of Theorem \ref{everything}, one can, of course, switch the notation $\mathbb{V} = \textbf{V} \oplus \mathbb{F}$ to $\widetilde{\tV}$. However, we still employ the notation $\mathbb{V} = \textbf{V} \oplus \mathbb{F}$ throughout the article, since in this way, the nontrivial functional subclass $\mathbb{F}$ in $\mathbb{V} =  \widetilde{\tV}$ is the most apparent to the readers.
\end{remark}
Then immediately it follows
\begin{cor}[Heywood's Theorem on $\H$]\label{heywoodthm}  Let   \text{\textnormal{${\textbf V}$}} be as defined in Section \ref{summary}. Then  \text{\textnormal{${\textbf V}$}} is strictly contained in  \text{\textnormal{$\widetilde{\textbf V}$}}.
\end{cor}
 
\subsection{Guiding principle for restoring uniqueness}\label{guide}

We now make explicit the guiding principle behind our present work. Recall that we have to restore the uniqueness property of finite energy solutions to the Navier-Stokes equation on $\mathbb{H}^2(-a^2)$ through ruling out all (except possibly one) in the following family of solutions
\begin{equation}\label{infinitemany}
\exp (\frac{-At}{2}) \dd F ,
\end{equation}
with $A \geq 4a^2$, and $\dd F \in \mathbb{F}$. In dealing with this issue we have two different options available:
\begin{itemize}
\item [\textbf{option I:}] Restore uniqueness by excluding all of $\mathbb{F}$ from the class of admissible finite energy initial data for the Cauchy problem of the Navier-Stokes equation.
\item [\textbf{option II}:] Accept all elements in $\mathbb{F}$ into the class of admissible finite energy initial data, yet select the single, most physically meaningful solution among the infinite family of solutions in \eqref{infinitemany}, which all arise from the same initial data $\dd F \in \mathbb{F}$.
\end{itemize}
Choosing one among the above two options is a task of making a philosophical judgement.   Our choice is guided by the following:
\begin{itemize}
\item \textbf{Lack of empirical data:} To our knowledge the experiments about the behavior of viscous incompressible fluid flows are done mostly in the Euclidean setting. Due to the lack of laboratory data about the behavior of incompressible fluid flows which take place on an open, noncompact curved space, we do not have any physical considerations about the empirical behavior of viscous incompressible fluid flows on  $\mathbb{H}^2(-a^2)$ to take into account.
\item \textbf{Conservation of energy:} Regardless of the actual mechanism leading to the phenomena of energy dissipation, any viscous incompressible fluid flow which is smooth, and which occurs without the involvement of an external force should obey conservation of energy in that: the total kinetic energy of the fluid at the time $t > 0$ plus the cumulative portion of energy being lost due to dissipation up to the time $t$ equals the total kinetic energy of the fluid at the initial time $t =0$.
\end{itemize}

The lack of empirical data is here to stop one from borrowing conventional physical considerations of viscous incompressible fluid flows in the $\mathbb{R}^2$ setting as guiding principles in making one's judgement. For instance, the vanishing property of the pressure in the far range of the domain of the fluid flow is based on a large body of well-understood, well-accepted laboratory results. The lack of empirical facts about viscous incompressible fluid flows on an infinite open curved space like $\mathbb{H}^2(-a^2)$ makes these conventional considerations in the Euclidean setting no longer reliable in the new setting of $\mathbb{H}^2(-a^2)$.

Next, we arrive at the conservation of energy.  The description of this conservation of energy law as stated above should be understood in the sense of mathematics. The statement of the law is a mere theoretical statement in that it provides no information at all about the actual mechanism leading to dissipative phenomena in viscous fluid flows in the empirical sense.
This pure \emph{a priori} property of the conservation law of energy is exactly what makes it acquire the status of a grounding principle for our forthcoming choice in resolving the non-uniqueness dilemma.

Now, observe that option II has a broader scope when compared with option I. This indicates that one should first check with option II to see whether or not it is compatible with the \emph{a priori} principle of the conservation law of energy. If it were true that option II violates the conservation law of energy, then, one would have to abandon it and proceed to consider option I.

In order to validate option II as a choice compatible with conservation of energy, one needs to carry out the following three steps.
\begin{itemize}
\item \textbf{step I:} Show that $\Def \dd F \in L^2(\mathbb{H}^2(-a^2))$ for any $\dd F \in \mathbb{F}$.
\item \textbf{step II:} Observe that among the solutions in the infinite family \eqref{infinitemany}, $\exp (-2a^2) \dd F$ is the one and only one which satisfies the conservation of energy law.
\item \textbf{step III:} Build up a self-contained theory of finite energy, finitely dissipative, weak solutions which embraces $\textbf{H}\oplus \mathbb{F}$ as the functional class of admissible finite energy initial velocity fields, and which is compatible with the law of conservation of energy.
\end{itemize}
Among these steps, step I is the most crucial one since it would not make sense to talk about the conservation of energy if the finite dissipation property were violated. Lemma \ref{IMFinitedisspation} below indicates that {option II} survives the test of {step I}.   We emphasize that step I or equivalently Lemma \ref{IMFinitedisspation} cannot just follow from \cite{CC10}, where  $\Def \dd F \in L^2(\mathbb{H}^2(-a^2))$ was shown for harmonic functions $F$, which were obtained as solutions of the Dirichlet problem with $C^1$ data at infinity.  Since the boundary behavior of $F$ is unknown, we need a new approach.

{Step II} is confirmed in Section \ref{confirm}.  Step III is completed in Sections \ref{exist} and \ref{unique} through successfully building up the Leray-Hopf theory by means of the classical Faedo-Galerkin approximation method \cite[pp. 172-200]{Temam}.

Theorem \ref{IMPORTANTTHEOREM} is the cumulative result which satisfies all the requests in {step III}. However, the conclusion of Theorem \ref{IMPORTANTTHEOREM} remains equally valid if one were to replace $\textbf{H}\oplus \mathbb{F}$ by the more restrictive space $\textbf{H}$ as the functional class of admissible initial velocity fields. This means that the classical Faedo-Galerkin approximation method is {uncritical} in regard to the proper choice of the functional space of admissible initial velocity fields, as well as that of the functional class of weak solutions. This in turns indicates that the Faedo-Galerkin method, however powerful and useful as a {tool} leading to Theorem \ref{IMPORTANTTHEOREM}, cannot serve as a guiding principle in choosing the proper functional class of admissible initial velocity fields. Our attitude is: our whole decision is guided by the conservation of energy law as a {first principle}, and the Faedo-Galerkin method is only used as a tool in completing the quest in {step III} under the guidance of the conservation of energy.

 \subsection{Organization of the article}

The rest of the article is organized as follows. In Section \ref{FAH} we set up the tools for the functional analysis on $\H$.  This involves for example the definition of  the deformation tensor $\Def u$ in the weak sense and definitions of Sobolev spaces for $1$-forms using the deformation tensor.  We finish the section with a proof of the Ladyzhenskaya inequality on the hyperbolic space.

In Section \ref{MAINMAINMAIN} we establish several properties of the space $\F$ as a subspace of $H^1_0$.  For instance, we show that finite energy of the elements in the space implies their finite dissipation, and that $\F$ is orthogonal to $\V$ with respect to the inner product on $H^1_0$.

In Section \ref{AB} we define and give estimates for the variational Stokes operator $A$, and of the nonlinear operator $B$.

In Section \ref{exist} we carry out the Faedo-Galerkin procedure in seven steps to show the existence of the weak Leray-Hopf solution.  In the last step, we show the initial data is satisfied.

In Section \ref{unique} we give a short proof of uniqueness.  Section \ref{Briefsubsect} is devoted to the proofs of the corollaries.  Section \ref{pressure} treats the pressure.  There we use the language of the currents and follow DeRham \cite{DeRhamEng}. Finally, in Section \ref{EVERYTHINGSECT} using the Hodge-Kodaira decomposition \cite{Kodaira} we show Theorem \ref{everything}.

To minimize needed background in geometry, we placed technical computations in coordinates (with the exception of Lemma \ref{IMFinitedisspation} ) in Appendix \ref{appendixa} and the reader, if they so wish, can only refer to the final results quoted in the main body of the article.   
 
In Appendix \ref{awesome} we use elementary complex analysis to establish a lemma needed to show finite dissipation of the elements in $\F$.  Finally, Appendix \ref{faa} gathers functional analysis lemmas from \cite{Temam}, and we give a short proof regarding choosing a special kind of basis in a separable Banach space.
 \section*{Acknowledgements}

We would like to thank Pawe\l  \ Konieczny for discussions  in 2010 about the weak formulation, Helmholtz and Hodge decompositions and the non-uniqueness of Heywood.  The first author would like to thank Binghamton University and the second author would like to thank National Chiao Tung University and IMA at Minnesota for their hospitality during the respective visits when this work was completed.

We would like to thank M. Anderson, J. Colliander, R. Jerrard, M. Keel, J. Kelliher, B. Khesin, A. Mazzucato, M. Struwe, D. Sullivan, V. \v{S}ver\'ak and M. Taylor for their interest in our work.

The first author was partially supported by a grant from the National Science Council of Taiwan (NSC 101-2115-M-009-016-MY2).  The second author was partially supported by a grant from the Simons Foundation \#246255.

\section{Functional Analysis on $\H$}\label{FAH}

\subsection{Weak derivatives }

We denote by $g(\cdot ,\cdot )$ and $\overline{g}(\cdot , \cdot )$ the induced Riemannian metrics on $T^*\mathbb{H}^2(-a^2)$ and $\otimes^2 T^*\mathbb{H}^2(-a^2)$ respectively.

Recall that, for any smooth $1$-form $u$ on $\mathbb{H}^2(-a^2)$, $\Def u$ is, by definition, the symmetrization of the tensor $\overline{\nabla} u$, with $\overline{\nabla}$ to be the induced Levi-Civita connection acting on the space of smooth $1$-forms on $\mathbb{H}^2(-a^2)$.    Then the formal adjoint operator $$\Def^{\ast} : C^{\infty}(\otimes^2 T^* \mathbb{H}^2(-a^2) ) \rightarrow C^{\infty}(T^* \mathbb{H}^2(-a^2))$$ can be defined as follows
\begin{defn}
Given a smooth tensor $\theta\in C^{\infty}(\otimes^2 T^* \mathbb{H}^2(-a^2) )$, $\Def^* \theta$ is the uniquely determined smooth $1$-form on $\mathbb{H}^2(-a^2)$ for which the following relation holds for any smooth $1$-form $u $ which has compact support in $\mathbb{H}^2(-a^2)$
\begin{equation}\label{abstractcharacter}
\int_{\mathbb{H}^2(-a^2)} \overline{g} (\Def u , \theta ) \Vol_{\mathbb{H}^2(-a^2)} = \int_{\mathbb{H}^2(-a^2)} g(u , \Def^* \theta ) \Vol_{\mathbb{H}^2(-a^2)}.
\end{equation}
\end{defn}

We next give the definition of $\Def u \in L^2(\mathbb{H}^2(-a^2))$ in the weak sense for an $L^2(\mathbb{H}^2(-a^2))$-integrable vector field $u$ on $\mathbb{H}^2(-a^2)$.
\begin{defn}
[$\Def u$ \textbf{in the weak sense}]\label{DefofWeakDerivative} Let $u$  be an $L^2$-integrable $1$-form on $\mathbb{H}^2(-a^2)$. We say that $\Def u \in L^2(\mathbb{H}^2(-a^2))$ exists in the \emph{weak sense} if there exists some $L^2$-integrable tensor $\omega$ of type
$\otimes^2 T^* \mathbb{H}^2(-a^2)$ on $\mathbb{H}^2(-a^2)$ for which the following relation holds for any \emph{compactly supported} smooth tensor $\theta$ of type $\otimes^2 T^* \mathbb{H}^2(-a^2)$ on $\mathbb{H}^2(-a^2)$
\begin{equation}\label{abstractcharacterWEAK}
\int_{\mathbb{H}^2(-a^2)} \overline{g} (\omega , \theta ) \Vol_{\mathbb{H}^2(-a^2)} = \int_{\mathbb{H}^2(-a^2)} g(u , \Def^* \theta ) \Vol_{\mathbb{H}^2(-a^2)} .
\end{equation}
Such tensor $\omega$ will be understood as $\Def u$ in the weak sense, and we will simply write $\Def u = \omega$.
\end{defn}

For the purpose of dealing with the pressure term, we also need the following definitions of $\dd u$ and $\dd u^*$ in the weak sense.

\begin{defn}\label{Defford}
Let $u$ be an $L^2$-integrable $1$-form on $\mathbb{H}^2(-a^2)$. Then we say that $\dd u \in L^2(\mathbb{H}^2(-a^2))$ exists in the weak sense if there exists an $L^2$-integrable $2$-form $\omega$ on $\mathbb{H}^2(-a^2)$ for which the following relation holds for any smooth $2$-form $\phi$ with compact support on $\mathbb{H}^2(-a^2)$
\begin{equation}
\int_{\mathbb{H}^2(-a^2)} \overline{g} (\omega , \phi ) \Vol_{\mathbb{H}^2(-a^2)} = \int_{\mathbb{H}^2(-a^2)}g (u, \dd^* \phi ) \Vol_{\mathbb{H}^2(-a^2)}.
\end{equation}
In such a case, such $\omega$ will be  called $\dd u$ in the weak sense, and we will simply write $\omega = \dd u$.
\end{defn}

\begin{defn}\label{DeffordSTAR}
Let $u$  be an $L^2$-integrable $1$-form on $\mathbb{H}^2(-a^2)$, then, we say that $\dd^*u \in L^2(\mathbb{H}^2(-a^2))$ exists in the weak sense if there exists an $L^2$-integrable function $w$ on $\mathbb{H}^2(-a^2)$ for which the following relation holds for any smooth test function $\phi \in C_c^{\infty}(\mathbb{H}^2(-a^2))$
\begin{equation}
\int_{\mathbb{H}^2(-a^2)} w\cdot v \Vol_{\mathbb{H}^2(-a^2)} = \int_{\mathbb{H}^2(-a^2)} g (u, \dd \phi ) \Vol_{\mathbb{H}^2(-a^2)}.
\end{equation}
In this case, we say that $w$ is called $\dd^\ast  u$ in the weak sense, and we will write $w = \dd^*u$.
\end{defn}

\subsection{Sobolev Spaces}

Next we have the definition of the Sobolev space of $L^2$-integrable $1$-forms $u$ with $L^2$-weak derivative $\Def u$ on $\mathbb{H}^2(-a^2)$.
\begin{defn}\label{H1def}
$H^1 (\mathbb{H}^2(-a^2) , T^*\mathbb{H}^2(-a^2)  )$ is the space of all $L^2$-integrable $1$-forms $u$ with $L^2$-integrable weak derivatives $\Def u$ on $\mathbb{H}^2(-a^2)$. Moreover, $H^1 (\mathbb{H}^2(-a^2) , T^*\mathbb{H}^2(-a^2)  )$ is a Hilbert space equipped with the following inner product \begin{equation}
[u , v ] = \int_{\mathbb{H}^2(-a^2)} g (u , v) \Vol_{\mathbb{H}^2(-a^2)} + 2 \int_{\mathbb{H}^2(-a^2)} \overline{g} (\Def u , \Def v) \Vol_{\mathbb{H}^2(-a^2)}.
\end{equation}
Of course, the Sobolev norm of $H^1 (\mathbb{H}^2(-a^2) , T^*\mathbb{H}^2(-a^2)  )$ is given by $\|u\|_{H^1} = [u,u]^{\frac{1}{2}}$. In general, one will write the inner product structure $[u,v]$ in the following convenient manner.
\begin{equation}\label{H1norm}
[u,v] = (u,v) + ((u, v)) ,
\end{equation}
where the terms $(u,v)$ and $((u,v))$ are just
\begin{equation}\label{homogeneouspart}
\begin{split}
(u,v) & = \int_{\mathbb{H}^2(-a^2)} g (u , v) \Vol_{\mathbb{H}^2(-a^2)} ,\\
((u,v)) & = 2 \int_{\mathbb{H}^2(-a^2)} \overline{g} (\Def u , \Def v) \Vol_{\mathbb{H}^2(-a^2)}.
\end{split}
\end{equation}
The term $((u,u))^{\frac{1}{2}}$ is often called the homogeneous part of the $H^1$-norm $\|u\|_{H^1}$ of $u$ .
\end{defn}

\begin{remark}
For simplicity, we will often use the symbol $H^1(\mathbb{H}^2(-a^2))$ (or even just $H^1$) to abbreviate $H^1 (\mathbb{H}^2(-a^2) , T^*\mathbb{H}^2(-a^2)  )$. The same remark also applies when we speak of the $L^2$ norm of a  $1$-form $u$ on $\mathbb{H}^1(-a^2)$.

\end{remark}

Recall that $\Lambda_{c}^1(\mathbb{H}^2(-a^2))$ is the space of all smooth $1$-forms with compact support on $\mathbb{H}^2(-a^2)$.  Next

\begin{defn}\label{H1def0}
$H^1_{0}(\mathbb{H}^2(-a^2))$ is the completion of the space $\Lambda_{c}^1(\mathbb{H}^2(-a^2))$ with respect to the $H^1$-norm as defined in  \eqref{H1norm}.

\end{defn}

Next we have the following easy lemma, which asserts that $\Def u$ is the strongest type of a weak derivative when being compared with the weak derivatives $\dd u$ and $\dd^* u$ in that the latter two can be recovered from $\Def u$ alone.

\begin{lemma}\label{TrivialLemma}
For any $1$-form $u  \in H^1_{0}(\mathbb{H}^2(-a^2))$, it follows that both weak $\dd u \in L^2(\mathbb{H}^2(-a^2))$ and weak $\dd^* u \in L^2 (\mathbb{H}^2(-a^2))$ exist in the sense of Definitions \ref{Defford} and \ref{DeffordSTAR} respectively. Moreover, $\|\dd u\|_{L^2(\mathbb{H}^2(-a^2))}$ and $\|\dd^*u\|_{L^2(\mathbb{H}^2(-a^2))}$ are related to the homogeneous part $((u,u))^{\frac{1}{2}}$ of the $H^1$-norm of $u$ through the following identity

\begin{equation}\label{strongestderivative}
\begin{split}
((u,u)) & = 2\int_{\mathbb{H}^2(-a^2)} \overline{g} (\Def u , \Def u ) \Vol_{\mathbb{H}^2(-a^2)} \\
 & = 2\|\dd^*u\|_{L^2(\mathbb{H}^2(-a^2)}^2 + \|\dd u\|_{L^2(\mathbb{H}^2(-a^2))}^2 + 2a^2\|u\|_{L^2(\mathbb{H}^2(-a^2))}^2 .
\end{split}
\end{equation}
\end{lemma}
\begin{proof}
Using
\be\label{defnDefstarDef}
2\Def^* \Def w = 2 \dd \dd^*w + \dd^* \dd w + 2a^2w,
\ee
 which holds for all $w \in \Lambda_c^1(\mathbb{H}^2(-a^2))$,  we integrate by parts to obtain 
\begin{equation}\label{intbypartseasy}
\begin{split}
& 2\int_{\mathbb{H}^2(-a^2)} \overline{g} \big ( \Def w , \Def w \big ) \Vol_{\mathbb{H}^2(-a^2)}
= 2   \int_{\mathbb{H}^2(-a^2)} g(w , \Def^* \Def w ) \Vol_{\mathbb{H}^2(-a^2)} \\
&\quad=2 \int_{\mathbb{H}^2(-a^2)} g(w, \dd \dd ^* w) \Vol_{\mathbb{H}^2(-a^2)} + \int_{\mathbb{H}^2(-a^2)} g(w , \dd^* \dd w )\Vol_{\mathbb{H}^2(-a^2)} \\
&\qquad\qquad + 2a^2 \int_{\mathbb{H}^2(-a^2)} g(w, w) \Vol_{\mathbb{H}^2(-a^2)} \\
&\quad= 2\int_{\mathbb{H}^2(-a^2)} |\dd^*w|^2  \Vol_{\mathbb{H}^2(-a^2)} + \int_{\mathbb{H}^2(-a^2)} \overline{g} (\dd w , \dd w ) \Vol_{\mathbb{H}^2(-a^2)} \\
&\qquad\qquad + 2a^2 \int_{\mathbb{H}^2(-a^2)} g(w, w) \Vol_{\mathbb{H}^2(-a^2)} .
\end{split}
\end{equation}
Now, we take any $1$-form $u \in H^1_0(\mathbb{H}^2(-a^2))$. Then, we can take, by the definition of $H^1_0(\mathbb{H}^2(-a^2))$, a sequence $\{w_m\}_{m=1}^{\infty} \subset \Lambda_{c}^1(\mathbb{H}^2(-a^2))$ such that we have $\|w_m - u\|_{H^1} \rightarrow 0$, as $m\rightarrow \infty$. But identity \eqref{intbypartseasy} gives the following relation, for any positive integers $k,l$
\begin{equation}\label{EASY}
\begin{split}
\|w_k- w_l\|_{H^1}^2 \geq & 2\int_{\mathbb{H}^2(-a^2)} \overline{g} \big ( \Def (w_k-w_l) , \Def (w_k-w_l) \big ) \Vol_{\mathbb{H}^2(-a^2)} \\
= & 2\int_{\mathbb{H}^2(-a^2)} |\dd^*(w_k -w_l)|^2  \Vol_{\mathbb{H}^2(-a^2)} \\
&\quad+ \int_{\mathbb{H}^2(-a^2)} \overline{g} (\dd (w_k-w_l) , \dd (w_k-w_l) ) \Vol_{\mathbb{H}^2(-a^2)} \\
&\quad + 2a^2 \int_{\mathbb{H}^2(-a^2)} g(w_k-w_l, w_k-w_l) \Vol_{\mathbb{H}^2(-a^2)} .
\end{split}
\end{equation}
since $\|w_k- w_l\|_{H^1}\rightarrow 0$, as $k,l\rightarrow \infty$, it follows from \eqref{EASY} that the sequence $\{\dd^* w_m\}_{m=1}^{\infty}$ of smooth functions is Cauchy in $L^2(\mathbb{H}^2(-a^2))$, and also that the sequence $\{\dd w_m\}_{m=1}^{\infty}$ of smooth $2$-forms is Cauchy in the Banach space of $L^2$-integrable $2$-forms on $\mathbb{H}^2(-a^2)$. So, there must exist some unique limiting function $f^* \in L^2(\mathbb{H}^2(-a^2))$, and some unique limiting $L^2$-integrable $2$-form $\omega^* \in L^2(\mathbb{H}^2(-a^2))$ such that we have
\begin{equation}\label{GoodLimitRel}
\begin{split}
\lim_{m\rightarrow \infty } \|\dd^*w_m - f^*\|_{L^2(\mathbb{H}^2(-a^2))} & = 0 , \\
\lim_{m\rightarrow \infty } \|\dd w_m - \omega^* \|_{L^2(\mathbb{H}^2(-a^2))} & = 0 .
\end{split}
\end{equation}
Next, observe that we definitely have the following relation, with $\phi \in C_c^{\infty}(\mathbb{H}^2(-a^2))$ to be any test function.
\begin{equation}\label{supone}
\int_{\mathbb{H}^2(-a^2)} \dd^* w_m \phi \Vol_{\mathbb{H}^2(-a^2)} =  \int_{\mathbb{H}^2(-a^2)} g(w_m , \dd\phi ) \Vol_{\mathbb{H}^2(-a^2)} .
\end{equation}
Then, the first line of \eqref{GoodLimitRel} allows us to pass to the limit in \eqref{supone} to yield the following relation,
\begin{equation}
\int_{\mathbb{H}^2(-a^2)} f^* \phi \Vol_{\mathbb{H}^2(-a^2)} =  \int_{\mathbb{H}^2(-a^2)} g(u , \dd\phi ) \Vol_{\mathbb{H}^2(-a^2)} .
\end{equation}
This indicates that the relation $\dd^*u = f^* \in L^2(\mathbb{H}^2(-a^2))$ holds in the weak sense. Also, we can consider the following identity, which holds for all $\eta \in \Lambda_c^2(\mathbb{H}^2(-a^2))$
\begin{equation}\label{suptwo}
\int_{\mathbb{H}^2(-a^2)} \overline{g}(\dd w_m , \eta ) \Vol_{\mathbb{H}^2(-a^2)} = \int_{\mathbb{H}^2(-a^2)} g(w_m, \dd^* \eta ) \Vol_{\mathbb{H}^2(-a^2)}.
\end{equation}
Then, the second line of \eqref{GoodLimitRel} allows us to pass to the limit in \eqref{suptwo} to yield the following relation  \begin{equation}
\int_{\mathbb{H}^2(-a^2)} \overline{g}(\omega^* , \eta ) \Vol_{\mathbb{H}^2(-a^2)} = \int_{\mathbb{H}^2(-a^2)} g(u, \dd^* \eta ) \Vol_{\mathbb{H}^2(-a^2)}.
\end{equation}
This means that $du = w^* \in L^2(\mathbb{H}^2(-a^2))$ holds in the weak sense. Finally, notice that relation \eqref{intbypartseasy} gives
\begin{equation}\label{LastEasy}
\begin{split}
& 2\int_{\mathbb{H}^2(-a^2)} \overline{g} \big ( \Def w_m , \Def w_m \big ) \Vol_{\mathbb{H}^2(-a^2)} \\
= & 2\int_{\mathbb{H}^2(-a^2)} |\dd^*w_m|^2  \Vol_{\mathbb{H}^2(-a^2)} + \int_{\mathbb{H}^2(-a^2)} \overline{g} (\dd w_m , \dd w_m ) \Vol_{\mathbb{H}^2(-a^2)} \\
& + 2a^2 \int_{\mathbb{H}^2(-a^2)} g(w_m, w_m) \Vol_{\mathbb{H}^2(-a^2)} .
\end{split}
\end{equation}
By passing to the limit in \eqref{LastEasy}, as $m\rightarrow \infty$, it follows that identity \eqref{strongestderivative} must be valid for the $1$-form $u \in H^1_0(\mathbb{H}^2(-a^2))$ as desired.
\end{proof}
\begin{cor}[Equivalent norm on $H^1_0$]\label{equivH1}
Let  $u$ be a $1$-form.  If $u\in H^1_0(\H)$, then $\norm{\overline{\nabla}u}_{L^2}$
is finite
and $\norm{\cdot}_{L^2}+\norm{\overline{\nabla}\cdot}_{L^2}$ defines an equivalent norm on $H^1_0$, where $\overline{\nabla}$ is understood in a week sense analogous to Definition \eqref{DefofWeakDerivative} (using $\overline{\nabla}^\ast$ in place of $\Def^\ast$).
\end{cor}
\begin{proof}
The proof relies on \eqref{strongestderivative} and on the Weitzenb\"ock formula (see for example \cite{Taylor3})
\[
-\overline{\nabla}^\ast\overline{\nabla}=-\Delta + \Ric.
\]
\end{proof}

\subsection{Ladyzhenskaya inequality}
To prove the Ladyzhenskaya inequality we use
\begin{lemma}[Sobolev embedding on a hyperbolic space]\cite[Proposition 3.6, p. 54]{Hebey}  Let $f$ be a function on $\H$.  Then
\be\label{s11}
\norm{f}_{L^2(\H)}\leq C\large(\norm{f}_{L^1(\H)}+\norm{\grad f}_{L^1(\H)}\large).
\ee
\end{lemma}
\begin{lemma}[Ladyzhenskaya on a hyperbolic space] \label{LadyZhenskaya} Let $u$ be a $1$-form in  $H^1_0(\H)$, where $H^1_0(\H)$ is defined as in Definition \ref{H1def0}, then
\be\label{ldh}
\norm{u}_{4}\leq\sqrt{2 C} \norm{u}^\frac 12_{2}\norm{u}^\frac 12_{H^1},
\ee
where $C$ is the constant appearing in the Sobolev embedding \eqref{s11}.
\end{lemma}
\begin{proof}
Apply \eqref{s11} to the function $\abs{u}^2$, then
\begin{align*}
\norm{u}_4^2=\norm{\abs{u}^2}_2&\leq C \norm{\abs{u}^2}_1+C\int_{\H}\abs{\grad \abs{u}^2} \Vol_{\mathbb{H}^2(-a^2)}\\
&=C \large(\norm{u}_2^2+\int_{\H}\abs{\grad \abs{u}^2}\Vol_{\H}\large).
\end{align*}
By \eqref{ip2} and Cauchy-Schwarz
\[
\int_{\H}\abs{\grad \abs{u}^2}\Vol_{\H}\leq 2 \norm{u}_2\norm{\nablab u}_2.
\]
It follows
\[
\norm{u}_4^2\leq C \large(2\norm{u}_2\norm{\nablab u}_2+\norm{u}^2_2\large)\leq2C\norm{u}_2\norm{u}_{H^1}.
\]

\end{proof}
 \begin{remark}
 We note that in comparison to the Ladyzhenskaya inequality on $\R^2$ we have the nonhomogenous $H^1$ norm appear on the right-hand side instead of just $\dot H^1$.  
 \end{remark}

\section{The space of finite energy gradient flows}\label{MAINMAINMAIN}

Consider the following vector space 
\begin{equation}\label{GD}
\mathbb{F} = \{\dd F \in L^2(\mathbb{H}^2(-a^2)) : \triangle F = 0 \} .
\end{equation}

We start with a technical lemma.
\begin{lemma}\label{EstDissLEMMA}
Let $F$ be an arbitrary smooth function on $\mathbb{H}^2(-a^2)$. Then, the following estimate holds provided the right-hand side is finite
\begin{equation}\label{EstofDissipation}
\norm{ \overline{\nabla} (\nabla F )}_{L^2(\H)}^2 \leq 10{a^2}\int_{D_O(1)} \bigg \{ |\nabla^{\mathbb{R}^2}(F\circ Y^{-1})|^2
+ (1-|y|^2)^2 |\nabla^{\mathbb{R}^2} \nabla^{\mathbb{R}^2} (F\circ Y^{-1})|^2   \bigg \} \dd y_1 \dd y_2 .
\end{equation}
In the above expression, $D_O(1)$ is the unit Euclidean disc, and $Y : \mathbb{H}^2(-a^2) \rightarrow D_O(1)$ is the coordinate system given in \eqref{PoincareDisc}. The symbol $|\nabla^{\mathbb{R}^2} \nabla^{\mathbb{R}^2} f|^2$ means $|\nabla^{\mathbb{R}^2} \nabla^{\mathbb{R}^2} f|^2 = \sum_{1\leq i,j \leq 2} |\partial_{y_{i}} \partial_{y_{j}}f|^2$.
\end{lemma}
\begin{proof}
Let $F$ be a smooth function on $\mathbb{H}^2(-a^2)$. By a computation (see \eqref{gradF1}-\eqref{gradF2})
\begin{equation}
\nabla F = \frac{a^2}{4} (1-|Y|^2)^2 \bigg \{ \frac{\partial F}{\partial Y^1} \frac{\partial}{\partial Y^1} + \frac{\partial F}{\partial Y^2} \frac{\partial}{\partial Y^2} \bigg \} .
\end{equation}
Then we can decompose the tensor $\overline{\nabla} (\nabla F)$ into two (smooth) components as follows
\begin{equation}
\overline{\nabla} (\nabla F) =\frac {a^2}{4}( \textbf{I} + \textbf{II}) ,
\end{equation}
where
\begin{equation}
\begin{split}
\textbf{I} & = \sum_{ j,k} \frac{\partial }{\partial Y^j} \bigg \{ (1-|Y|^2)^2 \frac{\partial F }{\partial Y^k} \bigg \} dY^j\otimes \frac{\partial }{\partial Y^k} =\sum_{ j,k} \frac{\partial }{\partial Y^j} \bigg \{ (1-|Y|^2)^2 \frac{\partial F }{\partial Y^k} \bigg \} e_j^*\otimes e_k , \\
\textbf{II} & = (1-|Y|^2)^2 \sum_{1\leq j \leq 2} \frac{\partial F }{\partial Y^j} \overline{\nabla} \frac{\partial }{\partial Y^j} .
\end{split}
\end{equation}
Now, since $\{e_j^* \otimes e_k : 1 \leq j,k \leq 2\}$ is an orthonormal frame, with respect to the induced metric $\overline{g}(\cdot ,\cdot )$ on the tensor bundle $T^*(\mathbb{H}^2(-a^2)) \otimes T(\mathbb{H}^2(-a^2))$, it follows that (using $(a+b)^2\leq 2(a^2+b^2)$ and $\abs{Y^j}\leq 1$)
\begin{equation}\label{EstI}
\begin{split}
|\textbf{I}|^2 & =  \sum_{j,k} \bigg |  \frac{\partial }{\partial Y^j} \bigg \{ (1-|Y|^2)^2 \frac{\partial F }{\partial Y^k} \bigg \}  \bigg |^2 \\
& \leq \sum_{j,k} 2 \bigg \{ 16 \big (1-|Y|^2 \big )^2 \big |\frac{\partial F}{\partial Y^{k}} \big |^2 + \big ( 1-|Y|^2 \big )^4 \abs{ \frac{\partial }{\partial Y^{j}} \big ( \frac{\partial F}{\partial Y^{k}} \big ) }^2  \bigg \} \\
& \leq 64 \big (1-|Y|^2 \big )^2 \bigg \{  \sum_{ k } \big |\frac{\partial F}{\partial Y^{k}} \big |^2 +  \big ( 1-|Y|^2 \big )^2 \sum_{ j,k} \abs{ \frac{\partial }{\partial Y^{j}} \big ( \frac{\partial F}{\partial Y^{k}} \big ) }^2   \bigg \} .
\end{split}
\end{equation}
Next, by \eqref{simpleCalculation}
\begin{equation}\label{ESTII}
|\textbf{II}|^2  \leq \bigg \{ 2\cdot 2^{\frac{1}{2}} |Y| \big ( 1-|Y|^2 \big ) \sum_{k }\big |\frac{\partial F}{\partial Y^{k}} \big |  \bigg \}^2 \leq 16 \big ( 1-|Y|^2 \big )^2 \sum_{k }\big |\frac{\partial F}{\partial Y^{k}} \big |^2 .
\end{equation}
By combining \eqref{EstI} with \eqref{ESTII}, we yield
\begin{equation}
\begin{split}
|\overline{\nabla} \nabla F|^2 & \leq \frac{a^4}{8} \big ( |\textbf{I}|^2 + |\textbf{II}|^2  \big ) \\
& \leq 10a^4 \big ( 1-|Y|^2 \big )^2 \bigg \{ \sum_{ k }\big |\frac{\partial F}{\partial Y^{k}} \big |^2 +  \big ( 1-|Y|^2 \big )^2   \sum_{j, k }\big |\frac{\partial}{\partial Y^j} \big ( \frac{\partial F}{\partial Y^{k}} \big ) \big |^2 \bigg \},
\end{split}
\end{equation}
from which it follows at once that the following estimate holds, as needed.
\begin{equation*}
\begin{split}
&\int_{\mathbb{H}^2(-a^2)}|\overline{\nabla} \nabla F|^2 \Vol_{\mathbb{H}^2(-a^2)}  = \int_{\mathbb{H}^2(-a^2)}|\overline{\nabla} \nabla F|^2 \cdot
\frac{4}{a^2 \big ( 1 - |Y|^2 \big )^2} \dd Y^1 \wedge \dd Y^2 \\
&\qquad\qquad \leq  10a^2\int_{D_O(1)} \bigg \{ |\nabla^{\mathbb{R}^2}(F\circ Y^{-1})|^2
+ (1-|y|^2)^2 |\nabla^{\mathbb{R}^2} \nabla^{\mathbb{R}^2} (F\circ Y^{-1})|^2   \bigg \} \dd y_1 \dd y_2 .
\end{split}
\end{equation*}
\end{proof}

By combining Lemma \ref{EstDissLEMMA} with Lemma \ref{UsefulLEMMAONTHEDISC} , we immediately get the following important lemma.

\begin{lemma}[Finite energy of $\dd F\in \mathbb{F}$ implies finite dissipation]\label{IMFinitedisspation}
Let $F$ be a harmonic function on $\mathbb{H}^2(-a^2)$ which satisfies the condition that
\begin{equation}
\|\dd F\|_{L^2(\mathbb{H}^2(-a^2))} =  \|\nabla F\|_{L^2(\mathbb{H}^2(-a^2))} < \infty .
\end{equation}
Then, it follows that there exists some absolute constant $C_a > 0$, such that
\begin{equation}\label{WellONE}
\int_{\mathbb{H}^2(-a^2)}|\overline{\nabla} \nabla F|^2 \Vol_{\mathbb{H}^2(-a^2)} \leq C_a \|d F\|_{L^2(\mathbb{H}^2(-a^2))}^2 .
\end{equation}
\end{lemma}
\begin{proof}
The proof is straightforward. Indeed, by combining inequalities \eqref{EstofDissipation} and \eqref{FinalGoryPride}, we get
\begin{equation*}
\begin{split}
\norm{\overline{\nabla}\nabla F}^2_{L^2(\H)}
& \leq  10a^2\int_{D_O(1)} \bigg \{ |\nabla^{\mathbb{R}^2}(F\circ Y^{-1})|^2
+ (1-|y|^2)^2 |\nabla^{\mathbb{R}^2} \nabla^{\mathbb{R}^2} (F\circ Y^{-1})|^2   \bigg \} \dd y_1 \dd y_2 \\
& \leq 20a^2(1+2\tilde C_0) \int_{D_O(1)} \big | \nabla^{\mathbb{R}^2} (F\circ Y^{-1})\big |^2 \\
& = 20a^2 \big \|\dd F  \big \|_{L^2(\mathbb{H}^2(-a^2))}^2 ,
\end{split}
\end{equation*}
where the last equality follows from identity \eqref{writtenonDisc}.\end{proof}

Next, we give the following easy fact about the space $\mathbb{V} = \textbf{V}\oplus \mathbb{F}$. Recall that the spaces $\textbf{V}$ and $\mathbb{F}$ are defined as follows
\begin{itemize}
\item $\textbf{V} = \overline{\Lambda_{c , \sigma }^1(\mathbb{H}^2(-a^2))}^{H^1}$ with $\Lambda_{c,\sigma}^1(\mathbb{H}^2(-a^2))$ to be the space of smooth $d^*$-closed $1$-forms with compact support on $\mathbb{H}^2(-a^2)$.
\item $\mathbb{F} = \{ \dd F \in L^2(\mathbb{H}^2(-a^2)) : \triangle F = 0 \}$.
\end{itemize}

\begin{lemma}\label{Anothertrivial}
The relation $\mathbb{V} \subset H^1_0(\mathbb{H}^2(-a^2))$ holds.
\end{lemma}
\begin{proof}
First, it follows directly from the definition of $\textbf{V}$ that $\textbf{V} \subset H^1_0(\mathbb{H}^2(-a^2))$. So, all we need to prove is that $\mathbb{F} \subset H^1_0(\mathbb{H}^2(-a^2))$.

Let $F$ be a harmonic function on $\mathbb{H}^2(-a^2)$ such that $\dd F \in L^2(\mathbb{H}^2(-a^2))$. Then, with respect to an arbitrarily chosen positive number $R > 1$, we consider some radially symmetric (with respect to some preferred point of reference $O$ in $\mathbb{H}^2(-a^2)$) test function  $\eta_R \in C_{c}^{\infty}(\mathbb{H}^2(-a^2))$ which satisfies the following two conditions
\begin{itemize}
\item $\chi_{B_O(R)} \leq \eta_R \leq \chi_{B_O(2R)}$. Here, $B_O(r)$ is the geodesic ball with radius $r$ centered at $O$.
\item $|\nabla \eta_R| \leq \frac{2}{R}$.
\end{itemize}
 Since
\begin{align*}
\overline{\nabla}{fu}=\dd F\otimes u + f\overline{\nabla}u,\quad
\overline{\nabla}^Tfu=u\otimes \dd F + f\overline{\nabla}^Tu,
\end{align*}
and $\Def=\frac 12(\overline\nabla+\overline\nabla^T)$, we have
\begin{equation}\label{productrule}
\Def (\eta_R \dd F ) = \frac{1}{2} \big \{ \dd \eta_R \otimes \dd F + \dd F \otimes \dd \eta_R \big \} + \eta_R \Def (\dd F).
\end{equation}
Using $(a+b+c)^2\leq 3(a^2+b^2+c^2)$, we then carry out the following pointwise estimate, in which we temporarily use the abbreviation $|\theta |_{\overline{g}} = \overline{g}(\theta , \theta )^{\frac{1}{2}}$, with $\theta$ to be a smooth section of $T^*(\mathbb{H}^2(-a^2))\otimes T^*(\mathbb{H}^2(-a^2))$
\begin{equation*}
\begin{split}
|\Def \big ( \eta_R \dd F - \dd F    \big )|_{\overline{g}}^2 & \leq 3 \bigg \{ \frac{1}{4} |\dd \eta_R \otimes \dd F|_{\overline{g}}^2  +  \frac{1}{4} |\dd F \otimes \dd \eta_R |_{\overline{g}}^2   + (1-\eta_R)^2 |\Def (\dd F )|_{\overline{g}}^2                    \bigg \} \\
& \leq 3\bigg \{  \frac{2}{R^2} |\dd F|^2 + \chi_{\mathbb{H}^2(-a^2)- B_O(R)} |\Def (\dd F )|_{\overline{g}}^2 \bigg \} .\label{defest1}
\end{split}
\end{equation*}
Then
\begin{equation*}
\begin{split}
\|\eta_R \dd F - \dd F\|_{H^1(\mathbb{H}^2(-a^2))} & = \int_{\mathbb{H}^2(-a^2)} |\eta_R \dd F - \dd F|^2 \Vol_{\mathbb{H}^2(-a^2)} \\
&\quad+ 2 \int_{\mathbb{H}^2(-a^2)} |\Def \big ( \eta_R \dd F - \dd F \big ) |_{\overline{g}}^2 \Vol_{\mathbb{H}^2(-a^2)} \\
& \leq \int_{\mathbb{H}^2(-a^2)- B_O(R)} |\dd F|^2 \Vol_{\mathbb{H}^2(-a^2)} + \frac{12}{R^2} \int_{\mathbb{H}^2(-a^2)} |\dd F|^2 \Vol_{\mathbb{H}^2(-a^2)} \\ &\quad + 6 \int_{\mathbb{H}^2(-a^2)-B_O(R)} |\Def (\dd F )|_{\overline{g}}^2\Vol_{\H} .
\end{split}
\end{equation*}
Since by Lemma \ref{IMFinitedisspation}, $\dd F \in \mathbb{F}$ implies $\dd F \in H^1(\mathbb{H}^2(-a^2))$, it follows that each term on the right-hand side of the above estimate will tend to $0$ as $R\rightarrow \infty$. So, in particular, we know that the sequence $\{\eta_{m} \dd F\}_{m=1}^{\infty} \subset \Lambda_{c}^1(\mathbb{H}^2(-a^2))$ will converge to $\dd F$ in the $H^1$-norm. This shows that $\mathbb{F} \subset H^1_0(\mathbb{H}^2(-a^2))$.
\end{proof}

 We can now deduce the following

\begin{lemma}\label{SILEMMA}
For any harmonic function $F$ on $\mathbb{H}^2(-a^2)$ with $\dd F \in L^2(\mathbb{H}^2(-a^2))$, it follows that $\Def \dd F \in L^2(\mathbb{H}^2(-a^2))$ and that the following identity holds.
\begin{equation}\label{SIMPLEIMPORTANT}
((\dd F , \dd F )) = 2a^2 \|\dd F\|_{L^2(\mathbb{H}^2(-a^2))} .
\end{equation}
Here, we recall that for a $1$-form $u \in H^1(\mathbb{H}^2(-a^2))$, $((u,u))^{\frac{1}{2}}$ is the homogeneous part of the $H^1$-norm of $u$ as specified in \eqref{homogeneouspart}of Definition \ref{H1def}.
\end{lemma}

\begin{proof}
Let $F$ be harmonic on $\mathbb{H}^2(-a^2)$ and $\dd F \in L^2(\mathbb{H}^2(-a^2))$.

From Lemma \ref{Anothertrivial}, $\mathbb{F} \subset H^1_0(\mathbb{H}^2(-a^2))$.  Hence $\dd F \in H^1_0(\mathbb{H}^2(-a^2))$, and we can  invoke Lemma \ref{TrivialLemma} to deduce
\begin{equation}\label{easyeasyeasy}
((\dd F , \dd F )) = 2 \|\dd^* \dd F \|_{L^2(\mathbb{H}^2(-a^2))}^2 + \|\dd \dd F \|_{L^2(\mathbb{H}^2(-a^2))}^2 + 2a^2 \|\dd F\|_{L^2(\mathbb{H}^2(-a^2))}^2,
\end{equation}
and \eqref{SIMPLEIMPORTANT} follows since $\dd\dd F = 0$ and $\dd^*\dd F = (-\triangle )F = 0 $.
\end{proof}

We are now ready to prove the following simple lemma, which uses elementary techniques from elliptic theory.  We include the proof only for completeness.

\begin{lemma}\label{Gradientspacecomplete}
The vector space $\mathbb{F}$ as defined in \eqref{GD} is a Banach space with respect to the Sobolev norm $\|\cdot \|_{H^1(\mathbb{H}^2(-a^2))}$.
\end{lemma}


\begin{proof}
It suffices to show $\F$ is a closed vector space in $H^1(\mathbb{H}^2(-a^2))$.
Consider a sequence $\{\dd F_m\}_{m=1}^{\infty}$ in $\mathbb{F}$, with each $F_{m}$ to be a harmonic function on $\mathbb{H}^2(-a^2)$, such that $\{\dd F_m\}_{m=1}^{\infty}$ is Cauchy with respect to $\|\cdot \|_{H^1(\mathbb{H}^2(-a^2))}$, so

\begin{equation}\label{goal1}
\lim_{m,n \rightarrow \infty} \bigg \{ \|\dd F_m - \dd F_n\|_{L^2(\mathbb{H}^2(-a^2))}^2 + (( \dd F_m - \dd F_n , \dd F_m - \dd F_n )) \bigg \} = 0 .
\end{equation}
Our task is to prove that there exists some harmonic function $F^*$ on $\mathbb{H}^2(-a^2)$ which satisfies $\dd F^* \in H^1(\mathbb{H}^2(-a^2))$, and for which $\{\dd F_m\}_{m=1}^{\infty}$ converges to $\dd F^*$ with respect to the norm $\|\cdot \|_{H^1(\mathbb{H}^2(-a^2))}$.
To this end, we consider, for each $m \in \mathbb{Z}^+$, the function $f_m \in C^{\infty}(D_0(1))$ defined by
\begin{equation}
f_m = F_m \circ Y^{-1} ,
\end{equation}
where $Y^{-1} : D_0(1) \rightarrow \mathbb{H}^2(-a^2)$ is the smooth inverse of the coordinate system $Y$ on $\mathbb{H}^2(-a^2)$, which we introduce in \eqref{PoincareDisc}. Since  $F_m$ is harmonic on $\mathbb{H}^2(-a^2)$, we have (see \eqref{LaplaceManifold})
\[
\triangle_{\mathbb{R}^2}f_m = 0 \quad\mbox{on the Euclidean disk}\quad D_0(1),\quad \forall m \in \mathbb{Z}^+.
\]
Now, consider the sequence $\{\overline{f}_m\}$ of harmonic functions defined by
\begin{equation}
\overline{f}_m = f_m - \frac{1}{|D_0(1)|} \int_{D_0(1)} f_m(y) \dd y .
\end{equation}
Then, by the Poincare's inequality and \eqref{writtenonDisc}
\begin{equation}\label{easypoincareest}
\|\overline{f}_m - \overline{f}_n\|_{L^2(D_0(1))} \leq C_0 \| \nabla^{\mathbb{R}^2} ( f_m - f_n  )   \|_{L^2(D_0(1))}
= C_0 \| \dd f_m - \dd f_n \|_{L^2(\mathbb{H}^2(-a^2))} .
\end{equation}
Now, since $\overline{f}_m - \overline{f}_n$ is a harmonic function on the Euclidean disc $D_0(1)$, for any $r \in (0,1)$, we can apply the mean-value formula to deduce that the following pointwise estimate holds for any point $y \in D_0(r) = \{y : |y| < r\}$.

\begin{equation}\label{mvp}
\begin{split}
|\overline{f}_m(y) - \overline{f}_n(y)| & \leq \frac{1}{\pi (1-r)^2}  \int_{|z-y| < 1-r}  |\overline{f}_m(z) - \overline{f}_n(z)| dz \\
& \leq \frac{1}{\pi^{\frac{1}{2}}(1-r)} \|\overline{f}_m - \overline{f}_n\|_{L^2(D_0(1))} .
\end{split}
\end{equation}

Combining \eqref{mvp} with \eqref{easypoincareest}, we have
\begin{equation}\label{locallyunformly}
\|\overline{f}_m - \overline{f}_n\|_{L^{\infty}(D_0(r))} \leq \frac{C_0}{(1-r)} \| \dd f_m - \dd f_n \|_{L^2(\mathbb{H}^2(-a^2))},\quad \forall m,n \in \mathbb{Z}^+, \quad r \in (0,1).
\end{equation}
Then estimate \eqref{locallyunformly} ensures that the sequence $\{\overline{f}_m \}_{m=1}^{\infty}$ converges locally uniformly to a limiting harmonic function
$f^*$ on $D_0(1)$ (see for example \cite[Theorem 2.8, p. 21]{GilbargTrudinger}).

Recall that we have
\begin{equation}
\lim_{m,n\rightarrow \infty }\|\nabla^{\mathbb{R}^2}\overline{f}_m -\nabla^{\mathbb{R}^2}\overline{f}_n \|_{L^2(D_0(1))} =
\lim_{m,n\rightarrow \infty }\|\dd F_m - \dd F_n\|_{L^2(\mathbb{H}^2(-a^2))}^2 = 0.
\end{equation}
Hence, there must exists some $\mathbb{R}^2$-valued function $w^* \in L^{2}(D_0(1))$ to which the sequence $\{ \nabla_{\mathbb{R}^2}\overline{f}_m \}_{m=1}^{\infty}$ converges in $L^2(D_0(1))$. That is, we have
\begin{equation}
\lim_{m\rightarrow \infty} \|\nabla^{\mathbb{R}^2}\overline{f}_m  - w^* \|_{L^{2}(D_0(1))} = 0.
\end{equation}
Since by  \eqref{easypoincareest} $\{\overline{f}_m\}_{m=1}^{\infty}$ also converges to the limiting function $f^*$ in $L^2(D_0(1))$, it follows that $\nabla^{\mathbb{R}^2}f^* = w^* \in L^2(D_0(1))$. Hence, the function $F^* \in C^{\infty}(\mathbb{H}^2(-a^2))$ defined by $F^* = f^*\circ Y$ satisfies the following properties:
\begin{itemize}
\item $F^*$ is a harmonic function on $\mathbb{H}^2(-a^2)$.
\item $\dd F^* \in L^2(\mathbb{H}^2(-a^2))$ since $\|\dd F^*\|_{L^2(\mathbb{H}^2(-a^2))} = \|\nabla^{\mathbb{R}^2}f^*  \|_{L^2(D_0(1))} < \infty $.
\item $\|\dd F_m - \dd F^*\|_{L^2(\mathbb{H}^2(-a^2))} = \|\nabla^{\mathbb{R}^2}f_m-\nabla^{\mathbb{R}^2}f^*\|$ converges to $0$ as $m\rightarrow \infty$.
\end{itemize}
Finally by Lemma \ref{SILEMMA}, we deduce that
\begin{equation}
\lim_{m\rightarrow \infty }(( \dd F_m - \dd F^* , \dd F_m - \dd F^* )) =2 a^2 \lim_{m\rightarrow \infty}\|\dd F_m - \dd F^*\|_{L^2(\mathbb{H}^2(-a^2))}^2 = 0.
\end{equation}
Hence, \eqref{goal1} holds as desired.
\end{proof}

\begin{remark}
Actually, the proof of Theorem \ref{Gradientspacecomplete} shows, as a by-product, that $\mathbb{F}$ is also a closed subspace in $L^2(\mathbb{H}^2(-a^2))$ with respect to the norm $\|\cdot \|_{L^2}$.
\end{remark}

\begin{lemma}\label{orthogonal}
The vector subspaces $\textbf{V}$ and $\mathbb{F}$ are orthogonal to each other in $H^1_0(\mathbb{H}^2(-a^2))$ with respect to the inner product $[\cdot, \cdot ]$.
\end{lemma}
\begin{proof}
Recall that $\Lambda_{c,\sigma}^1(\mathbb{H}^2(-a^2))$ is the space of all smooth $1$-forms $\theta$ with compact support in $\mathbb{H}^2(-a^2)$ which satisfy $\dd^*\theta  = 0$, and
\begin{equation*}\label{closure}
\textbf{V} = \overline{\Lambda_{c,\sigma}^1(\mathbb{H}^2(-a^2))}^{\|\cdot \|_{H^1}} .
\end{equation*}
To begin, we consider some element $v \in \Lambda_{c,\sigma}^1(\mathbb{H}^2(-a^2))$, and some $\dd F \in \mathbb{F}$, with $F$ to be harmonic on $\mathbb{H}^2(-a^2)$. Then by \eqref{defnDefstarDef} and \eqref{ip4}
\begin{equation}
\begin{split}
((v, \dd F )) & = 2 \int_{\mathbb{H}^2(-a^2)} \bar g ( \Def v , \Def \dd F  ) \\
& = \int_{\mathbb{H}^2(-a^2)} g (v, 2 \Def^* \Def \dd F ) \\
& = 2a^2 \int_{\mathbb{H}^2(-a^2)} g (v , \dd F) \\
& = -2a^2 \int_{\mathbb{H}^2(-a^2)} \dd^* (F v) = 0,
\end{split}
\end{equation}
where the last equality follows, since $Fv$ is a smooth $1$-form with compact support on $\mathbb{H}^2(-a^2)$.

This shows that $$[v,\dd F] = (v,\dd F)+((v, \dd F))=0$$  for any $v \in \Lambda_{c,\sigma}^1(\mathbb{H}^2(-a^2))$, and any $\dd F \in \mathbb{F}$.

To deal with the general case of $v \in \textbf{V}$, we choose a sequence $\{v_k\}_{k=1}^{\infty} \subset \Lambda_{c,\sigma}^1(\mathbb{H}^2(-a^2))$ which converges to $v$ with respect to the Sobolev norm $\|\cdot\|_{H^1(\mathbb{H}^2(-a^2))}$. Then
\begin{equation}
\big | [v, \dd F ]\big | = \big | [v-v_k , \dd F ]\big | \leq \|v - v_k \|_{H^1(\mathbb{H}^2(-a^2))}\cdot \|\dd F\|_{H^1(\mathbb{H}^2(-a^2))} .
\end{equation}
So, by passing to the limit as $k\rightarrow \infty$, the above estimate gives $[v,\dd F] = 0$, as desired.
\end{proof}

\section{Operators $A$ and $B$}\label{AB} Following \cite{Temam}, we rewrite $((u,v))$  in terms of an operator
 \[
 A: \V \rightarrow \V',
 \]
with $A$ given by
\be\label{Adefn}
A(u)(v)=\ip{Au,v}_{\V'\otimes  \V}:=((u,v)), \quad v \in \V.
\ee
We have this simple lemma.
\begin{lemma}[Estimates on $A$]\label{AONE}
 \begin{align}
\norm{ Au}_{L^2(0,T;\V')}&\leq \norm{ u}_{L^2(0,T;\V)},\label{A1}\\
\norm{ Au}_{L^2(0,T;H^{-1})}&\leq \norm{ u}_{L^2(0,T;\V)}.\label{A2}
 \end{align}
\end{lemma}
\begin{proof}
For \eqref{A1}, by definition and Cauchy-Schwarz
\begin{align*}
\norm{ Au}^2_{L^2(0,T;\V')}&=\int_0^T\norm{Au(t)}_{\V'}^2\dd t\\
&=\int_0^T\sup_{\norm{v}_\V=1}\abs{((u(t),v))}^2\dd t\\
&\leq \int_0^T\sup_{\norm{v}_\V=1} \norm{\Def u(t)}^2_{L^2} \norm{\Def v}^2_{L^2}\dd t\\
&\leq \int_0^T \norm{\Def u(t)}^2_{L^2} \dd t\\
&\leq \norm{ u}^2_{L^2(0,T;\V)}.
\end{align*}
Proof of  \eqref{A2} is identical.

\end{proof}
Next we introduce the notation for the nonlinear term.  Let
\[
b(u,v,w)=\int_{\H}  g(\nablab_u v, w)\Vol_{\H}
\]
and
\[
B: \V \rightarrow \V'
\]
so that
 $B(u) \in \V'$ and is defined by
\be\label{Bdefn}
B(u)(v)=\ip{Bu,v}_{\V'\otimes  \V}:=b(u,u,v),\quad v \in \V.
\ee

We now have $\H$ version of \cite[Lemma 3.4 p.198]{Temam} together with some other properties.  The fact that we are considering the space
$\V = \textbf{V} \oplus \mathbb{F}$ instead of just $\textbf{V}$ requires some additional care.
\begin{lemma}[Properties of $b$ and $B$]  \label{ATWO} We have
\begin{align}
b(u,v,v)&=0  \quad \mbox{if} \quad u \in \V, v \in H^1_0,\label{bp2}\\
b(u,v,w)&=-b(u,w,v) \quad \mbox{if} \quad u \in \V, v, w \in H^1_0,  \label{bp1}\\
 \abs{b(u,v,w)}&\leq C\norm{u}^\frac 12_{2}\norm{u}^\frac 12_{H^1}\norm{v}_{\dot H^1}\norm{w}^\frac 12_{2}\norm{w}^\frac 12_{H^1}, \label{b1}
 \end{align}
 and
 \begin{align}
\norm{ Bu}_{L^2(0,T;\V')}&\leq C\norm{u}_{L^\infty(0,T;H)}\norm{ u}_{L^2(0,T;\V)}, \label{b2}\\
\norm{ Bu}_{L^2(0,T;H^{-1})}&\leq C\norm{u}_{L^\infty(0,T;H)}\norm{ u}_{L^2(0,T;\V)},\label{b4}\\
\norm{ Bu}_{L^1(0,T;\V')}&\leq C\norm{ u}^{2}_{L^2(0,T;\V)}.\label{b3}
 \end{align}

\end{lemma}
\begin{proof}  If $u\in \V$ and $v\in H^1_0$, then $\dd^\ast  u=0$ and by \eqref{ip0},
\[
b(u,v,v)=\int_{\H}\left(-\frac 12\dd^\ast  (\abs{v}^2u)+ \frac 12\abs{v}^2 \dd^\ast u \right)\Vol_{\H}=0,
\]
where the first term is $0$ because $v \in H^1_0$, so it is a limit of elements in $\Lambda^1_{c}(\H)$.  This shows \eqref{bp2}.

  \eqref{bp1} follows from \eqref{bp2} and bilinearity by replacing $v$ by $v+w\in H^1_0$ since
  \[
  0=b(u, v+w, v+w)=b(u,v,v)+b(u,v,w)+b(u,w,v)+b(u,w,w),
  \]
  and the first term and the last are $0$ by \eqref{bp2}.

To show \eqref{b1}, we use \eqref{csg} and H\"older
\[
 \abs{b(u,v,w)}\leq \int_M \abs{u}\abs{\nablab v}\abs{w}\leq \norm{u}_4\norm{w}_4\norm{\nablab v}_2
\]
and the result follows by \eqref{ldh}.  Next
\be\label{bp3}
\norm{ Bu}^2_{L^2(0,T;\V')}=\int_0^T\sup_{\norm{v}_\V=1}\abs{b\big(u(t),u(t),v\big)}^2\dd t,
\ee
so for $v\in \V$ with $\norm{v}_\V=1$ consider
 \be\label{bp35}
\abs{b(u,u,v)}=\abs{-b(u,v,u)} \leq \norm{u}^2_4\norm{\nablab v}_2\leq  \norm{u}^2_4
 \ee
by  \eqref{csg} and H\"older.  Hence by \eqref{ldh}
\be\label{bp4}
\abs{b(u,u,v)}\leq  C \norm{u}_2\norm{u}_{H^1}.
\ee
Then \eqref{bp3} and \eqref{bp4} imply

\[
\norm{ Bu}^2_{L^2(0,T;\V')}\leq C^2 \int_0^T \norm{u(t)}^2_2\norm{u(t)}^2_{H^1}\dd t,
\]
which immediately gives \eqref{b2}.  The proof of \eqref{b4} is identical.
 Finally, for \eqref{b3} we have using \eqref{bp4}
\begin{align*}
\norm{ Bu}_{L^1(0,T;\V')}=\int_0^T\sup_{\norm{v}_\V=1}\abs{b\big(u(t),u(t),v\big)}\dd t\lesssim \int_0^T  \norm{u(t)}_2\norm{u(t)}_{H^1}\dd t \lesssim\int_0^T \norm{u(t)}^{2}_{H^1} \dd t
\end{align*}
as needed.  \end{proof}

\section{Existence}\label{exist}

In the weak formulation of Leary-Hopf solutions to the Navier-Stokes equation, we consider
\begin{equation}
\mathbb{V} = \textbf{V} \oplus \F\subset H^1_0(\H) ,
\end{equation}
with $\textbf{V}$ and $\mathbb{F}$ as defined in Section \ref{summary}. Recall by Lemma \ref{orthogonal}, $\tV$ and $\F$ are orthogonal with respect to the following inner product on $\V$
\begin{equation}
[ \cdot , \cdot  ] = (\cdot , \cdot )+ (( \cdot , \cdot  )) .
\end{equation}

Next, we need a sequence $\{w_k\}_{k=1}^{\infty}$ lying in $\Lambda_{c,\sigma}^1(\mathbb{H}^2(-a^2)),$
which can serve as a basis of $\textbf{V}$. $H^1_0(\mathbb{H}^2(-a^2))$ is separable, so $\textbf{V}$, as a closed subspace of $H^1_0(\mathbb{H}^2(-a^2))$ must always have a basis $\{e_k\}_{k=1}^{\infty}$. The question here is how to find a basis of elements for $\textbf{V}$ which lie entirely  in the dense subspace $\Lambda^1_{c,\sigma}(\H)$. Indeed, an affirmative answer to this question follows directly from elementary Lemma \ref{peaceofmindlemma} from functional analysis, which we show in the appendix.

Thanks to Lemma \ref{peaceofmindlemma}, we can choose a sequence $\{e_k\}_{k=1}^{\infty}$ in $\Lambda_{c,\sigma}^1(\mathbb{H}^2(-a^2))$ which is a basis of $\textbf{V}$. On the other hand, $\mathbb{F}$ is an infinite-dimensional separable Banach space with respect to the Sobolev norm $\|\cdot \|_{H^1}$. So, we can also find a basis $\{\dd F_{k}\}_{k=1}^{\infty}$ of $\mathbb{F}$. Now, we consider the sequence of elements $\{w_k\}_{k=1}^{\infty}$ in $\mathbb{V}$ which is defined through the following rules:

\begin{itemize}
\item $w_{2k-1} = e_k$, for any $k \in \mathbb{Z}^+$.
\item $w_{2k} = \dd F_k$, for any $k \in \mathbb{Z}^+$.
\end{itemize}

Then, by Lemma \ref{orthogonal}, $\{w_k\}_{k=1}^{\infty}$ becomes a basis for $\mathbb{V} = \textbf{V} \oplus \mathbb{F}$. So, if we let $\textbf{E}_m = \Span \{w_1, w_2 , .... ,w_m\}$, then
\begin{equation}\label{totalclosure}
\overline{\bigcup_{m\in \mathbb{Z}^+} \textbf{E}_m}^{H^1} = \mathbb{V} .
\end{equation}

We now follow the main steps of the Faedo-Galerkin method as given in \cite[pp. 192-196]{Temam}.  Including space $\F$ requires additional work even in the setup above.  Also see subcase 2 in Step 6 below.\\

\noindent{\bf Step 1: Selection of the finite energy initial data.}\\
Recall
\begin{equation*}
\textbf{H} = \overline{\Lambda_{c,\sigma}^1(\mathbb{H}^2(-a^2))}^{L^2}.
\end{equation*}
Let $u_{0}\in \textbf{H}\oplus \mathbb{F}$ be the initial data.   Since by definition $\textbf{V} \subset \textbf{H}$, the vector subspace $\textbf{E}_m$ of $\textbf{V}\oplus \mathbb{F}$ can be viewed as a vector subspace of $\textbf{H}\oplus \mathbb{F}$.   Now define $$u_{0m} \in \textbf{E}_m$$ to be the orthogonal projection of $u_0$ onto $\textbf{E}_m$ with respect to the inner product $(\cdot ,\cdot ) $.\\

\noindent{\bf Step 2: Constructing approximate solutions in subspaces $\textbf{E}_m$. }\\

Let $T > 0$ be given. We seek a $C^1$-function $u_m : [0,T] \rightarrow \textbf{E}_{m}$ in the following form
\begin{equation}\label{theformofUm}
u_m(t) = \sum_{j=1}^{m} g_{jm}(t) w_j ,
\end{equation}
with $u_m(0) = u_{0m}$, and which also satisfies the following equation for $j \in \{1,2,...,m\}$
\begin{equation}\label{finiteApproxeq}
(u_m'(t), w_j) +  (( u_m(t) , w_j )) + b(u_m(t) , u_m(t), w_j) = 0 .
\end{equation}
Notice that for each $1\leq j \leq m$, equation \eqref{finiteApproxeq} can be expressed in the following equivalent form.
\begin{equation}\label{finiteApproxTWO}
\sum_{k=1}^{m} (w_k , w_j)g_{km}'(t) + \sum_{k=1}^{m} ((w_k ,w_j )) g_{km}(t) + \sum_{1\leq k,l \leq m} b(w_k,w_l, w_j) g_{km}(t) g_{lm}(t) = 0 .
\end{equation}
The linear independence of the list $\{w_1 , w_2 , ... w_m\}$ ensures that the $m\times  m$-matrix with $(w_i,w_j)$ as its $(i,j)$-entry, is invertible. So, let $(\alpha^{ij} )$ be the inverse matrix. By multiplying   \eqref{finiteApproxTWO} by $\alpha^{ji}$, and then summing with respect to $j$, we yield the following system of equations for $1\leq i \leq m$
\begin{equation}\label{systemofODE}
g_{im}'(t) + \sum_{1\leq j,k \leq m} g_{km}(t) ((w_k , w_j))\alpha^{ji} + \sum_{1\leq j,k,l \leq m} b(w_k,w_l,w_j)g_{km}(t)g_{lm}(t)\alpha^{ji} = 0.
\end{equation}

The basic existence and uniqueness theorem of ODE ensures the existence of time $T_m>0$ and a $C^1$-solution $g_{im}$ on $[0,T_m)$.  Then we define $u_m : [0,T_m) \rightarrow \textbf{E}_m$, by \eqref{theformofUm}, and now $u_m$ is a solution to the system \eqref{systemofODE} and satisfies $u_m(0) = u_{0m}$. Since the system \eqref{systemofODE} is equivalent to the system \eqref{finiteApproxeq}, $u_m$ is also a solution to  \eqref{finiteApproxeq} with the initial data $u_m(0) = u_{0m}$.

Following the reasoning in \cite{Temam}, we show the time of existence $T_m$ of the solution $u_m$ can actually coincide with $T$ through obtaining a uniform bound for the magnitude of the function $\big ( \sum_{j=1}^{m} |g_{im}|^2 \big )^{\frac{1}{2}}$ over $[0,T_m)$.

Indeed, the required uniform estimate for $\big ( \sum_{j=1}^{m} |g_{im}|^2 \big )^{\frac{1}{2}}$ over $[0,T_m)$ is obtained as a by-product of the following a priori estimate for $\|u_m(t)\|_{L^2}$.\\

\noindent{\bf Step 3: a priori estimate for the energy of $u_m$.}\\

As in \cite{Temam}, we deduce from the system \eqref{finiteApproxeq} the following relation
\begin{equation}
(u_m'(t) , u_m(t))  +  (( u_m(t) , u_m(t) )) + b(u_m(t) , u_m(t) , u_m(t) ) = 0,\quad t \in [0,T_m).
\end{equation}
which using \eqref{bp2} further reduces down to
\begin{equation}\label{relation2}
\frac{\dd}{\dd t} \|u_m(t)\|_{L^2}^2 +  2(( u_m(t) , u_m(t) )) = 0.
\end{equation}
By integrating \eqref{relation2} over $[0,t]$  we deduce  \begin{equation}\label{energym}
\sup_{t \in [0,T_m)} \|u_m(t)\|_{L^2}^2 + 2\int^{T_m}_0  (( u_m(s) , u_m(s) )) \dd s = \|u_{0m}\|_{L^2}^2 \leq \|u_0\|_{L^2}^2,
\end{equation}
where the last inequality follows from the fact that $u_{0m} \in \textbf{E}_m$ is the orthogonal projection of $u_0$ onto $\textbf{E}_m$.\\

Now, consider the \emph{new} norm $|\cdot |_{\textbf{E}_m}$ on the $m$-dimensional vector space $\textbf{E}_m$, which is defined by the following rule, for any element $v = \sum_{j=1}^{m} \lambda_j w_j$, with $\lambda_j \in \mathbb{R}$
\begin{equation}
| v |_{\textbf{E}_m} = \{ \sum_{j=1}^{m} (\lambda_j)^2 \}^{\frac{1}{2}} .
\end{equation}
Since all norms on a finite dimensional vector space are equivalent, there exists some constant $\Lambda_m > 1$ such that  \begin{equation}\label{compare}
  | v |_{\textbf{E}_m} \leq \Lambda_m  \|v \|_{L^2} .
\end{equation}
Hence from \eqref{compare} and \eqref{energym}
\begin{equation}
\sup_{t \in [0,T_m)} \big \{ \sum_{j=1}^{m} |g_{jm}|^2 \big \}^{\frac{1}{2}} = \sup_{t \in [0,T_m)} |u_m(t)|_{\textbf{E}_m} \leq \Lambda_m \|u_0\|_{L^2},
\end{equation}
and the standard ODE theory implies $T_m$ can be extended to be as large as the prescribed $T > 0$.

 \noindent{\bf Step 4: Showing $u_{m}$ is uniformly bounded in $H^{\gamma}(\R; \V, \textbf{H}\oplus \F)$. }\\
 Let $\gamma \in (0,\frac 14)$, and
 \[
 H^{\gamma}(\R, \V, \textbf{H}\oplus \F):=\set{v \in L^2(\R; \V), D^\gamma_tv\in L^2(\R; \textbf{H}\oplus \F)} ,
 \]
 where
 \[
\widehat{ D^\gamma_tv}(\tau)=(2\pi\tau)^\gamma \hat v(\tau).
 \]
 Then
 \[
 \norm{v}_{H^\gamma}^2=\norm{v}^2_{L^2(\R; \V)}+\norm{\abs{\tau}^\gamma \hat v}_{L^2(\R;\textbf{H}\oplus \F)}^2.
 \]
From the definition of the norm of $H^{\gamma}$ and \eqref{energym} we just need to show $D^{\gamma}_{t}u_{m}$ is bounded in $L^{2}(\R; L^{2})$ (note the uniform bound will depend on $T$). To that end, let $\tilde u_{m}(t)=u_{m}(t), t\in [0,T]$, and $0$  otherwise.  This definition potentially creates jump discontinuities that lead to delta functions in the time derivative.  More precisely
\be\label{t'}
\partial_{t} \tilde u_{m}=\partial_{t} u_{m} +u_{m}(0)\delta_{0} -u_{m}(T)\delta_{T}.
\ee
Then the goal is to rewrite \eqref{finiteApproxeq} on the Fourier transform side by replacing $u_{m}$ with $\tilde u_{m}$.  First, by solving for $\partial_t u_m$ in  \eqref{t'} and plugging into \eqref{finiteApproxeq}, and using the definitions of operators $A$ and $B$ we have
\[
(\partial_{t}\tilde u_m(t)- u_{m}(0)\delta_{0} +u_{m}(T)\delta_{T} , w_{j}) + \ip{A \tilde u_{m} , w_{j}}_{\V'\otimes  \V} + \ip{B\tilde u_{m}, w_{j}}_{\V'\otimes  \V} = 0.
\]
After Fourier transform in time, denoted by $\widehat{\cdot}$ , and rearranging, we have
\[
2\pi i \tau ( \widehat{\tilde u_m}(\tau),w_{j})= (u_{m}(0) +u_{m}(T)\exp(-2\pi i T\tau) , w_{j}) - \ip{\widehat{A \tilde u_{m}}(\tau) , w_{j}} - \ip{\widehat{B\tilde u_{m}}(\tau), w_{j}} .
\]
Now, recall $g_{jm}$ is the coefficient function of $u_{m}$ corresponding to the basis $w_{j}$. Let $\tilde g_{jm}(t)=g_{jm}(t)$ for $t\in [0,T]$ and $0$ otherwise, and multiply by $\widehat{\tilde g_{jm}}$ and add to obtain
\begin{align}
2\pi i \tau \norm{ \widehat{\tilde u_m}(\tau)}^2&= (u_{m}(0) +u_{m}(T)\exp(-2\pi i T\tau) , \widehat{\tilde u_{m}}(\tau))\label{hat1}\\
&\quad - \ip{\widehat{A \tilde u_{m}}(\tau) , \widehat{\tilde u_{m}}(\tau)} - \ip{\widehat{B\tilde u_{m}}(\tau), \widehat{\tilde u_{m}}(\tau)}\nonumber\\
&= I + II +III.\nonumber
\end{align}
Next by Cauchy-Schwarz and \eqref{energym}
\be\label{I}
I\leq (\norm{u(0)}_{L^2}+\norm{u(T)}_{L^2})\norm{\widehat {\tilde u_{m}}(\tau)}_{L^2}\leq 2\norm{u_{0}}_{L^2}\norm{\widehat {\tilde u_{m}}(\tau)}_{L^2}.
\ee
Similarly, by Cauchy-Schwarz, definition of the Fourier transform and $Au$, and by \eqref{energym}
\begin{align}
II\leq  \norm{\widehat{ A\tilde u_{m}}(\tau)}_{\V'}\norm{\widehat {\tilde u_{m}}(\tau)}_{\V}&\leq  \int^{\infty}_{-\infty}\norm{A\tilde u_{m}(t)}_{\V'}\dd t\norm{\widehat {\tilde u_{m}}(\tau)}_{\V}\nonumber\\
&=  \int^{T}_{0}\norm{A u_{m}(t)}_{\V'}\dd t\norm{\widehat {\tilde u_{m}}(\tau)}_{\V}\nonumber\\
&\leq  \int^{T}_{0}\norm{\Def u_{m}(t)}_{L^2}\dd t\norm{\widehat {\tilde u_{m}}(\tau)}_{\V}\nonumber\\
&\leq  T^{\frac 12}\int^{T}_{0}\norm{\Def u_{m}(t)}^{2}_{L^2}\dd t\norm{\widehat {\tilde u_{m}}(\tau)}_{\V}\nonumber\\
&\leq  T^{\frac 12}\norm{u_{0}}^{2}_{L^2}\norm{\widehat {\tilde u_{m}}(\tau)}_{\V}.\label{II}
\end{align}
Finally, by Cauchy-Schwarz, definition of Fourier transform, \eqref{b3} and \eqref{energym}
\begin{align}
III&\leq \norm{\widehat{ B\tilde u_{m}}(\tau)}_{\V'}\norm{\widehat {\tilde u_{m}}(\tau)}_{\V}\nonumber\\
&\leq \int^{T}_{0}\norm{B u_{m}(t)}_{\V'}\dd t\norm{\widehat {\tilde u_{m}}(\tau)}_{\V}\nonumber\\
&\leq 2C(1+T)\norm{u_{0}}^{2}_{L^2}\norm{\widehat {\tilde u_{m}}(\tau)}_{\V}.\label{III}
\end{align}
We are now ready to estimate.   We have
\[
\abs{\tau}^{2\gamma}\leq \frac{1+\abs{\tau}}{(1+\abs{\tau})^{1-2\gamma}}.
\]
Then
\begin{align*}
\int^{\infty}_{-\infty}\abs{\tau}^{2\gamma}\norm{\widehat {\tilde u_{m}}(\tau)}_{L^2}^{2}\dd \tau &\leq \int^{\infty}_{-\infty}\frac{1+\abs{\tau}}{(1+\abs{\tau})^{1-2\gamma}}\norm{\widehat {\tilde u_{m}}(\tau)}^{2}_{L^2}\dd\tau\\
&= \int^{\infty}_{-\infty}\frac{1}{(1+\abs{\tau})^{1-2\gamma}}\norm{\widehat {\tilde u_{m}}(\tau)}^{2}_{L^2}\dd\tau+  \int^{\infty}_{-\infty}\frac{\abs{\tau}}{(1+\abs{\tau})^{1-2\gamma}} \norm{\widehat {\tilde u_{m}}(\tau)}^{2}_{L^2} \dd\tau\\
&=A +B.
\end{align*}
We are done if we can show $A$ and $B$ are uniformly bounded.
\[
A\leq \int^{\infty}_{-\infty}\norm{\widehat {\tilde u_{m}}(\tau)}^{2}_{L^2}\dd \tau= \int^{\infty}_{-\infty}\norm{ {\tilde u_{m}}(t)}^{2}_{L^2}\dd t
= \int^{T}_{0}\norm{ u_{m}(t)}^{2}_{L^2}\dd t\leq T\norm{u_{0}}^{2}_{L^2},
\]
by \eqref{energym}.
Next  by \eqref{hat1}, \eqref{I}-\eqref{III}, and by Cauchy-Schwarz and $0<\gamma<\frac 14$ and similar estimates as for term $A$
\begin{align*}
B&\leq C(T, \norm{u_{0}}) \int^{\infty}_{-\infty}\frac{1}{(1+\abs{\tau})^{1-2\gamma}} \norm{\widehat {\tilde u_{m}}(\tau)}_{L^2} d\tau \\
&\leq C(T, \norm{u_{0}},\gamma)\left\{ \int^{T}_{0}\norm{ u_{m}(t)}^{2}_{L^2}\dd t\right\}^{\frac 12}\\
&\leq C( T,\norm{u_{0}},\gamma)
\end{align*}
This completes Step 4.
 
\noindent{\bf Step 5 : Passing to the limit for the linear part of the Navier-Stokes equation.}\\
To begin, we consider a function $\psi \in C^1([0,T])$ which satisfies $\psi (T) =0$, and multiply \eqref{finiteApproxeq} by $\psi (t)$ and take the time integral over $[0,T]$. Then, a simple integration by parts gives
\begin{equation}\label{forpasstothelimiteq}
\begin{split}
-\int_{0}^{T} \big ( u_m(t) , w_j \big ) \psi'(t) \dd t & + \int_{0}^{T} ((  u_m(t) , w_j )) \psi (t) \dd t \\
& + \int_{0}^{T} b(u_m(t) , u_m(t) , w_j) \psi(t) \dd t = \big ( u_{0m} , w_j\big ) \psi (0) .
\end{split}
\end{equation}
Then by \eqref{energym} (and using $L^2$ norm for $\textbf{H}\oplus \F$)
\begin{equation}\label{Uniformfunctionalest}
\begin{split}
\|u_m\|_{L^{\infty}(0,T ; \textbf{H}\oplus \mathbb{F})} & \leq \|u_0\|_{L^2}, \\
\|u_m\|_{L^2(0,T ; \mathbb{V})} & \leq (T+1)^{\frac{1}{2}} \|u_0\|_{L^2} .
\end{split}
\end{equation}
Since both $\textbf{H}\oplus \mathbb{F}$ and $\mathbb{V}$ are reflexive (simply due to the fact that they are both Hilbert spaces), the two uniform estimates in \eqref{Uniformfunctionalest} ensure the existence of some subsequence $u_{m'}$, together with some limiting element $u \in L^{\infty}(0,T ; \textbf{H}\oplus \mathbb{F}) \cap L^2(0,T ; \mathbb{V})$ for which we have the following weak-$\ast$ convergence and weak convergence\\

\begin{itemize}
\item The weak-$\ast$ convergence of the subsequence $u_{m'}$ to $u$ in $L^{\infty}(0,T ; \textbf{H}\oplus \mathbb{F}) $ holds in the sense that: for any $v \in L^1(0,T ; \textbf{H}\oplus \mathbb{F})$, we have $\int_{0}^T \big ( u_{m'}(t) -u(t) , v(t)\big )  \dd t \rightarrow 0$, as $m'\rightarrow \infty$.
\item The weak convergence of the subsequence $u_{m'}$ to $u$ in $L^2(0,T ; \mathbb{V})$ holds in the sense that: for any $v \in L^2(0,T ; \mathbb{V}')$, we have $\int_{0}^{T}\big < v(t) , u_{m'}(t) - u(t) \big >_{\mathbb{V}'\otimes  \mathbb{V}} \dd t \rightarrow 0$, as $m' \rightarrow \infty$.
\end{itemize}
Recall that $\psi \in C^1([0,T])$, so it is plain to see that $\psi' w_j \in C^0([0,T] ; \mathbb{V}) \subset L^1 (0,T ; \textbf{H}\oplus \mathbb{F})$. Hence, we can invoke the weak-$\ast$ convergence of $u_{m'}$ to $u$ in $L^{\infty}(0,T ; \textbf{H}\oplus \mathbb{F}) $ to deduce

\begin{equation}\label{PASSONE}
\lim_{m'\rightarrow \infty} \int_{0}^T \big ( u_{m'}(t) , \psi'(t) w_j \big )\dd t =  \int_{0}^T \big ( u(t) , \psi'(t) w_j \big ) \dd t .
\end{equation}

Also, it is equally plain to see that $\psi A(w_j) \in L^2(0,T ; \mathbb{V}')$.  Hence, we invoke the weak convergence of $u_{m'}$ to $u$ in $L^2(0,T ; \mathbb{V})$ to obtain
\begin{equation}\label{PASSTWO}
\begin{split}
\lim_{m'\rightarrow \infty} \int_{0}^T (( u_{m'}(t) , \psi(t) w_j  )) \dd t & = \lim_{m\rightarrow \infty} \int_{0}^T
\big < \psi(t) A (w_j) , u_{m'}(t)  \big >_{\mathbb{V}'\otimes  \mathbb{V}} \dd t \\
& =  \int_{0}^T \big < \psi(t) A (w_j) , u(t)  \big >_{\mathbb{V}'\otimes  \mathbb{V}} \dd t \\
& = \int_{0}^T (( u(t) , \psi(t) w_j  )) \dd t .
\end{split}
\end{equation}

Next, we have to deal with the more delicate part, namely passing to the limit for the nonlinear term $b(u_{m'}(t) , u_{m'}(t) , w_j)$, as $m\rightarrow \infty$.\\

\noindent{\bf Step 6 : Passing to the limit for the nonlinear part of the Navier-Stokes equation.}
From Step 4 we have
\begin{equation}\label{fractionalestONE}
\|\tilde u_m\|_{H^{\gamma}(\mathbb{R} , \mathbb{V} , \textbf{H}\oplus \mathbb{F} )} \leq C_0(T, u_0, \gamma), \quad \gamma \in (0,\frac 14) .
\end{equation}

Let $R > 0$, and  let $O \in \mathbb{H}^2(-a^2)$  be any selected reference point, and consider the geodesic ball $B_{0}(R) $.  We need to consider  the restriction of the $1$-form $\tilde u_m (t) = \sum_{j=1}^{m} \tilde g_{jm} (t) w_j$ onto the geodesic ball $B_{O}(R)$, which we denote by $\tilde u_m \big |_{B_{O}(R)}$. Then  we have $\tilde u_m \big |_{B_O(R)} \in L^{\infty}\big ( \mathbb{R} ; L^2(B_O(R)) \big ) \cap L^2 \big ( \mathbb{R} ; H^1(B_O(R)) \big ) $, which is ensured by the following estimates

\begin{equation}\label{TRIVIALClarify}
\begin{split}
\|\tilde u_m  \|_{L^\infty (\mathbb{R} ; L^2(B_O(R)) )} & \leq \|\tilde u_m\|_{L^\infty (\mathbb{R} ; \textbf{H}\oplus \mathbb{F})} \leq \|u_0\|_{L^2(\mathbb{H}^2(-a^2))} ,\\
\|\tilde u_m \|_{L^2(\mathbb{R} ; H^1(B_O (R)))} & \leq \|\tilde u_m\|_{L^2( \mathbb{R} ; \mathbb{V})} \leq \|u_0\|_{L^2(\mathbb{H}^2(-a^2))} .
\end{split}
\end{equation}

\begin{remark}
The two estimates in \eqref{TRIVIALClarify} are indeed trivial. The intention of demonstrating them here is to emphasize that, by taking restriction of $\tilde u_m$ onto $B_{O}(R)$, we \emph{no longer} work with the original functional spaces such as $ L^\infty (\mathbb{R} ; \textbf{H}\oplus \mathbb{F})$, or $L^2(\mathbb{R} ; \mathbb{V})$.
In passing to a suitable subsequence of $\{u_m\}_{m=0}^{\infty}$ in order to achieve strong convergence of $u_m$ to $u$ \textbf{in} $L^2(B_O(R))$, we will work with the functional spaces $L^\infty (0,T ; L^2(B_O(R)) )$ and $L^2(0,T ; H^1(B_O (R)))$.
\end{remark}

Also, one has to make a careful distinction between $\widehat{ \tilde u_m}(t)$ and $  ( \tilde u_m \big |_{B_O(R)} )^{\widehat{}}(t)$.
They are related via the simple identity

\begin{equation}\label{Verytrivial}
( \tilde u_m \big |_{B_O(R)} )^{\widehat{}}(t) = \widehat{ \tilde u_m}(t) \big |_{B_O(R)} ,
\end{equation}
whose validity is ensured by the following straightforward computation
\begin{equation}
( \tilde u_m \big |_{B_O(R)} )^{\widehat{}}(t) = \sum_{j=1}^{m} \widehat{\tilde g_{jm}}(t)  w_j \big |_{B_O(R)}
= \big ( \sum_{j=1}^m \widehat{\tilde g_{jm}}(t) w_j \big ) \big |_{B_O(R)}
= \widehat{\tilde u_m}(t) \big |_{B_O(R)} .
\end{equation}

So by \eqref{fractionalestONE} we have the following estimate
\begin{equation}\label{restrict1}
\begin{split}
\int_{-\infty}^{\infty}|\tau|^{2\gamma} \big \| ( \tilde u_m \big |_{B_O(R)} )^{\widehat{}}(\tau ) \big \|_{L^2(B_O(R))}^2 \dd \tau
& = \int_{-\infty}^{\infty}|\tau|^{2\gamma} \big \| \widehat{\tilde u_m} (\tau )   \big \|_{L^2(B_O(R))}^2 \dd \tau \\
& \leq \|\tilde u_m\|_{H^{\gamma}(\mathbb{R} , \mathbb{V} , \textbf{H}\oplus \mathbb{F} )}^2 \leq C_0^2.
\end{split}
\end{equation}
Then  the first estimate in \eqref{TRIVIALClarify} and \eqref{restrict1} simply yield the following uniform estimate
\begin{equation}\label{fractionaestTwo}
\|\tilde u_m\|_{H^{\gamma} (\mathbb{R} ; H^1(B_O(R)) , L^2(B_O(R))  )} \leq \big ( C_0^2 + \|u_0\|_{L^2}^2\big )^{\frac{1}{2}} .
\end{equation}
Observe that by construction, we have $\supp (\tilde u_m \big |_{B_O(R)} ) \subset [0,T]$, for each $m \in \mathbb{Z}^+$. So, by invoking the compactness Lemma \ref{compact2}, we deduce from \eqref{fractionaestTwo} that by further passing to a subsequence of the subsequence $u_{m'}$, which will still be written as $u_{m'}$ for simplicity, it follows that we have the following strong convergence of $u_{m'}$ to $u$
\begin{equation}\label{strongL^2convergence}
\lim_{m'\rightarrow \infty} \|u_{m'} - u \|_{L^2 ( 0,T ;L^2(B_O(R))  )} = 0.
\end{equation}

Now, recall that the basis $\{w_j\}_{j=1}^{\infty}$ of $\mathbb{V} = \textbf{V} \oplus \mathbb{F}$ is given by the following rule
\begin{equation}
\begin{split}
w_{2k-1} & = e_k ,\\
w_{2k} & = \dd F_k ,
\end{split}
\end{equation}
where $\{e_k\}_{k=1}^{\infty} \subset \Lambda^1_{c,\sigma}(\H)$ is a basis of $\textbf{V}$, while $\{\dd F_k\}_{k=1}^{\infty}$ constitutes a basis for $\mathbb{F}$. Now, we further split our discussion into two subcases as follows.\\

\noindent{\bf Subcase 1, the easier one: passing to the limit in the case of $w_{2k-1} = e_k$. }\\

In the case of $w_{2k-1} = e_k$, we follow \cite[p. 196]{Temam} and pass to the limit for the term $\int_0^T b(u_{m'} , u_{m'} , e_k) \dd t$ as follows. Since $e_k$ is a smooth $d^*$-closed $1$-form with compact support, we can just choose the radius $R > 0$ to be sufficiently large so that $\supp e_k \subset \subset B_O(R)$, where $O$ is some preferred reference point in $\mathbb{H}^2(-a^2)$. Then, we have the following estimate

\begin{equation}
\begin{split}
\big | b(u-u_{m'} , e_k , u_{m'} )(t)  \big | & = \bigg | \int_{B_O(R)} \overline{g} \big ( \overline{\nabla}_{(u - u_{m'})(t)} e_k ,u_{m'}(t)  \big )\Vol_{\H} \bigg | \\
& \leq \big \|\overline{\nabla} e_k \big \|_{L^{\infty}} \big \| u(t) - u_{m'}(t)  \big \|_{L^2(B_O(R))}  \big \| u_{m'}(t) \big \|_{L^2(B_O(R))}
\end{split}
\end{equation}
Here, of course, $\|\overline{\nabla} e_k\|_{L^{\infty}}$ means the $L^{\infty}$-norm of $\overline{\nabla} e_k$ over $\mathbb{H}^2(-a^2)$, which is definitely finite, since $e_k$ is a smooth 1-form with compact support. By integrating the above inequality over the time interval $[0,T]$, we yield
\begin{align}
\int_0^T \big | b(u-u_{m'} , e_k , u_{m'} )(t) \big | \dd t & \leq \big \|\overline{\nabla} e_k \big \|_{L^{\infty}} \cdot
\big \| u-u_{m'}  \big \|_{L^2(0,T ; L^2(B_O(R))  )}  \big \| u_{m'}  \big \|_{L^2(0,T ; L^2(B_O(R))  )} \nonumber\\
& \leq T^{\frac{1}{2}} \big \|\overline{\nabla} e_k \big \|_{L^{\infty}}  \big \|u_0 \big \|_{L^2}  \big \| u-u_{m'}  \big \|_{L^2(0,T ; L^2(B_O(R))  )} ,\label{GoodestONE}
\end{align}
by \eqref{energym}.
In the same way, we have
\begin{equation}\label{GoodestTWO}
\begin{split}
\int_0^T \big | b(u , e_k , u - u_{m'} )(t) \big | \dd t
  \leq  T^{\frac{1}{2}} \big \|u_0\|_{L^2}\cdot  \big \|\overline{\nabla} e_k \big \|_{L^{\infty}} \cdot \big \| u-u_{m'}  \big \|_{L^2(0,T ; L^2(B_O(R))  )} .
\end{split}
\end{equation}
Now, by means of the strong convergence of $u_{m'}$ to $u$ in $L^2(0,T ; L^2(B_O(R)))$, we pass to the limit in inequalities \eqref{GoodestONE}, and \eqref{GoodestTWO} to deduce that
\begin{equation}\label{PASSLIMITONE}
\begin{split}
& \lim_{m'\rightarrow \infty} \int_0^T \big | b(u-u_{m'} , e_k , u_{m'} )(t) \big | \dd t  = 0 ,\\
& \lim_{m'\rightarrow \infty} \int_0^T \big | b(u , e_k , u - u_{m'} )(t) \big | \dd t = 0.
\end{split}
\end{equation}
However, observe that by \eqref{bp1} we have the following identity.
\begin{equation}\label{usualid}
\begin{split}
\int_0^T b(u_{m'} ,u_{m'} , e_k )(t) - b(u,u,e_k)(t) \dd t & =
\int_0^T -b(u_{m'}, e_k, u_{m'})(t) + b(u,e_k,u)(t) \dd t \\
& =  \int_0^T b(u-u_{m'} , e_k , u_{m'})(t) \dd t + \int_0^T b(u,e_k, u-u_{m'})(t) \dd t .
\end{split}
\end{equation}
As a result, \eqref{PASSLIMITONE} immediately implies
\begin{equation}
\lim_{m' \rightarrow \infty}\int_0^T b(u_{m'} ,u_{m'} , e_k )(t) \dd t = \int_0^T b(u,u,e_k)(t) \dd t .
\end{equation}

\noindent{\bf Subcase 2, the more delicate one: passing to the limit in the case of $w_{2k} = \dd F_k \in \mathbb{F}$. }\\
In this case of $w_{2k} = \dd F_{k} \in \mathbb{F}$, the idea in \cite[p.196]{Temam} cannot be used directly, simply due to the fact that $w_{2k} = \dd F_{k}$ is no longer compactly supported. So, to make that idea survive in this case, we first decompose $w_{2k}= \dd F_{k}$ into a local part within an open ball $B_O(R)$ with sufficiently large radius $R >0$ (i.e. , $\dd F_{k} \big |_{B_O(R)}$), and a nonlocal part outside of $B_O(R)$ (i.e. the restriction of $\dd F_k$ on $\mathbb{H}^2(-a^2)-B_O(R)$ ). The secret behind the success of this decomposition lies in the fact that $\dd F_k \in \mathbb{F}$ enjoys the property that $\|\dd F_k\|_{H^1} < \infty$, which enables us to reduce the $L^2$-norm of $\overline{\nabla}\dd F_k$ on the \emph{exterior domain} $\mathbb{H}^2(-a^2)-B_O(R)$ to become as small as we want by choosing the radius $R > 0$ to be sufficiently large. In this way, the contribution to the term as appears in \eqref{mainTerm} below due to the nonlocal part $\dd F_k \big |_{\mathbb{H}^2(-a^2)-B_O(R)}$ will become as small as we want by choosing a large radius $R$:
\begin{equation}\label{mainTerm}
\int_0^T b(u_{m'} ,u_{m'} , \dd F_k )(t) - b(u,u,\dd F_k)(t) \dd t.
\end{equation}
Then, we will deal with the contribution to the term as appears in \eqref{mainTerm} due to the local part $\dd F_k \big |_{B_O(R)}$ by passing to the limit as $m'\rightarrow \infty$, just as what we did in dealing with the case of $w_{2k-1} =e_k$. Here, let us demonstrate the details of this argument as follows.\\

First, we have, as in \eqref{usualid}
\begin{equation}\label{Goodstuff}
\begin{split}
& \int_0^T b(u_{m'} ,u_{m'} , \dd F_k )(t) - b(u,u,\dd F_k)(t) \dd t \\
=&  \int_0^T b(u-u_{m'} , \dd F_k , u_{m'})(t) \dd t + \int_0^T b(u, \dd F_k, u-u_{m'})(t) \dd t .
\end{split}
\end{equation}
We first decompose the term $\int_0^T b(u-u_{m'} , \dd F_k , u_{m'})(t) \dd t$ into the local part and the far-range part in the following manner.
\begin{equation}\label{localFarrangeONE}
\begin{split}
& \int_0^T  b(u-u_{m'} , \dd F_k , u_{m'})(t) \dd t  \\
= & \int_0^T \int_{B_O(R)} \overline{g} \big ( \overline{\nabla}_{(u-u_{m'})(t)} \dd F_k , u_{m'} \big ) \dd t
+ \int_0^T \int_{\mathbb{H}^2(-a^2)-B_O(R)} \overline{g} \big ( \overline{\nabla}_{(u-u_{m'})(t)} \dd F_k , u_{m'} \big ) \dd t
\end{split}
\end{equation}

Now, we estimate the far range part of $\int_0^T  b(u-u_{m'} , \dd F_k , u_{m'})(t) \dd t$  by means of Holder's inequality as follows

\begin{equation}\label{HOLDERFARONE}
\begin{split}
& \bigg |\int_0^T \int_{\mathbb{H}^2(-a^2)-B_O(R)} \overline{g} \big ( \overline{\nabla}_{(u-u_{m'})(t)} \dd F_k , u_{m'}(t) \big ) \dd t  \bigg | \\
\leq & \big \|\overline{\nabla}\dd F_k  \big \|_{L^2(\mathbb{H}^2(-a^2)-B_O(R))} \cdot \big \| u-u_{m'} \big \|_{L^2(0,T ; L^4(\mathbb{H}^2(-a^2)))}
\cdot \big \| u_{m'} \big \|_{L^2(0,T ; L^4(\mathbb{H}^2(-a^2)))}.
\end{split}
\end{equation}
However, by means of Ladyzhenskaya's inequality \eqref{ldh}, we deduce from \eqref{Uniformfunctionalest} that we have the following uniform estimates
\begin{equation}\label{USEFULUNIFORM}
\begin{split}
\big \| u_{m'} \big \|_{L^2(0,T ; L^4(\mathbb{H}^2(-a^2)))} & \leq C_0 \big \| u_{m'} \big \|_{L^2(0,T ; H^1(\mathbb{H}^2(-a^2)))}  \leq C_0 (1+ T)^{\frac{1}{2}} \|u_0\|_{L^2} , \\
\big \| u \big \|_{L^2(0,T ; L^4(\mathbb{H}^2(-a^2)))} & \leq C_0 \big \| u \big \|_{L^2(0,T ; H^1(\mathbb{H}^2(-a^2)))}  \leq C_0 (1+ T)^{\frac{1}{2}} \|u_0\|_{L^2}.
\end{split}
\end{equation}
Hence, it follows from \eqref{HOLDERFARONE} that we have for all $R > 0$
\begin{equation}\label{OK}
\begin{split}
& \bigg |\int_0^T \int_{\mathbb{H}^2(-a^2)-B_O(R)} \overline{g} \big ( \overline{\nabla}_{(u-u_{m'})(t)} \dd F_k , u_{m'}(t) \big ) \dd t  \bigg |  \\
\leq & C_0^2 (1+T) \|u_0\|_{L^2}^2 \cdot \big \|\overline{\nabla}\dd F_k  \big \|_{L^2(\mathbb{H}^2(-a^2)-B_O(R))}.
\end{split}
\end{equation}
Now, thanks a lot to the fact that $\dd F_k$, as an element in $\mathbb{F}$, must enjoy the property that
\begin{equation}
\big \|\overline{\nabla}\dd F_k  \big \|_{L^2(\mathbb{H}^2(-a^2))} \leq \|\dd F_k\|_{H^1 (\mathbb{H}^2(-a^2))} < \infty ,
\end{equation}
 we easily deduce that we must have
\begin{equation}
\lim_{R \rightarrow \infty} \big \|\overline{\nabla}\dd F_k  \big \|_{L^2(\mathbb{H}^2(-a^2)-B_O(R))} = 0.
\end{equation}
So, for any arbitrary small positive number $\epsilon > 0$, we can choose some sufficiently large positive $R_{\epsilon} >0$, depending only on $\epsilon >0$ and $\dd F_k$, such that
\begin{equation}\label{smallL2}
\|\overline{\nabla}\dd F_k  \big \|_{L^2(\mathbb{H}^2(-a^2)-B_O(R_{\epsilon}))} < \epsilon.
\end{equation}
Hence, from \eqref{localFarrangeONE} and \eqref{OK} that
\begin{equation}\label{GOODOUTCOMEONE}
\begin{split}
& \bigg | \int_0^T  b(u-u_{m'} , \dd F_k , u_{m'})(t) \dd t \bigg | \\
\leq & \bigg | \int_0^T \int_{B_O(R_{\epsilon})} \overline{g} \big ( \overline{\nabla}_{(u-u_{m'})(t)} \dd F_k , u_{m'} \big ) \dd t \bigg |
+ C_0^2 (1+T) \big \|u_0 \big \|_{L^2}^2 \cdot \epsilon .
\end{split}
\end{equation}
In exactly the same way, we decompose the term $\int_0^T  b(u , \dd F_k , u- u_{m'})(t) \dd t$ into the local and far range part  (except for this time, we knew already that we use the radius $R_{\epsilon} > 0$ in our decomposition), and obtain
\begin{equation}\label{GOODOUTCOMETWO}
\begin{split}
& \bigg | \int_0^T  b(u , \dd F_k , u - u_{m'})(t) \dd t \bigg |  \\
\leq & \bigg | \int_0^T \int_{B_O(R_{\epsilon})} \overline{g} \big ( \overline{\nabla}_{u(t)} \dd F_k , (u-u_{m'})(t) \big ) \dd t \bigg |
+ C_0^2 (1+T) \|u_0\|_{L^2}^2  \cdot  \epsilon .
\end{split}
\end{equation}
Now, by combining the uniform estimates \eqref{GOODOUTCOMEONE} and \eqref{GOODOUTCOMETWO}, we deduce from \eqref{Goodstuff} that
\begin{equation}\label{FinalGlory}
\begin{split}
&\bigg | \int_0^T b(u_{m'} ,u_{m'} , \dd F_k )(t) - b(u,u,\dd F_k)(t) \dd t \bigg | \\
\leq & \bigg | \int_0^T \int_{B_O(R_{\epsilon})} \overline{g} \big ( \overline{\nabla}_{(u-u_{m'})(t)} \dd F_k , u_{m'} \big ) \dd t \bigg | \\
&+ \bigg | \int_0^T \int_{B_O(R_{\epsilon})} \overline{g} \big ( \overline{\nabla}_{u(t)} \dd F_k , (u-u_{m'})(t) \big ) \dd t \bigg |
+ 2C_0^2 (1+T) \|u_0\|_{L^2}^2  \cdot  \epsilon .
\end{split}
\end{equation}
Next, we carry out the following two uniform estimates by means of Holder's inequality in the same spirit of estimates \eqref{GoodestONE} and \eqref{GoodestTWO}
\begin{equation}\label{GoodestThree}
\begin{split}
& \bigg | \int_0^T \int_{B_O(R_{\epsilon})} \overline{g} \big ( \overline{\nabla}_{(u-u_{m'})(t)} \dd F_k , u_{m'} \big ) \dd t \bigg | \\
\leq &  \big \|\overline{\nabla} \dd F_k \big \|_{L^{\infty}(B_O(R_{\epsilon}))}
\big \| u-u_{m'}  \big \|_{L^2(0,T ; L^2(B_O(R_{\epsilon}))  )}  \big \| u_{m'}  \big \|_{L^2(0,T ; L^2(B_O(R_{\epsilon}))  )} \\
\leq & T^{\frac{1}{2}} \big \|\overline{\nabla} \dd F_k \big \|_{L^{\infty}(B_O(R_{\epsilon}))}  \big \|u_0 \big \|_{L^2}  \big \| u-u_{m'}  \big \|_{L^2(0,T ; L^2(B_O(R_{\epsilon}))  )} ,
\end{split}
\end{equation}
and
\begin{equation}\label{GoodestFour}
\begin{split}
& \bigg | \int_0^T \int_{B_O(R_{\epsilon})} \overline{g} \big ( \overline{\nabla}_{u(t)} \dd F_k , (u-u_{m'})(t) \big ) \dd t \bigg | \\
\leq &  T^{\frac{1}{2}} \big \|u_0\|_{L^2}  \big \|\overline{\nabla} \dd F_k \big \|_{L^{\infty}(B_O(R_{\epsilon}))}  \big \| u-u_{m'}  \big \|_{L^2(0,T ; L^2(B_O(R_{\epsilon}))  )}.
\end{split}
\end{equation}
\begin{remark}
In  \eqref{GoodestThree} and \eqref{GoodestFour}, we have employed the local $L^{\infty}$-estimate $\big \|\overline{\nabla} \dd F_k \big \|_{L^{\infty}(B_O(R_{\epsilon}))}$ which is always finite, simply due to the classical smoothness of $\overline{\nabla} \dd F_k $ over $\mathbb{H}^2(-a^2)$.
\end{remark}
Now, recall that by \eqref{strongL^2convergence} we have strong convergence of $u_{m'}$ to the limiting function $u$ in $L^2(0,T ; L^2(B_O(R_{\epsilon})))$.
So, by taking $\limsup$ in \eqref{GoodestThree} and \eqref{GoodestFour}, it follows from \eqref{FinalGlory} that the following relation must hold
\begin{equation}
\limsup_{m'\rightarrow \infty}  \bigg | \int_0^T b(u_{m'} ,u_{m'} , \dd F_k )(t) - b(u,u,\dd F_k)(t) \dd t \bigg | \leq 2C_0^2 (1+T) \|u_0\|_{L^2}^2  \cdot  \epsilon .
\end{equation}
Since the above $\limsup$-estimate holds for any positive number $\epsilon >0$, it follows that we must have the following conclusion
\begin{equation}
\lim_{m'\rightarrow \infty}  \bigg | \int_0^T b(u_{m'} ,u_{m'} , \dd F_k )(t) - b(u,u,\dd F_k)(t) \dd t \bigg | = 0 ,
\end{equation}
which is equivalent to saying that we finally have
\begin{equation}
\lim_{m'\rightarrow \infty}\int_0^T b(u_{m'} ,u_{m'} , \dd F_k )(t) \dd t = \int_0^T b(u,u,\dd F_k)(t) \dd t ,
\end{equation}
as needed.

\noindent{\bf Step 7: Satisfying $u(0)=u_0 \in \textbf{H}\oplus \mathbb{F}$.}

With all the works  in Step 1 through  Step 6, we are now able to pass to the limit in \eqref{forpasstothelimiteq} and obtain
\begin{equation}\label{limitequationONE}
\begin{split}
-\int_{0}^{T} \big ( u(t) , w_j \big )  \psi'(t) \dd t & + \int_{0}^{T} ((  u(t) , w_j )) \psi (t) \dd t \\
& + \int_{0}^{T} b(u(t) , u(t) , w_j) \psi(t) \dd t = \big ( u_{0} , w_j\big )  \psi (0) ,   \quad j \geq 1.
\end{split}
\end{equation}

Since $\{w_k\}_{k=1}^{\infty}$ is a basis of $\mathbb{V} = \textbf{V} \oplus \mathbb{F}$, it follows, through a simple density argument, that the above relation still holds if the basis element $w_k$ is replaced by an arbitrary test $1$-form $v \in \mathbb{V}$.
Hence,
\begin{equation}\label{limitequationTWO}
\begin{split}
-\int_{0}^{T} \big ( u(t) , v \big )   \psi'(t) \dd t & + \int_{0}^{T} ((  u(t) , v )) \psi (t) \dd t \\
& + \int_{0}^{T} b(u(t) , u(t) , v) \psi(t) \dd t = \big ( u_{0} , v\big )   \psi (0),\quad  v \in \mathbb{V} .
\end{split}
\end{equation}
By taking $\psi$ to be a smooth compactly supported function on $(0,T)$ in \eqref{limitequationTWO}, we have that the limiting function $u \in L^{\infty}(0,T ; \textbf{H}\oplus \mathbb{F}) \cap L^2(0,T; \mathbb{V})$ satisfies the following equation on $(0,T)$ in the distributional sense
\begin{equation}\label{equivalentequation}
\frac{\dd }{\dd t} \ip{u(t) , v}_{\mathbb{V}'\otimes \mathbb{V}} = - \ip{Au(t) + B u(t) , v}_{\mathbb{V}'\otimes \mathbb{V}} ,\quad v\in \V,
\end{equation}
where we recall that $Au(t), Bu(t) \in \mathbb{V}'$ are defined by
\begin{equation}
\begin{split}
\ip{Au(t) , v}_{\mathbb{V}'\otimes \mathbb{V}} & = ((u(t) , v)) ,\\
\ip{Bu(t) , v}_{\mathbb{V}'\otimes \mathbb{V}} & = b(u(t) , u(t) , v) = \int_{\mathbb{H}^2(-a^2)} g (\overline{\nabla}_{u}u(t) , v) \Vol_{\mathbb{H}^2(-a^2)} .
\end{split}
\end{equation}
We now invoke Lemma \ref{useful_lemma} to deduce from \eqref{equivalentequation} that the following relation holds in the weak sense
\begin{equation}
\frac{\dd u}{\dd t} = -Au - Bu,
\end{equation}
and that $u \in C^0([0,T] ; \mathbb{V}')$ satisfies the following relation, where $u(0)$ simply means the value of the continuous $\mathbb{V}'$-valued function $u(t)$ at the end-point $t =0$
\begin{equation}\label{ReallyGoodeq}
u(t) = u(0) - \int_{0}^t \big ( Au(\tau ) + Bu(\tau ) \big ) \dd \tau .
\end{equation}
We also note that since by Lemmas \ref{AONE} and \ref{ATWO}  $Au + Bu \in L^2(0,T ; \mathbb{V}')$, we therefore have
\be\label{betterregularity}
\frac{\dd u}{\dd t} \in L^2(0,T ; \mathbb{V}').
\ee

Next, we need to show that $u(0) \in \mathbb{V}'$ coincides with $u_0 \in L^2(\mathbb{H}^2(-a^2))$, \emph{provided} $u_0$ is being thought of as an element in the broader space $\mathbb{V}'$ through the natural inclusion $L^2(\mathbb{H}^2(-a^2)) \subset \mathbb{V}'$. For completeness we show the details omitted in \cite{Temam}.

First observe that, for any given $v \in \mathbb{V}$, \eqref{ReallyGoodeq} gives
\begin{equation}\label{ReallyGoodTwo}
\ip{u(t) , v}_{\mathbb{V}'\otimes \mathbb{V}} = \ip{u(0) , v}_{\mathbb{V}'\otimes \mathbb{V}} - \int_0^t \ip{Au(t) + Bu(t) , v}_{\mathbb{V}'\otimes \mathbb{V}}, \quad t \in [0,T]
\end{equation}
Using \eqref{ReallyGoodTwo} and a classical theorem  in Lebesgue measure theory (for instance \cite[Thm 3.35, p.106]{Folland}) allows us to deduce that $\ip{u(t) , v}_{\mathbb{V}'\otimes \mathbb{V}}$ must be absolutely continuous on $[0,T]$. Now, we take any $\psi \in C^1([0,T])$ satisfying $\psi(T) = 0$ and $\psi(0) = 1$. Then, it is easy to check that $\psi (t) \ip{u(t) , v}_{\mathbb{V}'\otimes \mathbb{V}} $ is also absolutely continuous on $[0,T]$, and hence by the same theorem  the \emph{classical} derivative $\frac{\dd }{\dd t} \big ( \psi (t) \ip{u(t) , v}_{\mathbb{V}'\otimes \mathbb{V}} \big ) $ exists for almost every $t \in [0,T]$, and

\begin{equation}\label{ReallyGoodThree}
\begin{split}
\ip{u(0), v}_{\mathbb{V}'\otimes \mathbb{V}} & = -\psi(T) \ip{u(T), v}_{\mathbb{V}'\otimes \mathbb{V}} + \psi(0) \ip{u(0), v}_{\mathbb{V}'\otimes \mathbb{V}} \\
& = - \int_0^T \frac{\dd }{\dd t} \bigg ( \psi (\tau ) \ip{u(\tau ) , v}_{\mathbb{V}'\otimes \mathbb{V}} \bigg ) \dd \tau.
\end{split}
\end{equation}

However, since the absolute continuity of $\ip{u(t), v}_{\mathbb{V}'\otimes \mathbb{V}}$ on $[0,T]$ also ensures the existence of the classical derivative $\frac{\dd}{\dd t}\ip{u(t), v}_{\mathbb{V}'\otimes \mathbb{V}}$ at almost every $t \in [0,T]$, we can apply the classical product rule at all those (almost everywhere) points to deduce
\begin{equation}\label{productRuleClassical}
\frac{\dd}{\dd t} \bigg ( \psi (\tau ) \ip{u(\tau ) , v}_{\mathbb{V}'\otimes \mathbb{V}} \bigg ) = \psi'(t)\ip{u(t ) , v}_{\mathbb{V}'\otimes \mathbb{V}}  + \psi(t) \frac{\dd}{\dd t}\ip{u(t) , v}_{\mathbb{V}'\otimes \mathbb{V}},\quad \mbox{a.e.}\ t.
\end{equation}
By combining \eqref{equivalentequation}, \eqref{ReallyGoodThree}, and \eqref{productRuleClassical}, we obtain
\begin{equation}\label{AnotherCopy}
\begin{split}
\ip{u(0), v}_{\mathbb{V}'\otimes \mathbb{V}} & = - \int_0^T \psi'(\tau )\ip{u(\tau ), v}_{\mathbb{V}'\otimes \mathbb{V}} \dd \tau \\
&+ \int_0^T \psi(\tau ) \ip{Au(\tau ) + Bu(\tau ) , v }_{\mathbb{V}'\otimes \mathbb{V}} \dd \tau \\
& = - \int_0^T \psi'(\tau ) (u(\tau ), v  ) \dd \tau + \int_0^T ((u(\tau ) , v )) \psi(\tau ) \dd \tau \\
& + \int_0^T b(u(\tau) u(\tau ) , v )\psi (\tau ) \dd \tau .
\end{split}
\end{equation}
Now, by comparing \eqref{limitequationTWO} with \eqref{AnotherCopy} we conclude that the following relation holds for any test $1$-form $v \in \mathbb{V}$.
\begin{equation}
\ip{u(0) , v}_{\mathbb{V}'\otimes \mathbb{V}} = \ip{u_0 , v}_{\mathbb{V}'\otimes \mathbb{V}} .
\end{equation}
This shows that as long as we think of the element $u_0 \in L^2(\mathbb{H}^2(-a^2))$ in the broader sense as an element in $\mathbb{V}'$, our limiting weak solution $u$, as a continuous $\mathbb{V}'$-valued function on $[0,T]$, will assume $u_0$ as its initial value at $t =0$. That is, the relation $u(0) = u_0$ holds in the sense of $\mathbb{V}'$.

However, since $\V\subset \textbf{H}\oplus \F \subset \V',$  we can use Lemma \ref{anotherlemma} and the facts that $u\in L^2(0,T; \V)$ and that $\frac{\partial u}{\partial t}\in L^2(0,T; \V')$ to deduce that in fact, $u \in C([0,T]; \textbf{H}\oplus \F)$.  Hence, $u(0)=u_0$ holds also in $\textbf{H}\oplus \F$.

\section{Uniqueness}\label{unique}
Let $u=u_1-u_2$, where $u_1, u_2$ solve \eqref{FINALEQUATION}.  The idea is to use the Gronwall's inequality applied to $\norm{u}_2$.  We stress that now that we have established the properties of $b$ and $B$ in the hyperbolic setting, the proof is identical to \cite{Temam} with the exception of one extra term that comes after application of \eqref{ldh} since that inequality now involves the  \emph{nonhomogenous} $H^1$ norm.

We have $u$ solves
\[
u_t+ Au+Bu_1-Bu_2=0,
\]
with $u(0)=0$.  By Lemma \ref{anotherlemma}, definition of $A$ and $B$
\begin{align*}
0&=\ip{u(t), u_t(t)+ Au(t)+Bu_1(t)-Bu_2(t)}_{\V\otimes  \V'}\\
&=\frac \dd{\dd t}\norm{u(t)}^2_2+2 \norm{u}^2_{\dot H^1}+2b(u_1(t),u_1(t),u(t))-2b(u_2(t),u_2(t),u(t)).
\end{align*}
Hence
\be\label{bu0}
\frac \dd{\dd t}\norm{u(t)}_2^2+2 \norm{u}^2_{\dot H^1}=2b(u_2(t),u_2(t),u(t))-2b(u_1(t),u_1(t),u(t)).
\ee
On the right-hand side we add and subtract $2b(u_1(t), u_2(t), u(t))$ and use multilinearity of $b$ and  \eqref{bp2} to obtain
\[
2b(u_2(t),u_2(t),u(t))-2b(u_1(t),u_1(t),u(t))=-2b(u(t),u_2(t),u(t)).
\]
Next by \eqref{b1} and Cauchy's inequality
\[
2\abs{b(u(t),u_2(t),u(t))}\leq 4C\norm{u(t)}_2\norm{u(t)}_{H^1}\norm{u_2(t)}_{\dot H^1}\leq 2\norm{u(t)}_{H^1}^2+{8C^2}\norm{u(t)}_2^2\norm{u_2(t)}^2_{\dot H^1}.
\]
Combining with \eqref{bu0} and rearranging we have
\[
\frac \dd{\dd t}\norm{u(t)}_2^2\leq 2 \norm{u(t)}^2_{2}+{8C^2}{}\norm{u(t)}_2^2\norm{u_2(t)}^2_{\dot H^1}=\norm{u(t)}_2^2\big(   2+{8C^2}\norm{u_2(t)}^2_{\dot H^1}      \big).
\]
Since $u(0)=0$, we are done by the Gronwall's inequality.

\section{Proofs of Corollaries}

\subsection{Conservation of energy: proof of Corollary \ref{ConserLemma}}\label{Briefsubsect}
Recall we wish to show
\be\label{conservE}
\|u(t,\cdot )\|_{L^2(\mathbb{H}^2(-a^2))}^2 + 2\int_{0}^{t} ((u,u))(s) \dd s = \|u_0\|_{L^2(\mathbb{H}^2(-a^2))}^2.
\ee
Now, with respect to a given initial data $u_0 \in \textbf{H}\oplus \mathbb{F}$, the unique solution $u \in L^{\infty}(0,T; \textbf{H}\oplus \mathbb{F} )\cap L^2(0,T; \mathbb{V} )$
 by \eqref{betterregularity} also satisfies
\begin{equation}
\partial_t u \in L^2(0,T; \mathbb{V}').
\end{equation}
This enables us to invoke Lemma \ref{anotherlemma}  to obtain
\begin{equation}\label{ONEONE}
\ip{\partial_t u(t) , u(t) }_{\mathbb{V}'\otimes \mathbb{V}} = \frac{1}{2} \frac{\dd }{\dd t} \|u(t)\|_{L^2(\mathbb{H}^2(-a^2))}^2.
\end{equation}
Thanks to \eqref{bp2} in Lemma \ref{ATWO}, we also have
\begin{equation}\label{TWO}
\ip{Bu(t) , u(t)}_{\mathbb{V}'\otimes \mathbb{V}} = b(u,u,u)(t) = 0, \forall t \in [0,T].
\end{equation}
So, by simply testing equation \eqref{FINALEQUATION} against $u$, it follows from \eqref{ONEONE} and \eqref{TWO} that
\begin{equation}
\frac{1}{2} \frac{\dd }{\dd t} \|u(t)\|_{L^2(\mathbb{H}^2(-a^2))}^2 + ((u,u))(t) = 0,
\end{equation}
form which we deduce, by taking the time integral from $0$ to $t$, that the desired identity, namely \eqref{conservE}, must hold for all $t \in [0,T]$.

\subsection{Survival of one solution: proof of Corollary \ref{survivalcor}}\label{confirm}
Let $u_{0}=\dd F\in L^2(\H)$ where $F$ is a harmonic function.  Now, the variational formulation says that $u$ is a solution if $u$ satisfies
\be\label{var1}
(u_{t},v)+(\Def u, \Def v)+b(u,u,v)=0
\ee
for every $v \in \V=\tV+\F$.  Now, consider $u=\exp(-\frac{A}{2}t)\dd F$, where $A$ is a constant.  Then $u(0)=u_{0}$.  

We show $u$ satisfies \eqref{var1} if and only if $A=4a^2$.   First,  since $(\Def u, \Def v)=(\Def^\ast \Def u, v)$,  we can write \eqref{var1} as
\be\label{var2}
(u_{t},v)-\int_{\H}g(\Delta u-2\Ric(u), v) +b(u,u,v)=0.
\ee
 Thanks to $\Delta u=0$ and $\Ric(u)=-a^{2}u$ we can  simplify LHS of \eqref{var2} to
\[
-\frac A2(u,v)+2a^{2}(u, v) +b(u,u,v).
\]
Consider $b(u,u,v)$.  By \cite[Lemma 6.1]{CC10}, $\nablab_{\nabla F}\dd F=\frac 12\dd \abs{\dd F}^2$, so by definition of $b$
\[
b(u,u,v)=\frac 12\exp(-A t)\int_{\H}g(\dd\abs{\dd F}^{2},v)=\frac12 \exp(- A t)\int_{\H}\dd^{\ast}(\abs{\dd F}^{2} v)=0
\]
if $\abs{\dd F}^{2}v$ in $L^{1}$, but that follows by Cauchy-Schwarz, $v \in \V$, Ladyzhenskaya and Lemma \ref{IMFinitedisspation}.  So we are left with needing to show
\[
-\frac A2(u,v)+2a^{2}(u, v)=0.
\]
Note, if $v \in \tV$, then this is automatically satisfied by the orthogonality property.  But in general, $v \in \tV+\F$, so the only way to guarantee that the equation is satisfied is to require $A=4a^{2}$.  This shows the survival of \emph{one} solution from the family of the non-unique solutions we have considered before. Moreover, this solution also gives \emph{equality} in the energy inequality  (can be seen by a direct computation or follows from Corollary  \ref{ConserLemma}).

\section{Pressure}\label{pressure}
The goal of this section is to show that the pressure can be recovered from the variational formulation.   More precisely we show Proposition \ref{pressuremainlemma}.  First, we collect some tools.

 In the classical literature (e.g. \cite{Lbook, Temam}) existence and regularity of the pressure is obtained usually  by means of the following lemma.
\begin{lemma}\label{p1}
Let $\Omega$ be any open set in   $\R^{n}$. If $w\in L^{2}_{loc}(\Omega)$, then $(w,v)=0$ for all
$v\in C^{\infty}_{c,\sigma}$ if and only if $w=\grad p$ for some $p\in  L^{2}_{loc}(\Omega)$  with $\grad p\in  L^{2}_{loc}(\Omega)$.
\end{lemma}
We would like to establish an analog of this in our setting.  First, as pointed out in \cite[p.10]{Temam}, existence of the pressure can follow from
\begin{thm}\cite[Theorem 17']{DeRhamEng}\label{DeRham}
The current $T$ is homologous to zero if and only if $T[\phi]=0$ for all closed $C^\infty$ forms with compact support.
\end{thm}
We translate this theorem  into the language of the fluid mechanics.
The space of currents can be viewed as the dual to the space of differential forms.  More precisely
\begin{defn} [Currents] \cite[p.34]{DeRhamEng} Let $M$ be an $n-$dimensional manifold, and $\Lambda^k_c(M)$ denote smooth $k$-forms that are compactly supported in $M$.  Then the current $T$ is a linear functional on $\Lambda^k_c(M)$, with the action denoted by
\[
T[\phi],\quad \phi\in \Lambda^k_c(M).
\]
\end{defn}
A relevant example is an analog of $f \in L^1_{loc}$ giving a rise to a distribution:  if $\alpha$ is a locally integrable $(n-k)$-form, we can introduce
\be\label{Talpha}
T_\alpha[\phi]=\int_M \alpha \wedge \phi.
\ee
Hence, sometimes we write $\alpha[\phi]$ to denote \eqref{Talpha}.
Before we define what it means to be homologous to zero, we need
\begin{defn}[Boundary of a current] \cite[p.45]{DeRhamEng} The boundary of a current $T$, denoted by $\bb T$, is a current defined by
\[
\bb T[\phi]=T[\dd\phi].
\]
\end{defn}
Then the operator $\dd$ on currents is given by \be\label{currentd}\dd \theta = (-1)^{k+1} \bb \theta,\ee
with $k$ to be the degree of the current $\theta$ in question.  For future reference, we also have
\begin{align}
\alpha[\ast \beta]&=\beta \wedge \ast \alpha=(-1)^{k(n-k)}\ast \alpha\wedge \beta=(-1)^{k(n-k)}\ast \alpha[\beta].\label{currentp1}
\end{align}
\begin{defn}[Homologous to zero] \cite[p.79]{DeRhamEng} A current $T$ is homologous to zero if there exists a current $S$ such that $T=\bb S$.
\end{defn}
We can now prove the distributional analog of Lemma \ref{p1} on any manifold $M$ for which Theorem \ref{DeRham} is valid.  We address the regularity part in the subsequent lemma.
\begin{lemma}\label{p2}
 Let $T_{w}$ be a current that corresponds to a locally integrable $1-$form $w$ as given by \eqref{Talpha}.
Then $(w,v)=0$ for all
$v \in \Lambda^1_{c,\sigma}(M)$ if and only if $T_w=\dd P$ for some current $P$.
\end{lemma}
\begin{proof}
Recall the definition of the Hodge $\ast$ operator \eqref{star}
\[
w\wedge \ast v =g(w,v)\Vol_M.
\]
Then
\be\label{p2b}
(w,v)=\int_M g(w,v) \Vol_M=\int_M w \wedge \ast v \Vol_M=T_w[\ast v].
\ee
Next, if $v \in  \Lambda^1_{c,\sigma}(M)$, then $\dd^{\ast}v=0$, so $\ast v$ is an $n-1$-form, and it is closed  by the definition of $\dd^{\ast}$ (see \eqref{star2}).  Conversely, if $\phi$ is $n-1$-form that is closed, then $\ast \phi$ is a $1$-form such that $\dd^{\ast}\ast \phi=0$.  Hence from \eqref{p2b} we have
\[
(w,v)=0\quad \forall v \in \Lambda^1_{c,\sigma} \quad \mbox{if and only if}\quad T_w[\phi]=0\quad  \forall \ \ \mbox{closed forms}\ \ \phi  \in \Lambda^{n-1}_c(M).
\]
So by Theorem \ref{DeRham}, $T_w=\dd P$ if and only if $(w,v)=0$ for all
$v \in \Lambda^1_{c,\sigma}(M)$ as needed.

\end{proof}

\begin{lemma}\label{fortheSakeofpressure}
Let $T \in H^{-1}(\mathbb{H}^2(-a^2))$. Suppose there exists a current $P$ of degree $0$  such that  $\dd P = T$ holds in the sense that
\begin{equation}\label{gradexpress}
\dd P [v] = \langle T , v  \rangle_{H^{-1} \otimes H^1_0}, \ \forall v \in \Lambda^1_{c}(\H).
\end{equation}
Then, it follows that $P \in L^2_{loc}(\mathbb{H}^2(-a^2))$.
\end{lemma}

\begin{proof}
Recall that $H^{-1}(\mathbb{H}^2(-a^2))$ is by definition the dual space of $H^1_{0}(\mathbb{H}^2(-a^2))$ .

Now, since $T \in H^{-1}(\mathbb{H}^2(-a^2))$, and $H^1_0(\mathbb{H}^2(-a^2))$ is Hilbert  we can use the Riesz Representation Theorem to deduce that there exists a unique $1$-form $\eta \in H^1_0(\mathbb{H}^2(-a^2))$ for which the following holds for all  $v \in H^1_0(\mathbb{H}^2(-a^2))$
\begin{equation}\label{repofHinverse}
\ip{ T , v }_{H^{-1} \otimes H^1_0} = \int_{\mathbb{H}^2(-a^2)} g(\eta, v) \Vol_{\mathbb{H}^2(-a^2)} + 2\int_{\mathbb{H}^2(-a^2)} \overline{g}(\Def \eta , \Def v) \Vol_{\mathbb{H}^2(-a^2)} .
\end{equation}
Now, let $v \in \Lambda^1_{c}(\H)$.  By means of the formula $2\Def^* \Def v = 2\dd \dd^*v + \dd^*\dd v + 2a^2 v $, \eqref{repofHinverse} can be rephrased as
\begin{equation}\label{repofHinversetwo}
\begin{split}
\ip{T , v }_{H^{-1} \otimes H^1_0} & = \int_{\mathbb{H}^2(-a^2)} g(\eta, v) \Vol_{\mathbb{H}^2(-a^2)} + 2\int_{\mathbb{H}^2(-a^2)}  g(\eta ,\Def^* \Def v) \Vol_{\mathbb{H}^2(-a^2)} \\
& = (1 + 2a^2) \int_{\mathbb{H}^2(-a^2)} g(\eta, v) \Vol_{\mathbb{H}^2(-a^2)} + 2 \int_{\mathbb{H}^2(-a^2)} g(\eta , \dd \dd^* v) \Vol_{\mathbb{H}^2(-a^2)} \\
&\quad + \int_{\mathbb{H}^2(-a^2)} g (\eta , \dd^* \dd v) \Vol_{\mathbb{H}^2(-a^2)}\\
& = (1+ 2a^2) \int_{\mathbb{H}^2(-a^2)} g(\eta, v) \Vol_{\mathbb{H}^2(-a^2)} + 2 \int_{\mathbb{H}^2(-a^2)} \dd^* \eta \dd^* v \Vol_{\mathbb{H}^2(-a^2)} \\
&\quad + \int_{\mathbb{H}^2(-a^2)}  \overline{g}(\dd \eta , \dd v) \Vol_{\mathbb{H}^2(-a^2)}.
\end{split}
\end{equation}
In \eqref{repofHinversetwo}, the first identity follows directly from the definition of $\Def \eta \in L^2$ in the weak sense. The second equal sign also holds, since Lemma \ref{TrivialLemma} informs us that $\Def \eta \in L^2$  leads to the existence of the weak derivatives $\dd \eta \in L^2$ and $\dd^*\eta \in L^2$.

Next,
 we can express the three individual terms which appear in the last line of \eqref{repofHinversetwo} by means of the language of currents as
\begin{equation}\label{USEFUL}
\begin{split}
\int_{\mathbb{H}^2(-a^2)} g(v, \eta) \Vol_{\mathbb{H}^2(-a^2)} & =\int_{\H} v\wedge \ast \eta= -*\eta [v] ,\\
\int_{\mathbb{H}^2(-a^2)} \dd^* \eta \dd^* v \Vol_{\mathbb{H}^2(-a^2)}  & = \dd^* \big( * \dd^* \eta \big )[v] ,\\
\int_{\mathbb{H}^2(-a^2)}  \overline{g}(\dd \eta , \dd v) \Vol_{\mathbb{H}^2(-a^2)} & = -\dd\big ( *\dd \eta \big ) [v] .
\end{split}
\end{equation}

Now, suppose that we have a current $P$ such that $\dd P = T$ holds in the sense of  \eqref{gradexpress}. Then, it follows directly from \eqref{repofHinversetwo} and \eqref{USEFUL} that the following identity holds in the sense of currents  \begin{equation}\label{dpequation}
\dd P = -(1+2a^2) * \eta + 2 \dd^* \big ( * \dd^* \eta \big ) - \dd \big ( * \dd \eta \big ).
\end{equation}
It is clear that $*\eta, *\dd^* \eta, * \dd\eta \in L^2(\mathbb{H}^2(-a^2))$.
Next, by the Hodge-Kodaira Theorem \cite{Kodaira}, we have
\[
L^2(\H)=\overline{\dd\Lambda^0_c}^{L^2}\oplus \overline{\dd^\ast \Lambda^2_c}^{L^2}\oplus \F.
\]
Since $*\eta$ is in $L^2$, we have the following unique decomposition
\begin{equation}\label{HodgeDecomp}
* \eta = \dd \alpha_{1} + \dd^* \alpha_2 + \dd F ,
\end{equation}
where $\dd \alpha_{1} \in L^2(\mathbb{H}^2(-a^2))$, where $\alpha_{1}$ can be shown to be in $L^2_{loc}$ on $\H$ (for example by a similar method employed in the proof of Lemma \ref{Gradientspacecomplete}), $\dd^* \alpha_2 \in L^2(\mathbb{H}^2(-a^2))$ with $\alpha_2$ to be a current of degree $2$ on $\mathbb{H}^2(-a^2)$, and $\dd F \in L^2(\mathbb{H}^2(-a^2))$ with $F$ to be a harmonic function on $\mathbb{H}^2(-a^2)$. So, \eqref{dpequation} can be rephrased as
\begin{equation}\label{GooddPexp}
\dd P = \dd \big ( -* \dd\eta - (1+2a^2) (\alpha_1 + F)\big ) + \dd^* \big ( 2 * \dd^* \eta - (1+2a^2) \alpha_2 \big ),
\end{equation}
which holds in the sense of currents.
Since $\dd \circ \dd = 0$ holds for all currents, by taking  $\dd$ on both sides of the above identity, we have
\begin{equation}\label{KeyHarmonic}
0 = \dd \dd^* \big ( 2 * \dd^* \eta - (1+2a^2) \alpha_2 \big ) = (-\triangle ) \big ( 2 * \dd^* \eta - (1+2a^2) \alpha_2 \big ).
\end{equation}
Hence the current $( 2 * \dd^* \eta - (1+2a^2) \alpha_2 \big )$ of the top degree $2$ is harmonic.  It follows by the Weyl's Lemma that $( 2 * \dd^* \eta - (1+2a^2) \alpha_2 \big )$ must be a smooth harmonic $2$-form on $\mathbb{H}^2(-a^2)$ in the classical sense, and hence $\dd^* \big ( 2 * \dd^* \eta - (1+2a^2) \alpha_2 \big )$ is indeed a smooth $1$-form on $\mathbb{H}^2(-a^2)$. Also, observe that \eqref{KeyHarmonic} tells us that the now classically smooth $1$-form  $\dd^* \big ( 2 * \dd^* \eta - (1+2a^2) \alpha_2 \big )$ is $\dd$-closed on $\mathbb{H}^2(-a^2)$. So, it must be $\dd$-exact in the sense that there exists a smooth function $F_2$ on $\mathbb{H}^2(-a^2)$ such that
\begin{equation}\label{dexact}
\dd^* \big ( 2 * \dd^* \eta - (1+2a^2) \alpha_2 \big ) = \dd F_2 .
\end{equation}
So, it directly follows from \eqref{GooddPexp} and \eqref{dexact} that we have the following relation
\begin{equation}
\dd P = \dd \big (  -* \dd\eta - (1+2a^2) (\alpha_1 + F) + F_2    \big ) ,
\end{equation}
which immediately tells us that the current $P$ should be the same as $ -* \dd\eta - (1+2a^2) (\alpha_1 + F) + F_2 $ up to some constant $C_0 \in \mathbb{R}$. That is, we have
\begin{equation}
P =-* \dd\eta - (1+2a^2) (\alpha_1 + F) + F_2  + C_0 ,
\end{equation}
in which $* \dd\eta$ and $\alpha_1$ are $L^2$ and $L^2_{loc}$ functions on $\mathbb{H}^2(-a^2)$ respectively and $F$ and $F_2$ are classically smooth functions on $\mathbb{H}^2(-a^2)$. Hence, it follows that the current $P$ is indeed an $L^2_{loc}$ function on $\mathbb{H}^2(-a^2)$ as desired.
\end{proof}
We are now ready to prove Lemma \ref{pressuremainlemma}.
\begin{proof}[Proof o Proposition  \ref{pressuremainlemma}]
Here, the basic idea is: in order to recover the pressure term, it is absolutely necessary for us to think of the elements $Au$ and $Bu$, which already lie in $L^2(0,T ; \mathbb{V}')$, to be in the more restrictive functional space $L^2(0,T ; H^{-1}(\mathbb{H}^2(-a^2)) )$.   This follows from estimates \eqref{A2} and \eqref{b4}.

Now, recall that the existence theory ensures, for any prescribed initial data $u_0 \in \textbf{H}\oplus \mathbb{F}$, the existence of an element $u \in L^{\infty}(0,T ; \textbf{H}\oplus \mathbb{F} )\cap L^2(0,T ; \mathbb{V})$, with $\partial_t u \in L^2(0,T ; \mathbb{V}')$.  Then the following relation holds in the sense of $L^2(0,T; \mathbb{V}')$
\begin{equation}\label{Restrictiveidentity}
\partial_t u + Au + Bu = 0.
\end{equation}
Next, consider the terms $U \in C^0(0,T ; \mathbb{V})$ and $\textbf{B} \in C^0(0,T ; \mathbb{V}')$, which are defined by
\begin{equation}
\begin{split}
U(t) & = \int_0^t u(\tau ) \dd \tau , \\
\textbf{B}(t) & = \int_0^t Bu(\tau) \dd \tau .
\end{split}
\end{equation}
Of course, we at once know that $AU \in C^0(0,T; H^{-1}(\mathbb{H}^2(-a^2)))$. On the other hand, $\partial_t u \in L^2(0,T ; \mathbb{V}')$
immediately gives the following relation for every $t \in [0,T] $
\begin{equation}\label{Vanishingfunctional}
u(t) - u_0 + AU(t) + \textbf{B}(t) = 0 ,
\end{equation}
in which each term on the left-hand side is understood to be an element in $\mathbb{V}'$. 

However, observe that we can think of the expression $u(t) - u_0 + AU(t) + \textbf{B}(t)$ in the \emph{broader sense}, as a bounded linear functional on $H^1_0(\mathbb{H}^2(-a^2))$, which acts by integrating against an element $v \in H^1_0$. Now, with this broader sense of understanding  we see  \eqref{Vanishingfunctional} as saying that the bounded linear functional $u(t) - u_0 + AU(t) + \textbf{B}(t)$ on $H^1_0(\mathbb{H}^2(-a^2))$ actually vanishes identically on the proper subspace $\mathbb{V}$, which by itself includes $\Lambda_{c, \sigma}^1(\mathbb{H}^2(-a^2))$. This observation allows us to invoke Lemma \ref{p2} and Lemma \ref{fortheSakeofpressure} to deduce that there exists, for each $t\in [0,T]$, a  $0$-current $\textbf{P}(t) \in L^2_{loc}(\mathbb{H}^2(-a^2))$ such that the following relation holds in the sense of $H^{-1}(\mathbb{H}^2(-a^2))$
\begin{equation}\label{FinalGood}
u(t) - u_0 + A U(t) + \textbf{B}(t) = d \textbf{P}(t).
\end{equation}
In other  words, it means the same as saying that equation \eqref{FinalGood} holds in the weak sense as long as we test against  an arbitrary test $1$-form $v \in H^1_0(\mathbb{H}^2(-a^2))$.

To obtain \eqref{eqwithp}, we differentiate \eqref{FinalGood} in the sense of distributions and set $p$ to be the \emph{distributional} time derivative of $\textbf P$, $p=\partial_t \textbf P$.
\end{proof}

\section{$\V=\widetilde {\tV}$}\label{EVERYTHINGSECT}
Here we establish Theorem \ref{everything}.  By definition we have $\V\subset \widetilde {\tV}$.  So we just need to show $\widetilde {\tV}\subset \V $.  The starting point is the Hodge-Kodaira decomposition \cite{Kodaira}
\[
L^2(\H)=\overline{\dd\Lambda^0_c}^{L^2}\oplus \overline{\dd^\ast \Lambda^2_c}^{L^2}\oplus \F.
\]
Let $\alpha \in \widetilde{\tV}$.  Since $\widetilde{\tV}\subset H^1_0\subset L^2$, we have
\[
\alpha=\dd \alpha_1+\dd^\ast \alpha_2 + \dd F,
\]
where $\dd \alpha_1 \in \overline{\dd\Lambda^0_c}^{L^2}, \dd^\ast \alpha_2 \in \overline{\dd^\ast \Lambda^2_c}^{L^2}, \dd F\in \F.$  Next,  $\dd^\ast \alpha=0$ implies
\[
0=\dd^\ast \dd \alpha_1+\dd^\ast \dd^\ast \alpha_2 + \dd^\ast \dd F=\dd^\ast \dd \alpha_1,
\]
because $\dd^\ast \dd^\ast=0$ and $F$ is harmonic.  This means $\alpha_1$ must be harmonic, so $\alpha_1=0$ by orthogonality.  Hence
\[
\alpha=\dd^\ast \alpha_2 + \dd F,
\]
and by Lemma \ref{IMFinitedisspation}, it suffices to show $\dd^\ast \alpha_2 \in \tV$.  Because $\dd^\ast \alpha_2 \in \overline{\dd^\ast \Lambda^2_c}^{L^2}$, by definition $\dd^\ast \alpha_2 =\lim_{k\rightarrow \infty} \dd^* \beta_k$, where $\beta_k$ is a sequence of smooth $2$-forms  with compact support, and the limit is in $L^2$.  To show $\dd^\ast \alpha_2 \in \tV$ we need to find a sequence   of smooth $1$-forms with compact support that converges in the $H^1$ norm.
\begin{remark}
It is very tempting to let  the sequence be $\dd^*\beta_k$ since then it  is divergence free and smooth with compact support.  However, then we only would have $L^2$ convergence guaranteed and not $H^1$.  On the other hand, since $\dd^\ast \alpha_2 \in \V\subset H^1_0$, there is a sequence of smooth compactly supported $1$-forms that converges to it in $H^1$, but we do not know if the sequence is divergence free.  So it looks like we can have one but not the other.  We combine the two approaches below to get the needed one.
\end{remark}
Let $\psi_{R} : [0,\infty ) \rightarrow \mathbb{R}$, which satisfies $\chi_{[0,R]} \leq \psi_{R} \leq \chi_{[0,2R]}$, and $|\psi_{R}'| \leq \frac{1}{R}$. Let $O$ be a selected base point in $\H$, and let $\rho$ be the distance function from $O$.  Let $\epsilon>0$.  Since $\dd^\ast\alpha_2 \in H^1_0$, we can choose $R$ large enough so that
\be\label{farrange}
\norm{\dd^\ast(\psi_{R}(\rho)\alpha_2)-\dd^*\alpha_2}_{H^1}< \epsilon.
\ee
Next $\alpha_2$ is a $2$-form on $\H$, so it can be written as $\alpha_2=f\Vol_{\H}$, where $f$ is a function.  Then observe
\be\label{needlater2}
 H^1_0\ni \dd^\ast \alpha_2=-\ast \dd \ast \alpha_2=-\ast \dd f,
\ee
so $\dd f \in H^1_0$, and by Corollary \ref{equivH1}
\[
\norm{\dd f}_{L^2}+\norm{\overline{\nabla}\dd f}_{L^2}<\infty.
\]
Moreover just like before, we could show $f\in L^2_{loc}(\H)$.  Then we have that $\psi_R f \in H_0^2(\H)$, and therefore there exists a sequence of smooth functions $\eta_k$ with compact support such that
\be\label{needlater}
\norm{\eta_k-\psi_R f}_{L^2}+\norm{d(\eta_k-\psi_Rf)}+\norm{\overline{\nabla}d(\eta_k-\psi_Rf)}_{L^2}\rightarrow 0,\mbox{as}\ k\rightarrow \infty.
\ee
Then let $ \omega_k=\eta_k \Vol_{\H}$ and consider
\[
\norm{\dd^\ast\omega_k-\dd^\ast \alpha_2}_{H^1}\leq \norm{\dd^\ast\omega_k-\dd^\ast (\psi_R\alpha_2)}_{H^1}+\norm{\dd^\ast(\psi_{R}\alpha_2)-\dd^*\alpha_2}_{H^1},
\]
so by \eqref{farrange} we just need to estimate the first term.  By same simplification as in \eqref{needlater2} we have
\begin{align*}
\norm{\dd^\ast\omega_k-\dd^\ast (\psi_R\alpha_2)}_{H^1}=\norm{\ast \dd \eta_k-\ast\dd (\psi_R f)}_{H^1}\lesssim\norm{\dd \eta_k-\dd (\psi_R f)}_{H^1} \rightarrow 0,\mbox{as}\ k\rightarrow \infty,
\end{align*}
by \eqref{needlater}.
\appendix

\section{Computations in coordinates}\label{appendixa}

\subsection{Hyperboloid model}
We first give a concrete description of the space form $\mathbb{H}^2(-a^2)$ by means of the standard hyperboloid model.

The $2$-dimensional hyperbolic space $\mathbb{H}^2(-a^2)$, as a differentiable manifold, can be regarded as a $2$D submanifold in $\mathbb{R}^3$ given by

\begin{equation}
\mathbb{H}^2(-a^2) = \{ (x_0, x_1 , x_2 ) : x_0^2 - x_1^2 - x_2^2 = \frac{1}{a^2}  \}.
\end{equation}
 
Next, for each point $x = (x_0, x_1, x_2) \in \mathbb{R}^3$, the tangent space $T_x\mathbb{R}^3$ is equipped with the following symmetric quadratic form

\begin{equation}\label{lorentz}
\ip{v, w} = -v_0w_0 + v_1w_1 + v_2w_2 ,  \quad v,w \in T_x\mathbb{R}^3.
\end{equation}
So, by definition, the Riemannian metric $g(\cdot ,\cdot )$ on $\mathbb{H}^2(-a^2)$ is induced through the restriction of $\ip{\cdot ,\cdot}$ onto the tangent bundle of the submanifold $\H$.
In other words, for each point $x \in \mathbb{H}^2(-a^2)$, $g(\cdot , \cdot )_x$ is given by the following relation

\begin{equation}
g(\cdot , \cdot )_{x} = \ip{\cdot , \cdot }\big |_{x} .
\end{equation}

From now on, a point $x = (x_0, x_1, x_2)$ in $\mathbb{R}^3$ will be written as $x = (x_0 , x')$, with $x' = (x_1 , x_2)$.

\subsection{Local coordinates}
 
Here, we consider the unit disc $D_{0}(1) = \{y \in \mathbb{R}^2 : |y| <1 \}$ in $\mathbb{R}^2$ and the smooth map $Y : \mathbb{H}^2(-a^2) \rightarrow D_{0}(1)$ which is defined by
\begin{equation}\label{PoincareDisc}
Y (x) = \frac{x'}{x_0 + \frac{1}{a}}, \quad x \in \H.
\end{equation}

The smooth map $Y : \mathbb{H}^2(-a^2) \rightarrow D_{0}(1)$ maps $\mathbb{H}^2(-a^2)$ bijectively onto $D_0(1)$ with a smooth inverse. Hence, $Y$ can be chosen as a coordinate system on the manifold $\mathbb{H}^2(-a^2)$ (with one chart). This coordinate system is standard. See for example \cite[Ex 6, p.83]{Jost}.

Observe that the inverse map $Y^{-1} : D_0(1) \rightarrow \mathbb{H}^2(-a^2)$ is given by

\begin{equation}
Y^{-1} (y) = \bigg ( \frac{2}{a(1-|y|^2)} -\frac{1}{a}  ,  \frac{2y_1}{a(1- |y|^2)}  ,  \frac{2y_2}{a(1-|y|^2)}        \bigg ), \quad   y\in D_0(1).
\end{equation}

Next, we express the Riemannian metric $g(\cdot , \cdot )$ on $\mathbb{H}^2(-a^2)$ in terms of the coordinate system $Y$ as follows. Consider the two smooth vector fields $\frac{\partial}{\partial Y^1}$, and $\frac{\partial}{\partial Y^2}$ on $\mathbb{H}^2(-a^2)$, which are induced by the coordinate system $Y$ through the following rule, with $j = 1,2$, and any $y\in D_0(1)$

\begin{equation}
\frac{\partial}{\partial Y^j} \bigg |_{Y^{-1}(y)}  = \bigg ( \partial_{y_j} \big ( \frac{2}{a(1-|y|^2)} -\frac{1}{a} \big )  ,  \partial_{y_j} \big ( \frac{2y_1}{a(1- |y|^2)} \big ) ,    \partial_{y_j} \big ( \frac{2y_2}{a(1-|y|^2)} \big )        \bigg ).
\end{equation}

Then, by a direct computation

\begin{equation}
\begin{split}
\frac{\partial}{\partial Y^1} \bigg |_{Y^{-1}(y)} & = \bigg (  \frac{4y_1}{a(1-|y|^2)^2} , \frac{2(1-|y|^2) + 4y_1^2}{a(1-|y|^2)^2} ,
\frac{4y_1y_2}{a(1-|y|^2)^2}                         \bigg ) , \\
\frac{\partial}{\partial Y^2} \bigg |_{Y^{-1}(y)} & = \bigg ( \frac{4y_2}{a(1-|y|^2)^2} ,  \frac{4y_1y_2}{a(1-|y|^2)^2}  ,  \frac{2(1-|y|^2) + 4y_2^2}{a(1-|y|^2)^2}                     \bigg ),
\end{split}
\end{equation}

from which it follows that

\begin{equation}
\begin{split}
& g( \frac{\partial}{\partial Y^1} , \frac{\partial}{\partial Y^1}   ) = \frac{4}{a^2(1-|y|^2)^2} = g( \frac{\partial}{\partial Y^2} ,
 \frac{\partial}{\partial Y^2}   ) \\
&  g( \frac{\partial}{\partial Y^1} , \frac{\partial}{\partial Y^2}  ) = 0 .
\end{split}
\end{equation}
Hence,
\begin{equation}\label{metricExpress}
g(\cdot , \cdot ) = \frac{4}{a^2(1-|y|^2)^2} \bigg \{ \dd Y^1\otimes \dd Y^1 + \dd Y^2 \otimes dY^2 \bigg \} .
\end{equation}

Next let
\begin{equation}\label{OrthFrame}
e_j = \frac{a(1-|Y|^2)}{2} \frac{\partial}{\partial Y^j}, \quad j=1, 2.
\end{equation}
Observe that the two smooth vector fields $\{e_1, e_2\}$ constitute a globally defined orthonormal moving frame on $\mathbb{H}^2(-a^2)$, which specifies an orientation on $\mathbb{H}^2(-a^2)$.

Then, the induced dual frame $\{e_1^*, e_2^* \}$ is given by the following expression
\begin{equation}\label{Dualframe}
e_j^* = \frac{2}{a(1-|Y|^2)} \dd Y^j, \quad j=1, 2.
\end{equation}
Notice that the induced Riemannian metric $\overline{g}$ on $T^*(\mathbb{H}^2(-a^2))$ is the one with respect to which the dual frame $\{e_1^*, e_2^* \}$ becomes everywhere orthonormal. Hence, we have

\begin{equation}\label{Gupperscript}
\begin{split}
& g^{11} = \overline{g} (\dd Y^1,\dd Y^1) = \frac{a^2(1-|y|^2)^2}{4} = \overline{g} (\dd Y^2,\dd Y^2) = g^{22} , \\
& g^{12} = \overline{g} (\dd Y^1 , \dd Y^2 ) = 0 .
\end{split}
\end{equation}

Moreover, due to \eqref{metricExpress}, the $2\otimes  2$ matrix $\big ( g_{ij}  \big )$, with
$g_{ij} = g(\frac{\partial}{\partial Y^i}  , \frac{\partial}{\partial Y^j} ) $, is a diagonal matrix with $g_{11} = g_{22} = \frac{4}{a^2(1-|y|^2)^2}$. Hence we have $\sqrt{\det \big ( g_{ij}  \big )} = \frac{4}{a^2 (1-|y|^2)^2}$.

It follows, the Hodge-Laplacian $-\triangle = d^*d : C^{\infty}(\mathbb{H}^2(-a^2)) \rightarrow C^{\infty}(\mathbb{H}^2(-a^2)) $ can be represented in terms of the coordinate system $Y$ in the following way

\begin{equation}\label{LaplaceManifold}
\begin{split}
\triangle  f & =  \frac{1}{\sqrt{\det \big (g_{ij}  \big ) }} \frac{\partial}{\partial Y^{\alpha}}
\left( \sqrt{\det \big (g_{ij}  \big ) }  g^{\alpha \beta }  \frac{\partial}{\partial Y^{\beta}} f  \right)\\
& = \frac{a^2(1-|y|^2)^2}{4} \bigg ( \big ( \frac{\partial}{\partial Y^1} \big )^2 f  + \big ( \frac{\partial}{\partial Y^2} \big )^2 f  \bigg ) ,
\end{split}
\end{equation}

where $f$ is any smooth function on $\mathbb{H}^2(-a^2)$.  It follows that a given smooth function $F \in C^{\infty}(\mathbb{H}^2(-a^2))$ is a harmonic function on $\mathbb{H}^2(-a^2)$ if and only if $f = F \circ Y^{-1}$ is a harmonic function on the \emph{Euclidean} disc $D_{0}(1)$ in the ordinary sense.

For convenience, we will use the symbol $\nabla^{\mathbb{R}^2}$ to denote the standard gradient operator $\nabla^{\mathbb{R}^2}f = (\partial_{y_1}f  , \partial_{y_2}f )$ on $\mathbb{R}^2$. The use of this symbol $\nabla^{\mathbb{R}^2}$ is to avoid possible confusion with the gradient operator $\nabla$ on $\mathbb{H}^2(-a^2)$. Now, we observe that, for any smooth function $F$ on $\mathbb{H}^2(-a^2)$, we have
\begin{equation}\label{gradF1}
\nabla F = g( \nabla F ,e_1  ) e_1 + g (\nabla , e_2) e_2 ,
\end{equation}
where
\begin{equation}\label{gradF2}
g(\nabla F , e_j) = \ip{\dd F , \frac{a(1-|Y|^2)}{2} \frac{\partial}{\partial Y^j} }_{T\H^\ast\otimes  T\H} = \frac{a(1-|Y|^2)}{2} \frac{\partial F}{\partial Y^j}, \ j=1,2 .
\end{equation}
Hence, the following identity holds for any smooth function $F$ on $\mathbb{H}^2(-a^2)$ with $\dd F \in L^2(\mathbb{H}^2(-a^2))$.
\begin{equation}\label{writtenonDisc}
\begin{split}
\|\dd F\|^2_{L^2(\H)} & = \int_{\mathbb{H}^2(-a^2)} |\nabla F|^2 e_1^*\wedge e_2^* \\
& = \int_{\mathbb{H}^2(-a^2)} \bigg ( \frac{a(1-|y|^2)}{2} \bigg )^2 \bigg \{ \big (\frac{\partial F}{\partial Y^1} \big )^2 +  \big (\frac{\partial F}{\partial Y^2} \big )^2  \bigg \} \frac{4}{a^2(1-|y|^2)^2} \dd Y^1 \wedge \dd Y^2 \\
& = \int_{\mathbb{H}^2(-a^2)} \bigg \{ \big (\frac{\partial F}{\partial Y^1} \big )^2 +  \big (\frac{\partial F}{\partial Y^2} \big )^2  \bigg \} \dd Y^1 \wedge \dd Y^2 \\
& = \int_{D_0(1)} |\nabla^{\mathbb{R}^2} (F\circ Y^{-1})|^2 \dd y_1 \dd y_2 .
\end{split}
\end{equation}

\subsection{The Levi-Civita connection on $\mathbb{H}^2(-a^2)$.  }

We begin by obtaining a representation of the Levi-Civita connection $\overline{\nabla}$ acting on the space of smooth vector fields over $\mathbb{H}^2(-a^2)$ in terms of the natural coordinate system $Y = (Y^1 , Y^2): \mathbb{H}^2(-a^2) \rightarrow D_O(1) $ as introduced in \eqref{PoincareDisc}. In other words, we first calculate the Christoffel symbols $\Gamma_{ij}^k$ (for $1\leq i,j,k \leq 2$) which appears in the following representation formula.
\begin{equation}
\overline{\nabla} \frac{\partial}{\partial Y^i} = \Gamma_{ij}^k dY^j \otimes \frac{\partial}{\partial Y^k} .
\end{equation}
Indeed, in accordance with basic Riemannian geometry, we have the following useful formula for $\Gamma_{ij}^k$, where $g_{ij} = g(\frac{\partial}{\partial Y^i} , \frac{\partial}{\partial Y^j} )$, and $g^{ij}$ are specified in \eqref{Gupperscript}.
\begin{equation}
\Gamma_{ij}^k = \frac{1}{2} g^{k\alpha} \bigg \{ \frac{\partial g_{i\alpha } }{\partial Y^j} + \frac{\partial g_{j\alpha}}{\partial Y^i} - \frac{\partial g_{ij}}{\partial Y^\alpha}      \bigg \} .
\end{equation}
By means of the above formula, we  obtain
\begin{equation}\label{TABLE}
\begin{split}
\Gamma_{11}^1 & = \frac{2Y^1}{1-|Y|^2};\quad\Gamma_{12}^1  = \frac{2Y^2}{1-|Y|^2} = \Gamma_{21}^1, \\
\Gamma_{11}^2 & = \frac{-2Y^2}{1-|Y|^2} ; \quad \Gamma_{12}^2 = \frac{2Y^1}{1-|Y|^2} = \Gamma_{21}^2 , \\
\Gamma_{22}^1 & = \frac{-2Y^1}{1-|Y|^2} ; \quad \Gamma_{22}^2 = \frac{2Y^2}{1-|Y|^2} .
\end{split}
\end{equation}
 
In accordance with the information as provided in \eqref{TABLE}, we immediately get:
\begin{equation}\label{concreteCalculation}
\begin{split}
\overline{\nabla}\frac{\partial}{\partial Y^1} & = \frac{2}{1-|Y|^2} \bigg \{ Y_1 dY^1 \otimes \frac{\partial}{\partial Y^1} + Y_2 dY^2 \otimes \frac{\partial}{\partial Y^1}  - Y_2 dY^1 \otimes \frac{\partial}{\partial Y^2} + Y_1 dY^2 \otimes \frac{\partial}{\partial Y^2}   \bigg \} ,\\
\overline{\nabla}\frac{\partial}{\partial Y^2} & = \frac{2}{1-|Y|^2} \bigg \{ Y_2 dY^1 \otimes \frac{\partial}{\partial Y^1} - Y_1 dY^2 \otimes \frac{\partial}{\partial Y^1}  + Y_1 dY^1 \otimes \frac{\partial}{\partial Y^2} + Y_2 dY^2 \otimes \frac{\partial}{\partial Y^2}   \bigg \}.
\end{split}
\end{equation}
However, since it is plain to see that $dY^i \otimes \frac{\partial}{\partial Y^j} = e_i^* \otimes e_j$, where $\{e_1, e_2\}$ is the orthonormal frame on $\mathbb{H}^2(-a^2)$ as specified in \eqref{OrthFrame}, and $\{e_i^* , e_2^*\}$ is the respective orthonormal dual frame as specified in \eqref{Dualframe}. Since the induced Riemannian metric on the vector bundle $T^*\mathbb{H}^2(-a^2) \otimes \mathbb{H}^2(-a^2)$ is exactly the one with respect to which the frame $\{ e_i^* \otimes e_j : 1 \leq i,j \leq 2 \}$ is orthonormal. Due to these observations, it follows directly from \eqref{concreteCalculation} that we must have
\begin{equation}\label{simpleCalculation}
\bigg | \overline{\nabla}\frac{\partial}{\partial Y^1} \bigg | = \frac{2\cdot 2^{\frac{1}{2}}|Y|}{1-|Y|^2} = \bigg | \overline{\nabla}\frac{\partial}{\partial Y^2} \bigg | .
\end{equation}

\subsection{$\Def u$ in coordinates}\label{defudefuast}
$\overline{\nabla}$ is expressed in terms of the coordinate system $Y =(Y^1, Y^2)$ specified in \eqref{PoincareDisc} in the following way.  First,
\begin{equation}
\overline{dY^i} = - \Gamma_{jk}^i dY^k \otimes dY^j ,
\end{equation}
with $\Gamma_{jk}^i$ to be the Christoffel symbols as given in \eqref{TABLE}. So, we have the following straightforward expressions
\begin{equation}
\begin{split}
\overline{\nabla} \dd Y^1 & = \frac{2}{1-|Y|^2} \bigg \{ -Y^1 \dd Y^1\otimes \dd Y^1 -Y^2 \dd Y^1\otimes \dd Y^2 -Y^2 \dd Y^2\otimes \dd Y^1 + Y^1 \dd Y^2\otimes \dd Y^2     \bigg \} , \\
\overline{\nabla} \dd Y^2 & = \frac{2}{1-|Y|^2} \bigg \{ Y^2 \dd Y^1\otimes \dd Y^1 -Y^1 \dd Y^1\otimes \dd Y^2 -Y^1 \dd Y^2\otimes \dd Y^1 - Y^2 \dd Y^2\otimes \dd Y^2     \bigg \} .
\end{split}
\end{equation}
Now, for an arbitrary smooth $1$-form $u = u_1 dY^1 + u_2 dY^2$ on $\mathbb{H}^2(-a^2)$, $\overline{\nabla} u$ is expressed as follow.
\begin{equation}
\begin{split}
\overline{\nabla} u & = \big \{ \frac{\partial u_1}{\partial Y^1} -\frac{2Y^1u_1}{1-|Y|^2} + \frac{2Y^2u_2}{1-|Y|^2}   \big \} \dd Y^1\otimes \dd Y^1 \\
& + \big \{\frac{\partial u_2}{\partial Y^1} -\frac{2Y^2u_1}{1-|Y|^2} - \frac{2Y^1u_2}{1-|Y|^2}   \big \} \dd Y^1\otimes \dd Y^2 \\
& + \big \{ \frac{\partial u_1}{\partial Y^2} -\frac{2Y^2u_1}{1-|Y|^2} - \frac{2Y^1u_2}{1-|Y|^2}   \big \} \dd Y^2\otimes \dd Y^1 \\
& + \big \{ \frac{\partial u_2}{\partial Y^2} +\frac{2Y^1u_1}{1-|Y|^2} - \frac{2Y^2u_2}{1-|Y|^2}   \big \} \dd Y^2\otimes \dd Y^2.
\end{split}
\end{equation}
By symmetrizing all the terms in the above expression of $\overline{\nabla} u$, we yield the following expression for $\Def u = \frac{1}{2} \{\overline{\nabla} u + (\overline{\nabla} u)^T\}$.
\begin{equation}\label{DeformTENSOR}
\begin{split}
\Def u & = \big \{ \frac{\partial u_1}{\partial Y^1} -\frac{2Y^1u_1}{1-|Y|^2} + \frac{2Y^2u_2}{1-|Y|^2}   \big \} \dd Y^1\otimes \dd Y^1 \\
& + \big \{ \frac{1}{2} \big ( \frac{\partial u_1}{\partial Y^2} + \frac{\partial u_2}{\partial Y^1}  \big )
- \frac{2Y^2u_1}{1-|Y|^2} - \frac{2Y^1u_2}{1-|Y|^2} \big \} \big ( \dd Y^1 \otimes \dd Y^2 + \dd Y^2 \otimes \dd Y^1  \big ) \\
& + \big \{ \frac{\partial u_2}{\partial Y^2} +\frac{2Y^1u_1}{1-|Y|^2} - \frac{2Y^2u_2}{1-|Y|^2}   \big \} \dd Y^2\otimes \dd Y^2.
\end{split}
\end{equation}
 
\subsection{Integration by parts and all that}
We end this part of the appendix with some simple computations.
 \begin{lemma} Let $(M,g)$ be given. Then if $f$ is a $C^1$ function and $u, v$ are $C^1$  $1-$forms on $M$, then the following hold pointwise
\begin{align}
-\dd^{\ast} (fu)&=-f\dd^{\ast} u +g(\dd f, u),\label{ip4}\\
2 g( \nablab_uv, v)&=\abs{v}^2 \dd^\ast  u-\dd^\ast  (\abs{v}^2u),\label{ip0}\\
\abs{\grad \abs{u}^2}&\leq 2\abs{\nablab u}\abs{u},\label{ip2}\\
g(\nablab_uv,w)&\leq \abs{u}\abs{\nablab v}\abs{w}. \label{csg}
\end{align}
\end{lemma}
\begin{proof}
\eqref{ip4} follows from \cite[Ex 3-3, p.43 ]{Lee} where it is stated for vector fields.  For \eqref{ip0} we first observe
\[
\dd g(v,w)(u)=g(\nablab_u v, w)+g(v, \nablab_u w)\label{ip1}.
\]
Hence if $w=v$, and using $\abs{v}^2=g(v,v)$ we have
\[
2 g( \nablab_uv, v)=  \dd \abs{v}^2(u)=g(\nabla \abs{v}^2,u)=g(\dd \abs{v}^2,u),
\]
where the second equality follows from the definition of the gradient, and the last one by the definition of the induced metric on the 1-forms.  Then \eqref{ip0} follows from \eqref{ip4}. (Recall, for the simplicity of notation $u$ is used to denote both the vector field and the corresponding $1$-form.)

\eqref{ip2} follows from the compatibility of the metric with the connection and Cauchy-Schwarz since
\[
\abs{\grad \abs{u}^2}=\abs{\grad g (u,u)}=2\abs{\bar g(\nablab u, u)}\leq 2\abs{\nablab u}\abs{u}.
\]
Finally, for \eqref{csg} by Cauchy-Schwarz it is enough to show
\[
g(\nablab_uv,\nablab_uv)\leq \abs{\nablab v}^2\abs{u}^2.
\]
But since this is a pointwise estimate, this follows from computing in normal coordinates just like it would in the Euclidean case, because then $\nablab_u v=\sum_{jk}u^j \partial_j v^k\frac{\partial}{\partial x^k}$, and $g_{ij}=\delta_{ij}$.
\end{proof}

\section{Finite dissipation via Complex Analysis}\label{awesome}

First we have the following lemma based on some elementary complex analysis.

\begin{lemma}\label{OURWAY}
Let $f : D_O(1)\rightarrow \mathbb{C}$ to be a holomorphic function on the open unit disc $D_O(1) = \{z =y_1 + i y_2 \in \mathbb{C} : |z| < 1 \}$ which satisfies
\begin{equation}\label{holomorphicfiniteenergy}
\int_{D_O(1)} |f(z)|^2 \dd y_1 \dd y_2 < \infty .
\end{equation}
Then it follows that $f'(z)$ satisfies the following property
\begin{equation}\label{OURresult}
\int_{D_O(1)} (1-|z|)^2 |f'(z)|^2 \dd y_1 \dd y_2 \leq \frac 12 \int_{D_O(1)} |f(z)|^2 \dd y_1 \dd y_2.
\end{equation}
\end{lemma}

\begin{proof}
Let $f$ be a holomorphic function $f$ on $D_O(1)$ which satisfies condition \eqref{holomorphicfiniteenergy}.  Then $f$ can be expressed in a form of a power series $$f(z) = \sum_{k=0}^{\infty} a_k z^k,$$ whose radius of convergence is of course $1$. Then for any $r \in (0,1)$ and $\theta \in [0,2 \pi )$ we have
\begin{equation*}
|f(r \cdot e^{i\theta})|^2 = f(r \cdot e^{i\theta}) \overline{f(r \cdot e^{i\theta})}
= \sum_{k=0}^{\infty} \sum_{l=0}^{\infty} a_k   \overline{a_l} r^{k+l} e^{i(k-l)\theta} ,
\end{equation*}
from which it follows directly that

\begin{equation}\label{vitalrelation}
\pi \cdot \sum_{k=0}^{\infty} \frac{|a_k|^2}{k+1} = \int_0^1 \int_0^{2\pi} |f(r\cdot e^{i\theta })|^2 r \dd r \dd \theta = \int_{D_O(1)} |f(z)|^2 \dd y_1\dd y_2 < \infty .
\end{equation}
In the same way, we easily get the following relation for $|f'(z)|^2$.
\begin{equation*}
|f'(r \cdot e^{i\theta})|^2 = f'(r \cdot e^{i\theta})  \overline{f'(r \cdot e^{i\theta})}
= \sum_{k=0}^{\infty} \sum_{l=0}^{\infty} (k+1)(l+1) a_{k+1}  \overline{a_{l+1}}  r^{k+l} e^{i(k-l)\theta} ,
\end{equation*}
from which one immediately gets
\begin{equation*}
\int_0^{2\pi} |f'(r\cdot e^{i\theta})|^2 \dd \theta = 2\pi \sum_{k=0}^{\infty} |a_{k+1}|^2 (k+1)^2 r^{2k} .
\end{equation*}
By using polar coordinates we deduce
\begin{equation}\label{veryExcitinginequality}
\begin{split}
\int_{D_O(1)} (1-|z|)^2 \cdot |f'(z)|^2 \dd y_1 \dd y_2 & = \int_0^1 r (1-r)^2  \int_0^{2\pi} |f'(r\cdot e^{i\theta})|^2 \dd \theta \dd r \\
& =  2\pi \sum_{k=0}^{\infty} \bigg ( |a_{k+1}|^2 (k+1)^2 \int_0^1 (1-r)^2 r^{2k+1} \dd r \bigg ) .
\end{split}
\end{equation}
By a simple computation, we have
\begin{equation}\label{VeryCareful}
\int_0^1 (1-r)^2r^{2k+1} \dd r = \frac{2}{(2k+2)(2k+3)(2k+4)} .
\end{equation}
So, it follows directly from \eqref{veryExcitinginequality}, and \eqref{VeryCareful} that we must have
\begin{align*}
\int_{D_O(1)} (1-|z|)^2 \cdot |f'(z)|^2 \dd y_1 \dd y_2&=\pi\sum_{k=0}^\infty\frac{\abs{a_{k+1}}^2}{(2k+3)}\frac{(k+1)}{(k+2)}\\
&=\frac \pi 2\sum_{k=1}^\infty\frac{\abs{a_{k}}^2}{(2k+1)}\frac{2k}{(k+1)}\\
&\leq \frac 12 \int_{D_O(1)}\abs{f(z)}^2 \dd y^1\dd y^2,
\end{align*}
as needed.
 \end{proof}

\begin{remark}
Here, let us mention that the crucial identity \eqref{vitalrelation} is classical and can easily be found in standard textbooks in Complex Analysis. Inequality \eqref{OURresult} seems to be less well-known. However, in the light of the long history of the theory of holomorphic functions and the elementary nature of \eqref{OURresult} we suspect that an estimate of that kind  could be presented somewhere in the vast literature, though it does not seem to be as easily located as \eqref{vitalrelation}.
\end{remark}

What we really need is actually the following lemma, which is a straightforward byproduct of Lemma \ref{OURWAY}.
\begin{lemma}\label{UsefulLEMMAONTHEDISC}
Let $u : D_O(1) \rightarrow \mathbb{R}$  be a harmonic function on the Euclidean unit disc $D_O(1) = \{y = (y_1 , y_2) : |y| <1\}$ which satisfies the following condition
\begin{equation}\label{EuclidFiniteEnergy}
\int_{D_O(1)} \big | \nabla^{\mathbb{R}^2} u \big |^2 \dd y_1 \dd y_2  = \int_{D_O(1)} |\partial_{y_1} u|^2 + |\partial_{y_2} u|^2 < \infty .
\end{equation}
Then, it follows that $|\nabla^{\mathbb{R}^2} \nabla^{\mathbb{R}^2} u |^2 = \sum_{1\leq j,k \leq 2} |\partial_{y_j}\partial_{y_k}u|^2 $ satisfies the following property
\begin{equation}\label{FinalGoryPride}
\int_{D_O(1)} (1-|y|)^2 |\nabla^{\mathbb{R}^2} \nabla^{\mathbb{R}^2} u |^2 \dd y_1 \dd y_2 \leq   \int_{D_O(1)} \big | \nabla^{\mathbb{R}^2} u \big |^2 \dd y_1 \dd y_2 .
\end{equation}
\end{lemma}
\begin{proof}
Here, the point $y =(y_1 , y_2) \in D_O(1)$ is identified with the complex number $z= y_1 +i y_2$. Let $u : D_O(1)\rightarrow \mathbb{R}$  be a harmonic function which satisfies \eqref{EuclidFiniteEnergy}, and let $v : D_O(1) \rightarrow \mathbb{R}$ be the harmonic conjugate of $u$ on $D_O(1)$, so that $$f(z) = u(z) + i v(z)$$ is a holomorphic function on $D_O(1)$. Since $u$ and $v$ satisfy the Cauchy-Riemann equations we have
\begin{equation}
f' = \partial_{y_1}u + i  \partial_{y_1}v = \partial_{y_1}u - i \partial_{y_2} u ,
\end{equation}
which of course gives
\begin{equation}
|f'(z)|^2 = |\partial_{y_1}u |^2 + |\partial_{y_2} u|^2 = \big | \nabla^{\mathbb{R}^2}u \big |^2.
\end{equation}
As a result, condition \eqref{EuclidFiniteEnergy} is equivalent to
\begin{equation}
\int_{D_O(1)} |f'(z)|^2 \dd y_1 \dd y_2 =  \int_{D_O(1)} \big | \nabla^{\mathbb{R}^2} u \big |^2 \dd y_1 \dd y_2  < \infty .
\end{equation}
Hence, we can apply Lemma \ref{OURWAY} to $f'$ to deduce that we must have the following estimate \begin{equation}\label{relieveestimate}
\int_{D_O(1)} (1-|z|)^2 |f''(z)|^2 \dd y_1 \dd y_2 \leq \frac 12 \int_{D_O(1)} \big | \nabla^{\mathbb{R}^2} u \big |^2 \dd y_1 \dd y_2 .
\end{equation}
However, since $u$ is harmonic and $f''(z) = \partial_{y_1} \partial_{y_1} u - i \partial_{y_1} \partial_{y_2} u $, we obtain\begin{equation}
\begin{split}
|f''(z)|^2 & = \big | \partial_{y_1} \partial_{y_1} u \big |^2 +  | \partial_{y_1} \partial_{y_2} u \big |^2 \\
& = \frac{1}{2} \big \{ \big | \partial_{y_1} \partial_{y_1} u \big |^2 + 2 | \partial_{y_1} \partial_{y_2} u \big |^2  +  | \partial_{y_2} \partial_{y_2} u \big |^2   \big \} \\
& = \frac{1}{2} |\nabla^{\mathbb{R}^2} \nabla^{\mathbb{R}^2} u |^2.
\end{split}
\end{equation}
Thanks to this identity, estimate \eqref{relieveestimate} is equivalent to estimate \eqref{FinalGoryPride}, as desired.
\end{proof}


\section{Functional Analysis}\label{faa}

\begin{lemma}\label{useful_lemma} \cite[Lemma 1.1 p. 169]{Temam} Let $X$ be a Banach space and $X'$ its dual.  If $f, g \in L^1([a,b];X))$, then the following are equivalent
\begin{enumerate}
\item $u$ is a.e. equal to a primitive of $g$:
\[
u(t)=\xi +\int^t_0 g(s)ds \quad \xi \in X,\  \mbox{and for a.e.}\ t\in[a,b].
\]
\item $g$ is the weak time derivative of $u$:
\[
\int^b_a u(t)\phi'(t)\dd t=-\int^b_a g(t)\phi(t) \dd t\quad \mbox{for every  test function} \ \phi \in \mathcal D((a,b)).
\]
\item
\[
\frac \dd {\dd t}\ip{u,\eta}_{X\otimes  X'}=\ip{g,\eta}_{X\otimes  X'}\quad \mbox{for every }\ \eta \in X'.
\]
\end{enumerate}
\end{lemma}

\begin{lemma}\cite[Lemma 1.2, p. 176]{Temam}\label{anotherlemma}
Let $V, H$ be a Hilbert space satisfying
\[
V\subset H \subset V',
\]
with $V'$ the dual of $V$.  If $u\in L^2(0,T;V)$ and $\frac{\partial u}{\partial t} \in L^2(0,T;V')$, then $u$ is almost everywhere equal to a function continuous from $[0,T]$ into $H$ and
\[
\frac\dd {\dd t}\abs{u}^2=2\ip{\frac{\partial u}{\partial t},u}
\]
holds in a sense of distributions on $(0,T)$.
\end{lemma}

\begin{thm}[Compactness]\cite[Theorem 2.2 p.186]{Temam}\label{compact2}
Let $X_{0}, X, X_{1}$ be Hilbert spaces such that
\[
X_{0}\subset X\subset X_{1}
\]
where the injections are continuous and so that the injection of $X_{0}$ into $X$ is compact.  Let $\gamma>0$ and $\mathcal H^{\gamma}$ denote the space
\[
\mathcal H^{\gamma}:=\mathcal H^{\gamma}(\R: X_{0}, X_{1})=\{v\in L^{2}(\R; X_{0}), D^{\gamma}_{t}v\in L^{2}(\R; X_{1})\},
\]
where $D^{\gamma}_{t}v$ denotes a fractional time derivative of $v$, and with a norm
\[
\norm{v}_{\mathcal H^{\gamma}}=\norm{v}_{L^2(\R;X)}+ \norm{\abs{\tau}^\gamma \hat v}_{L^2(\R;X_1)}.
\]
Then for a bounded set $K\subset \R$, the space $H^{\gamma}_{K}$ given by
\[
\mathcal H^{\gamma}_{K}=\{ u \in \mathcal H^{\gamma}: \supp u \subset K\}
\]
is compactly imbedded into $L^{2}(\R,X)$.
\end{thm}

We emphasize that the following lemma should be a part of the standard working knowledge in functional analysis. However, we present a simple proof for it since it  fully justifies why we are able to select a basis for $\textbf{V}$ \emph{within} $\Lambda_{c,\sigma}^1(\mathbb{H}^2(-a^2))$.

\begin{lemma}\label{peaceofmindlemma}
Consider $V$ to be a separable Banach space equipped with the norm $\|\cdot \|_V$. Let $S$ be some dense subset of $V$. Then, there exists a sequence of elements $\{y_k\}_{k=1}^{\infty} \subset S$ which is a basis for $V$.
\end{lemma}

\begin{proof}
Since $V$ is separable, we can find a sequence $\{v_m\}_{m=1}^{\infty} \subset V$ such that $\{v_m\}_{m=1}^{\infty} $ is dense in $V$. Then, for each pair of integers $m,j \in \mathbb{Z}^+$, the density of the set $S$ in $V$ ensures that there is some element $e_{mj} \in S$ such that $\|e_{mj} - v_m\|_V < \frac{1}{2^j}$.

Next, we check that $\{e_{mj} : m,j \in \mathbb{Z}^+\} \subset S$ is also dense in $V$. For any $v \in V$, and any $\epsilon >0$, the density of $\{v_m\}_{m=1}^{\infty}$ ensures the existence of some integer $m_{\epsilon}$ such that
$\|v_{m_{\epsilon}} - v\|_V < \frac{\epsilon}{2}$. Then let  $j_{\epsilon}$ be some sufficiently large integer so that $\frac{1}{2^{j_{\epsilon}}} < \frac{\epsilon}{2}$. Hence 
\begin{equation*}
\|e_{mj_{\epsilon}} - v\|_V \leq \|e_{m_{\epsilon}j_{\epsilon}} - v_{m_{\epsilon}}\|_V + \|v_{m_{\epsilon}} - v\|_V < \epsilon 
\end{equation*}
as needed.  

For convenience, we need to enumerate  $\{e_{mj}\}$.  Take any bijective map $\Psi : \mathbb{Z}^+ \rightarrow
\mathbb{Z}^+\otimes  \mathbb{Z}^+$ so that
\begin{equation*}
\{e_{mj} : m,j \in \mathbb{Z}^+\} = \{e_{\Psi (k)}\}_{k=1}^{\infty} .
\end{equation*}
Consider
\begin{equation*}
E_{m} = \Span\{e_{\Psi (1)}, e_{\Psi (2)} , ..... e_{\Psi (m)}\} ,
\end{equation*}
then, the density of $\{e_{\Psi (k)}\}_{k=1}^{\infty}$ in $V$ of course implies that the closure of $\bigcup_{m} E_m$ coincides with $V$. That is we have
\begin{equation*}
\overline{\bigcup_{m} E_m} = V 
\end{equation*}
so the sequence $\{e_{\Psi (k)}\}_{k=1}^{\infty}$ is total in $V$.

In order to extract a basis of $V$ from $\{e_{\Psi (k)}\}_{k=1}^{\infty}$, we can carry out a simple procedure to eliminate those linearly dependent elements from $\{e_{\Psi (k)}\}_{k=1}^{\infty}$ in the following way: for each  $m \geq 2$,  eliminate the element $e_{\Psi(m)}$ if it happens that $E_m = E_{m-1}$. Otherwise,  we   keep  $e_{\Psi(m)}$. After carrying out this procedure indefinitely, we arrive at a strictly increasing sequence of positive integers $\{m_k\}_{k=1}^{\infty}$ such that the subsequence $\{e_{\Psi(m_k)}\}_{k=1}^{\infty}$ is linearly independent and also total in $V$. Hence, the sequence $\{e_{\Psi(m_k)}\}_{k=1}^{\infty}$ of elements in $S$ can serve as a basis of $V$, as desired.
\end{proof}

\bibliography{ChiHinMagda}

\def\cprime{$'$}
\begin{thebibliography}{10}

\bibitem{Anderson}
Michael~T. Anderson.
\newblock The {D}irichlet problem at infinity for manifolds of negative
  curvature.
\newblock {\em J. Differential Geom.}, 18(4):701--721 (1984), 1983.

\bibitem{AndersonSchoen}
Michael~T. Anderson and Richard Schoen.
\newblock Positive harmonic functions on complete manifolds of negative
  curvature.
\newblock {\em Ann. of Math. (2)}, 121(3):429--461, 1985.

\bibitem{AvezBamberger}
A.~Avez and Y.~Bamberger.
\newblock Mouvements sph\'eriques des fluides visqueux incompressibles.
\newblock {\em J. M\'ecanique}, 17(1):107--145, 1978.

\bibitem{CC10}
Chi~Hin Chan and Magdalena Czubak.
\newblock Non-uniqueness of the {L}eray-{H}opf solutions in the hyperbolic
  setting.
\newblock {\em Dyn. Partial Differ. Equ.}, 10(1):43--77, 2013.

\bibitem{DeRhamEng}
Georges de~Rham.
\newblock {\em Differentiable manifolds}, volume 266 of {\em Grundlehren der
  Mathematischen Wissenschaften [Fundamental Principles of Mathematical
  Sciences]}.
\newblock Springer-Verlag, Berlin, 1984.
\newblock Forms, currents, harmonic forms, Translated from the French by F. R.
  Smith, With an introduction by S. S. Chern.

\bibitem{DindosMitrea}
Martin Dindo{\v{s}} and Marius Mitrea.
\newblock The stationary {N}avier-{S}tokes system in nonsmooth manifolds: the
  {P}oisson problem in {L}ipschitz and {$C^1$} domains.
\newblock {\em Arch. Ration. Mech. Anal.}, 174(1):1--47, 2004.

\bibitem{Dodziuk}
Jozef Dodziuk.
\newblock {$L^{2}$} harmonic forms on rotationally symmetric {R}iemannian
  manifolds.
\newblock {\em Proc. Amer. Math. Soc.}, 77(3):395--400, 1979.

\bibitem{EbinMarsden}
David~G. Ebin and Jerrold Marsden.
\newblock Groups of diffeomorphisms and the motion of an incompressible fluid.
\newblock {\em Ann. of Math. (2)}, 92:102--163, 1970.

\bibitem{Folland}
Gerald~B. Folland.
\newblock {\em Real analysis}.
\newblock Pure and Applied Mathematics (New York). John Wiley \& Sons Inc., New
  York, second edition, 1999.
\newblock Modern techniques and their applications, A Wiley-Interscience
  Publication.

\bibitem{GilbargTrudinger}
David Gilbarg and Neil~S. Trudinger.
\newblock {\em Elliptic partial differential equations of second order}.
\newblock Classics in Mathematics. Springer-Verlag, Berlin, 2001.
\newblock Reprint of the 1998 edition.

\bibitem{Hebey}
Emmanuel Hebey.
\newblock {\em Nonlinear analysis on manifolds: {S}obolev spaces and
  inequalities}, volume~5 of {\em Courant Lecture Notes in Mathematics}.
\newblock New York University Courant Institute of Mathematical Sciences, New
  York, 1999.

\bibitem{Heywood}
John~G. Heywood.
\newblock On uniqueness questions in the theory of viscous flow.
\newblock {\em Acta Math.}, 136(1-2):61--102, 1976.

\bibitem{Hopf}
Eberhard Hopf.
\newblock \"{U}ber die {A}nfangswertaufgabe f\"ur die hydrodynamischen
  {G}rundgleichungen.
\newblock {\em Math. Nachr.}, 4:213--231, 1951.

\bibitem{Illin}
A.~A. Il{\cprime}in.
\newblock Navier-{S}tokes and {E}uler equations on two-dimensional closed
  manifolds.
\newblock {\em Mat. Sb.}, 181(4):521--539, 1990.

\bibitem{IFilatov}
A.~A. Il{\cprime}in and A.~N. Filatov.
\newblock Unique solvability of the {N}avier-{S}tokes equations on a
  two-dimensional sphere.
\newblock {\em Dokl. Akad. Nauk SSSR}, 301(1):18--22, 1988.

\bibitem{Jost}
J{\"u}rgen Jost.
\newblock {\em Riemannian geometry and geometric analysis}.
\newblock Universitext. Springer-Verlag, Berlin, fifth edition, 2008.

\bibitem{KhesinMisiolek}
Boris Khesin and Gerard Misio{\l}ek.
\newblock The {E}uler and {N}avier-{S}tokes equations on the hyperbolic plane.
\newblock {\em Proc. Nat. Acad. Sci.}, 2012.

\bibitem{Kodaira}
Kunihiko Kodaira.
\newblock Harmonic fields in {R}iemannian manifolds (generalized potential
  theory).
\newblock {\em Ann. of Math. (2)}, 50:587--665, 1949.

\bibitem{Lbook}
O.~A. Ladyzhenskaya.
\newblock {\em The mathematical theory of viscous incompressible flow}.
\newblock Second English edition, revised and enlarged. Translated from the
  Russian by Richard A. Silverman and John Chu. Mathematics and its
  Applications, Vol. 2. Gordon and Breach Science Publishers, New York, 1969.

\bibitem{Lee}
John~M. Lee.
\newblock {\em Riemannian manifolds}, volume 176 of {\em Graduate Texts in
  Mathematics}.
\newblock Springer-Verlag, New York, 1997.
\newblock An introduction to curvature.

\bibitem{Leray}
Jean Leray.
\newblock Sur le mouvement d'un liquide visqueux emplissant l'espace.
\newblock {\em Acta Math.}, 63(1):193--248, 1934.

\bibitem{Mazzucato2003}
Anna~L. Mazzucato.
\newblock Besov-{M}orrey spaces: function space theory and applications to
  non-linear {PDE}.
\newblock {\em Trans. Amer. Math. Soc.}, 355(4):1297--1364, 2003.

\bibitem{MitreaTaylor}
Marius Mitrea and Michael Taylor.
\newblock Navier-{S}tokes equations on {L}ipschitz domains in {R}iemannian
  manifolds.
\newblock {\em Math. Ann.}, 321(4):955--987, 2001.

\bibitem{Priebe}
Volker Priebe.
\newblock Solvability of the {N}avier-{S}tokes equations on manifolds with
  boundary.
\newblock {\em Manuscripta Math.}, 83(2):145--159, 1994.

\bibitem{Sullivan}
Dennis Sullivan.
\newblock The {D}irichlet problem at infinity for a negatively curved manifold.
\newblock {\em J. Differential Geom.}, 18(4):723--732 (1984), 1983.

\bibitem{Taylor}
Michael Taylor.
\newblock The dirichlet problem on the hyperbolic ball.
\newblock {\em Preprint}, http://www.unc.edu/math/Faculty/met/dphb.pdf.

\bibitem{Taylor3}
Michael~E. Taylor.
\newblock {\em Partial differential equations {III}. {N}onlinear equations},
  volume 117 of {\em Applied Mathematical Sciences}.
\newblock Springer, New York, second edition, 2011.

\bibitem{Temam}
Roger Temam.
\newblock {\em Navier-{S}tokes equations}, volume~2 of {\em Studies in
  Mathematics and its Applications}.
\newblock North-Holland Publishing Co., Amsterdam, third edition, 1984.
\newblock Theory and numerical analysis, With an appendix by F. Thomasset.

\bibitem{TemamWang}
Roger Temam and Shou~Hong Wang.
\newblock Inertial forms of {N}avier-{S}tokes equations on the sphere.
\newblock {\em J. Funct. Anal.}, 117(1):215--242, 1993.

\bibitem{Yau75}
Shing~Tung Yau.
\newblock Harmonic functions on complete {R}iemannian manifolds.
\newblock {\em Comm. Pure Appl. Math.}, 28:201--228, 1975.

\bibitem{QSZhang}
Qi~S. Zhang.
\newblock The ill-posed {N}avier-{S}tokes equation on connected sums of {$\bold
  R^3$}.
\newblock {\em Complex Var. Elliptic Equ.}, 51(8-11):1059--1063, 2006.

\end{thebibliography}
\bibliographystyle{plain}

\end{document}